\documentclass{article}

\usepackage{mycommands}

\begin{document}

\title{An Iterable Surgery Formula on Involutive Knot Lattice Homotopy}
\author{Seppo Niemi-Colvin}
\maketitle

\begin{abstract}
In ``Knots in lattice homology", Ozsv\'ath, Stipsicz, and Szab\'o showed that knot lattice homology satisfies a surgery formula similar to the one relating knot Floer homology and Heegaard Floer homology, and in previous work, I showed that knot lattice homology is the persistent homology of a doubly filtered space. Here I provide an iterable version of the surgery formula that, provided the initial knot lattice space with flip map, produces a space isomorphic as a doubly-filtered space to the corresponding knot lattice space for the dual knot with the corresponding flip map. If we include the involutive data for the original knot and ambient three-manifold, we can also produce the corresponding involutive data on the new knot lattice space without assuming that the original three-manifold is an \(L\)-space. I construct \(\infty\)-categories where these operations are functorial. Finally, I use the surgery formula to compute some examples of knot lattice spaces including for the regular fiber of \(\Sigma(2,3,7)\) and for a knot in a three-manifold that is not given by an almost rational graph.
\end{abstract}

\tableofcontents

\section{Introduction}

The framework of lattice homology and homotopy  has provided tools for computing Heegaard Floer homology and thus has had applications to 3- and 4-manifold topology.
For example, the algorithm of Can and Karakurt \cite{brieskornAlgorithm} has been helpful at providing a quick way to compute a variety of examples of Heegaard Floer homology for Brieskorn spheres.
Additionally, understanding the involutive structure on lattice homology and its relation to Heegaard Floer homology allowed for the calcuation of invariants that ultimately lead to the discovery of a \(\Z^{\infty}\) summand of the homology cobordism group \cite{cobordSummand,involCheck}.
Having a class of example manifolds where Heegaard Floer homology is more reasonable to compute has helped provide examples that realize said invariants.

However, these examples have been restricted to almost-rational graphs and include few examples of knot Floer homology calculations that were not already available despite the existence of knot lattice homology and homotopy.
A part of the issue is that while lattice homology and knot lattice homology are very explicit, they are still large enough to grow unwieldy, and knot lattice homology has only been explored in limited examples.
This paper addresses both issues at once, using knot lattice spaces with their involutive structures to inductively work out to more and more complicated three-manifolds and knots, while providing the opportunity for simplications.
This leads to the following applications.

\begin{theorem}\label{thm:Sigma237regFib}
The regular fiber of \(\Sigma(2,3,7)\) has Seifert genus and genus in a self homology cobordism both equal to 22.
\end{theorem}

\begin{theorem}\label{thm:nonARknotEx}
The knot \((Y,K)\) depicted in Figure \ref{fig:nonARknot} has Seifert genus and genus in \(Y\times I\) equal to 6. Its genus in a self homology cobordism is bounded below by 1. This bound increases to 3 if the involutive structure on lattice homology is confirmed to agree with that on Heegaard Floer homology.
\end{theorem}

The main tool we use is a new surgery formula for knot lattice homotopy, which fits into a broader research program of upgrading the surgery formula to track more information.
For example, upgrades to the surgery formula have allowed researchers to compute knot Floer complexes for knots in surgered three-manifolds, such as for the dual knot, i.e. the core of the solid tori that has been reglued in \cite{FilteredSurgeryFormula}, and this has been used in computing a \(\Z^{\infty}\) summand of the homology concordance group \cite{homolConcordSummand,homolConcordance}.
Others have computed a formula for cables of the dual knot \cite{surgeryFormulaCable}, which has been used to provide a family of knots in the boundaries of integral homology balls whose members have arbitrarily large PL-genus, i.e. the minimal genus of a surface in any homology cobordism to any knot in \(S^3\) \cite{PLgenusLarge}.

\begin{figure}
\begin{center}
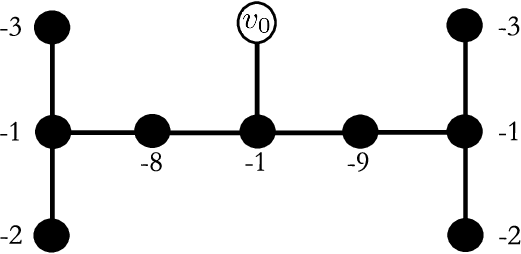
\end{center}
\caption{The knot of Theorem \ref{thm:nonARknotEx}.}\label{fig:nonARknot}
\end{figure}

Another variation on upgrading the surgery formula has focused on incorporating information from involutive Heegaard Floer homology and involutive knot Floer homology, but due to limitations on naturality in the Heegaard Floer complex this has only been done in the case where the ambient three-manifold for the starting knot is an \(L\)-space \cite{InvolMappingCone,InvolDualKnot}.
Zemke has also done a significant amount of work recently analyzing the surgery formula from the perspective of \(A_{\infty}\)-modules, which has made extrapolating new results easier, such as tracking what happens to the flip map on the dual knot \cite{ZemkeWhat,generalSurg}.

Our work also builds off of an already known version of the surgery formula that exists in the context of lattice homology and knot lattice homology, which takes as input the knot lattice complex of the original knot along with an additional map called the flip map and outputs the lattice complex of the surgered three-manifold \cite{knotsAndLattice}.
In fact, the knot lattice complex was defined to be precisely what it needs to be for this surgery formula to hold, and analysis from the perspective of the link surgery formula was ultimately what demonstrated that lattice homology and knot lattice homology agree with Heegaard Floer homology and knot Floer homology respectively \cite{linkLattice,ZemkeLattice}.

In addition to the potential of incorporating the newer surgery formula results into the lattice homology framework, the knot lattice complex has some additional benefits not seen in the knot Floer complexes.
In particular, the contenders for the analogs of the involutive maps are not defined using naturality but directly on the complex and thus can sidestep the naturality issues currently constraining involutive surgery formulas (though there a exist higher order naturality results for knot lattice homology as well).
Further, there exists a refinement of the knot lattice complex to a doubly filtered cube complex, which is itself an invariant for the knot up to doubly-filtered homotopy.
In addition to potential new invariants such as the cohomology ring and Steenrod squares, this also endows the knot lattice chain complex with an extra grading.

While an example is not currently known where the homotopy type detects more information than the chain complex, tools such as the surgery formula could be used to simplify calculations while looking for such an example.
The surgery formula and connect sums can allow us to build up more complicated knots from simpler ones, applying simplifications at each step to keep the overall complexity under control.

The surgery formula uses a fundamentally \(\infty\)-categorical construction, so before diving into the surgery formula itself, we will need to address the model of \(\infty\)-categories we will be using.
While Zemke uses \(A_{\infty}\)-modules to model this behavior, due to the author's background and preference, we will be using quasi-categories instead.
In this framework we will need to appropriately model the information we would like to include for our knot such as the flip map or the involutive information in a quasi-category and show that our particular surgery formula provides an \(\infty\)-functor between these quasi-categories.
The fact that the involutive maps for lattice homotopy and knot lattice homotopy shuffle \Spinc structures and, in the case of the knot lattice involution, are skew-filtered means that additional  technical care needs to be taken in these results. 

Since quasi-categories are less common than \(A_{\infty}\)-modules in the Heegaard Floer homology and knot Floer homology community, we spend Section \ref{sec:quasicat}, establishing the relevant background for quasi-categories, including the model category structures we will use in our proofs later.
In Section \ref{sec:hocolim}, we specifically discuss background on model categories of functors and how this allows us to understand homotopy limits and colimits.
Section \ref{sec:topBack} then outlines some background from low-dimensional topology that will be needed in order to articulate the surgery formula, especially because not assuming we are in an integral homology sphere let alone \(S^3\) eliminates some notational short cuts that are often used in these constructions.

Having established the necessary background, Section \ref{sec:quasicats4CFK} outlines the quasi-categories used to encode our needed information.
We start in Subsection \ref{subsec:FilteredsSet} with modeling doubly filtered simplicial sets and the various functors on doubly filtered simplicial sets that are used in the surgery formula and for connect sums.
This provides the base on which Subsection \ref{subsec:quasicatFlipMap} can provide the quasi-categories modeling the flip map and Subsection \ref{subsec:quasicatInvol} can provide the quasi-categories modeling the involutive information.
Discussion of equivalent choices of quasi-categories is provided here as well.

Section \ref{sec:InfFunc4Surgery} then provides the constructions for the surgery formula functors themselves.
Multiple versions of the functor exist, each layering on more information to the previous functor.
We start by producing the lattice space of the surgered 3-manifold from the knot lattice space of the knot in Section \ref{subsec:Xb}, then incorporate information about the dual knot into this picture in \ref{subsec:XK}, before adding the involutive information in Section \ref{subsec:XKI}. 

Section \ref{sec:Lattice} then starts with Subsections \ref{subsec:LatticeDef} and  \ref{subsec:knotLatticeDef} by providing a review of the construction of the lattice and knot lattice homotopy types, articulating them in a way that naturally fits into the quasi-categories constructed in Section \ref{sec:quasicats4CFK}.
In Subsection \ref{sec:Verify}, we then verify that the surgery functors do indeed work as advertised.
Finally, we spend some time comparing this functor to the other recent work on surgery formulas for knot Floer homology in Section \ref{subsec:CompareForm}.
In particular, we verify that on the level of chain complexes our functors agree with the previous results, with a chunk of the work here showing to translate between the formulation of the surgery functor here and the multiple other notational paradigms used elsewhere that could obscure similarities.
This is especially important when comparing work on the involutive surgery functors because the exact relationship between the involutive maps on the knot lattice complex and knot Floer complex is not completely known.
In particular, the knot lattice complex comes with an extra grading due to its construction from a doubly-filtered simplicial set, and the knot lattice involutive map may only be the portion of the knot Floer involutive map that preserves this extra grading.

Finally, Sections \ref{sec:ARknots} and \ref{sec:IteratedEx} show some applications of the surgery formula, including those of Theorems \ref{thm:Sigma237regFib} and \ref{thm:nonARknotEx}. 
Section \ref{sec:ARknots} highlights how the surgery formula provides a constraints on the doubly filtered homotopy type for a class of knots that we are calling almost rational knots, which are knots in almost rational three-manifolds directly connected to the manifold's almost rational'ness. 
This class of knots includes the regular fibers of Brieskorn spheres and we use these techniques to provide an algorithm for computing the knot lattice space for these regular fibers, using the regular fiber of \(\Sigma(2,3,7)\) as an example.
Meanwhile, Section \ref{sec:IteratedEx} highlights the surgery formula's versatility by using the formula iteratively to compute the knot lattice space for a knot that is not in an almost rational 3-manifold.

\section*{Acknowledgements}
I would like to thank Adam Levine under whose mentorship I started this project.
I appreciate how Daniel Litt introduced me to the idea of homotopy pushouts that helped set me on understanding homotopy colimits more broadly, as well as the help from Kirsten Wickelgren in understanding homotopy colimits.
I thank the Renyi Institute for giving me time to focus on this work and offering a number of relevant workshops that provided me an opportunity to speak with and see talks by Andrew Manion, Ian Zemke, Irving Dai, and Matthew Stoffergen that were relevant to this paper.
My postdoc mentor Dylan Thurston also provided useful conversations around \(A_{\infty}\) categories and modules during that time that helped me see the similarities between my approach and that of Ian Zemke.

\section{Models for \(\infty\)-categories}\label{sec:quasicat}
There are multiple ways to represent \(\infty\)-categories, including topologically, simplicially enriched categories, and quasi-categories.
Each version is an ordinary category, whose objects represent potential \(\infty\)-categories, equipped with a model structure whose fibrant objects are the true \(\infty\)-categories.
As we do not assume the knot theorists who may be interested in these constructions are familiar with \(\infty\)-categories, we review the the basic definitions we will be leveraging here and highlight examples particularly relevant to our setting.
We refer interested parties to \cite{CatHomoTheory}, \cite{HTT}, and \cite{hoveyMC} for more details on the subject.

\subsection{Model Categories}
Before continuing further, in order to avoid size issues, fix a grothendieck universe.
This in particular will allow us to talk above the objects of a large category as being the vertices of a simplicial set, and to be able to talk about taking various small colimits without running into issue.
Furthermore, while not always finite, our main colimits of interest are countable, which is quite small compared to our categories in question, thus avoiding size issues.

Moreover, model categories will provide one avenue for constructing \(\infty\)-categories with nice constructions that can be related back to ideas in ordinary categories.
In particular, model categories will provide both examples of infinity categories and provide the structure needed to describe quasi-categories, a particular model for infinity categories.
A more in depth approach to model categories can be found in \cite{hoveyMC} and \cite{CatHomoTheory}, though we will review the basic definitions we will need below for those more versed in knot theory than homotopy theory.
\begin{definition}
A \emph{weak factorization system} \((\Lc,\Rc)\) for a category \(C\) are two classes of morphisms so that every morphism in \(C\) factors into a morphism in \(\Lc\) followed by a morphism in \(\Rc\), \(\Lc\) are the morphisms that have the left lifting property against all morphisms in \(\Rc\), and \(\Rc\) the morphisms that have the right lifting property with all the morphisms in \(\Lc\).
A model category \(M\) is a complete and cocomplete ordinary category equipped with three classes of morphisms the cofibrations \(\Cc\), fibrations \(\Fc\), and weak equivalences \(\Wc\), where \((\Cc\cap \Wc,\Fc)\) and \((\Cc,\Fc\cap \Wc)\) are weak factorization systems.
The class \(\Cc\cap \Wc\) is called the \emph{acyclic cofibrations} and the class \(\Fc\cap \Wc\) is called the \emph{acylic fibrations}.

An object \(X\) is \emph{cofibrant} if the morphism from the intial object \(\emptyset\) to \(X\) is a cofibration, and an object is \emph{fibrant} if the map to the terminal object \(\bullet\) is a fibration.
\emph{Cofibrant replacement} is the factorization of \(\emptyset\) to \(X\) to \(\emptyset \to QX\to X\) as a cofibration followed by an acyclic fibration, and \emph{fibrant replacement} is the factorization of \(X \to \bullet\) as \(X \to RX \to \bullet\) of an acyclic cofibration followed by a fibration.

A model category is \emph{cofibrantly generated} if \(\Fc\) and \(\Fc\cap \Wc\) can be defined as the morphisms with the right lifting property against a small set of respectively acyclic cofibrations and cofibrations, called the \emph{generating acyclic cofibrations} and \emph{generating fibrations}.
A model category is \emph{monoidal} if the underlying category is closed monoidal, the pushout-product of two cofibrations is a cofibraton, and the pushout-product of an acyclic cofibration with a cofibration is an acylyic cofibration.
A model category \(M\) is \emph{enriched} in a symmetric monoidal model category \(N\) if \(M\) is enriched tensored, and cotensored over \(N\) and pushout-products across the tensor satisfy the same requirements as in the monoidal case.

An adjunction \(F\colon M_1 \leftrightharpoons M_2\colon G\) between model categories is a \emph{Quillen adjunction} if the left adjoint  \(F\) preserves cofibrations and the right adjoint \(G\) preserves cofibrations,  and it is a \emph{Quillen equivalence} if for every cofibrant object \(x\) in \(M_1\) and fibrant object \(y\) in \(M_2\) the adjunction between \(M_2(F(x),y)\) and \(M_1(x,G(y))\) preserves weak equivalences.
The \emph{homotopy category of a model category} \(M\) is the ordinary category of \(M\) localized at the weak equivalences.
\end{definition}

Sometimes the factorizations are also required to be functorial, though choosing a different factorization produces weakly equivalent results, and it can be helpful to reserve the option for example for cofibrant/fibrant replacement to do nothing on objects that are already cofibrant/fibrant.
Sometimes only finite limits and colimits are required rather than all small limits and colimits.

Key to note is that a single ordinary category can have multiple model category structures on it as will come up with \(\sSet_{\Quil}\) and \(\sSet_{\Joyal}\).
Due to model categories being defined using weak factorization systems, the definiton of monoidal and enriched can equivalently be defined using pull-back homs between cofibrations and fibrations being fibrations, and if one of the pair is acyclic, then the result is also acyclic (see \cite{CatHomoTheory} Lemma 11.1.10).

\subsection{Kan-complex enriched categories}

The category of simplicial sets \(\sSet\) has a classical combinatorial model category structure \(\sSet_{\Quil}\)
To describe \(\sSet_{\Quil}\) first recall that a \emph{simplical horn} \(\Lambda^n_k\) is the portion of the boundary of the standard \(n\)-simplex \(\Delta^n\)  that contains the \(k\)th vertex but not the face opposing the \(k\)th vertex, and a \emph{horn inclusion} is the canonical map \(j^n_k\) including \(\Lambda^n_k\) into \(\Delta^n\).
The cofibrations for \(\sSet_{\Quil}\) are the monomorphisms, generated by the inclusions \(\bou\Delta^n\to \Delta^n\), the generating acyclic cofibrations are the horn inclusions, and the weak equivalences are the maps that after geometric realization are weak homotopy equivalences.
A fibrant object in \(\sSet_{\Quil}\) is called a \emph{Kan complex}.

There is a classical cofibrantly generated (but not combinatorial) model category structure on \(\Top\), called \(\Top_{\Quil}\), whose weak equivalences are the weak homotopy equivalences and whose cofibrant objects are the retracts of cell complexes, where the geometric realization and singuliar simplicial set adjunction
\[|-|\colon \sSet_{\Quil} \leftrightharpoons \Top_{\Quil} \colon \Sing(-),\]
form a a Quillen equivalence.

The category of simplicially enriched categories \(\sCat\) has a model structure \(\sCat_{\Quil}\) whose weak equivalences are functors that induce equivalences on the corresponding \(\Ho\sSet_{\Quil}\) enriched categories and whose fibrant objects are categories enriched over Kan complexes.
The Quillen equilvalence \(|-|\dashv \Sing(-)\) ensures that every topologically enriched category can be viewed as a category enriched over Kan-complexes.

\subsection{Homotopy Coherent Functors into Kan-enriched categories}

Given a small category \(I\) one can construct a simplicial category \(FU_{\bullet}(I)\) via repeated iterations of the adjunction
\[ F\colon \Cat \leftrightharpoons \rDirGraph\colon U\]
where \(F\) forgets composition but remembers identities, and \(U\) forms the free category on the non-identity arrows in a reflexive directed graph.
Objects of \(FU_{\bullet}(I)\) are the objects of \(I\), the one simplices strings of composable non-identity morphisms in \(I\), and higher simplices keep track of ways to add and drop parentheses.
For more details on the construction see chapter 11 of \cite{crossedMenag} and 16.3 of \cite{CatHomoTheory}.
Given a Kan-complex enriched category \(M\) the simplicially encriched functors \(\sCat(FU_{\bullet}(I),M)\) form \emph{the homotopy coherent diagrams over \(I\)}.
Instead of forcing diagrams to commute over \(I\) on the nose, the category \(FU_{\bullet}(I)\) provides a simplex for every commutativity relation and thus a homtopy in \(M\) relating those maps.

Unfortunately, while categories of filtered spaces most naturally fit into the framework of enriched categories, the idea of homotopy coherent maps themselves forming an \(\infty\)-category is most directly formulated in terms of a different model for \(\infty\)-categories: quasi-categories.
In particular, while one can construct a simplicial set representing the natural transformations between two homotopy coherent diagrams this is too strict of a notion and won't necesarily agree with \emph{the homotopy coherent natural transforamtions}, i.e. the homotopy coherent diagrams of shape \(I\times [1]\).
However, if you take the more general homotopy coherent natural transformatons, a simplicial set can be constructed, but you are only guaranteed a notion of associativity up to homotopy and not associativity on the nose \cite[p.~28]{coherenceInSimpEnriched}.

\begin{remark}\label{rem:AinftyMod1}
One more familiar with the \(A_{\infty}\)-modules used in for example the work of Zemke in \cite{ZemkeWhat,generalSurg}, might recognize the idea being conveyed here, though the details of how this is implemented are slightly different.
For example, the idempotents in the algebra \(\Ac\) provide the objects of the corresponding small category (or rather the identity morphisms on said objects).
The action of the idempotents on an \(A_{\infty}\) module along with the action of \(\mu_0\) provides a chain-chain complex for each object of \(\Ac\).
The action of \(\mu_1\) provides how a morphism of \(\Ac\) should act on these chain complexes with the \(A_{\infty}\) relation ensuring that it is a chain map.
The action of \(\mu_2\) provides a homotopy between the composition of chain maps given by algebra elements \(a_1\) and \(a_2\) separately, and the chain map given by \(a_1a_2\).
One can check directly that \(\mu_3\) provides a homotopy filling in the square specified on page 502 of \cite{crossedMenag}, so the sum of the two simplifies depicted.
\end{remark}

\subsection{Quasi-categories}

Quasi-categories are defined as a particular type of simplicial sets, which model the objects of a category as the 0-simplices and morphisms of a category as 1-simplices.
However, instead of specifying a way to compose directly, one uses the higher dimensional simplicies to specify that certain homotopy coherent diagrams should exist.

\begin{definition}
An \emph{inner horn} is a \(\Lambda^n_k\) where \(0<k<n\) and an \emph{inner horn inclusion} is the canonical map \(j^n_k\) including \(\Lambda^n_k\) into \(\Delta^n\) when \(0<k<n\).
\emph{quasi-category} \(C\) is a simplicial set that has the right lifting property with respect to the inner horn inclusions, i.e. any map from an inner horn into \(C\) can be filled to a map from the simplex \(\Delta^n\) to \(C\).
The \emph{objects of a quasi-category \(C\)} are the 0-simplicies of \(C\) and \emph{the morphims of a quasi-category \(C\)} are the one-simplicies.
\end{definition}

Like with simplicial categories, the infinity categories can be seen as coming from a particular model strcutre on \(\sSet\).
The following proposition is due to Joyal, though a proof can be found in \cite[2.2.5.1]{HTT}

\begin{proposition}(Joyal)
There is a combinatorial closed monoidal model structure \(\sSet_{\Joyal}\) on \(\sSet\) whose fibrant objects are the quasi-categories and whose cofibrations are the monomorphisms.
\end{proposition}

To pass between the models of \(\infty\)-categories we have an Quillen equivalence
\[\Cf\colon \sSet \leftrightharpoons \sCat\colon \Rf\]
between the category of simplicial sets and the category of simplicial enriched categories.
Details of this adjunction can be found in Section 16.3 of \cite{CatHomoTheory}.
If the simplicial enrichment of \(C\) is discrete then \(\Rf(C)= N(C)\),  where \(N\) is the nerve, i.e. the simplicial set defined by how the various \([n]\) map into \(C\) (see Remark 16.3.6 of \cite{CatHomoTheory}), and if \(I\) is a small category then \(\Cf(N(I))= FU_{\bullet}(I)\) (See Remark 16.4.7 of \cite{CatHomoTheory}).

Given a model category \(M\) enriched over \(\sSet_{\Quil}\), from now on a \emph{simplicially enriched model category}, then the full subcategory of cofibrant-fibrant objects \(M^{\circ}\) will be enriched over Kan-complexes.
Additonally, Proposition 5.2.4.6 of \cite{HTT} guarantees that a simplicially enriched Quillen adjunction between two simplicially enriched model categories \(M\) and \(N\), will lift to an appropriate notion of adjunction of quasi-categories \(\Rf(M^{\circ})\) and \(\Rf(N^{\circ})\).

\subsection{Quasi-categories and homotopy coherent diagrams}

It is an already established fact that if \(X\) is a simplicial set and \(C\) is a quasi-category then the function complex \([X,C]\) is a quasi-category (See Proposition 1.2.7.3 of \cite{HTT} and Corollary 15.2.3 in \cite{CatHomoTheory}), though we will review the broad strokes of the argument here, because in Section \ref{subsec:quasicatInvol} we will need to mimic those techniques in order to also model some of the quirks of the involutive maps on lattice homotopy and knot lattice homotopy.
For simplicially enriched categories \(C_1\) and \(C_2\), note that 
\[[\Rf(C_1),\Rf(C_2)]_0=\sSet(\Rf(C_1),\Rf(C_2))=\sCat(\Cf\Rf(C_1),C_2),\]
and by construction \([X,C]_n := \sSet(X\times \Delta^n,C)\).
Thus the objects of this quasi-category are the homotopy coherent functors of shape \(X\) and the morphisms between them the homotopy coherent natural transformations and so on, so naturally \([X,C]\) represents homotopy coherent functors.
By \(\sSet_{\Joyal}\) being closed monoidal and \([X,C]\) representing the pullback hom of \(\emptyset \to X\) and \(C\to \bullet\), we have that if \(C\) is a quasicategory then \([X,C]\) will also be a quasicategory.

\begin{remark}
One should note from our discussion in Remark \ref{rem:AinftyMod1} that \(m_3\) provided a homotopy not for the individual simplicies in \(FU_{\bullet}(\Ac)\) but for each chain of composable morphisms \(a_1,a_2,a_3\) in \(\Ac\).
As such, viewing \(\Ac\) as having the discrete enrichment, we have that the maps \(m_i\) are most directly seen as elements of \([N(\Ac),\Rf(\Kom)]\) where \(\Kom\) is the simplicially enriched category of chain complexes.
\end{remark}

\subsection{The \(\infty\)-category of \(\infty\)-categories} \label{subsec:inftyOfInfty}

Because \(\sSet_{\Joyal}\) is closed monoidal it is enriched in itself instead of \(\sSet_{\Quil}\) and thus is not a simplicially enriched model category.
However, by restricting ourselves to natural transformations made out of equivalences, we can acheive this goal.
One way to acheive this would be for quasicastegories \(C_1\) and \(C_2\) to take the core of \([C_1,C_2]\), which is the largest simplicial subset whose morphisms are all equivalences.
The core construction is functorial and would produce a Kan-complex enriched category and thus an \(\infty\)-category.
There is a way to do this that produces a simplicially enriched model category whose objects represent the quasi-categories with apropriate morphism spaces.
This will in turn allow us to construct homotopy limits of quasi-categories and identify when they are weakly equivalent to other constructions, such as using a Reedy category (see Section \ref{sec:hocolim} below).

\begin{definition}
The category of marked simplical sets \(\sSet^+\) has objects simplicial sets \(X\) with a subset of the 1 simplicies \(X_e\), called the \emph{marked edges} containing all degenerate edges.
A morphism of marked simplicial sets is one that sends marked edges to marked edges.
Given a simplical set \(X\), the minimal marking \(X^{\flat}\) only has degenerate edges marked, and the maximal marking \(X^{\sharp}\) has all edges marked.
\end{definition}

First \(\sSet^+\) is cartesian closed with internal hom \([X,Y]\) given by the marked simplical set whose underlying simplicial set has form \(\sSet^+(X\times (\Delta^n)^{\flat},Y)\) with \((\Delta^n)^{\flat}\) only having the degenerate edges marked, and whose marked edges are \(\sSet^+(X\times (\Delta^1)^{\sharp},Y)\).
There are multiple ways to enrich, tensor and cotensor \(\sSet^+\) over \(\sSet\) and the most most useful for us uses the adjunction
\[(-)^{\sharp}\sSet\colon \leftrightharpoons (-)^{\sharp}\sSet^+,\]
whose right adjoint takes the maximal subsimplical set with edges the marked edges and whose left ajdoint induces the maximal marking with all edges marked.
In particular, for \(X\in \sSet\) and \(Y,Z\in \sSet^+\), the tensoring of \(X\) with \(Y\) is \(X^{\sharp}\times Y\) and the cotensoring is \([X^{\sharp},Y]\), and external hom is given by  \( [Y,Z]^{\sharp}\).
The following theorem is a specialization of results in chapter 3 of \cite{HTT}.

\begin{theorem}
There is a simplicially enriched combinatorial model structure \(\sSet^+_{\Joyal}\) on \(\sSet^+\) using the enrichment and cotensoring described above whose fibrant objects are the quasicategories with the equivalences marked, whose fibrations between fibrant objects are the fibrations in \(\sSet_{\Joyal}\) where the adjunction above becomes a Quillen adjunction.
\end{theorem}

\section{Homotopy Limits and Colimits}\label{sec:hocolim}

\subsection{Model Categories of Functors - Proj and Inj}\label{subsec:modelFunc}

One way to produce new model categories that we will use multiple times is to produce a model structure on diagrams in a model category \(M\).
There are multiple ways to do this with some assumptions on \(M\) but not the indexing category \(D\).
See Propositions A.2.8.2 and A.3.3.2 of \cite{HTT}.

Let \(M\) be a cofibrantly-generated model category, which is enriched over another excellent model category \(S\) and let \(D\) be a small \(S\)-enriched category, all of whose hom sets are cofibrant in \(S\).
The definition of excellent is technical, but A.3.2 of \cite{HTT} provides \(\sSet_{\Quil},\sSet^+_{\Joyal},\) and \(\sSet_{\Joyal}\) as excamples of excellent model categories.  
Then the category of \(S\)-enriched functors comes with an \(S\)-enriched cofibrantly generated model category structure called the \emph{projective model structure} \([D,M]_{\proj}\), where the weak equivalences are pointwise weak equivalences in \(M\), and the fibrations are also defined pointwise.
A generating cofibration of \([D,M]_{\proj}\) is specified by a generating cofibration \(i\) of \(M\) and an object \(d\) of \(D\) to form the natural transformation \(D(d,-)\otimes i\).
The generating acylcic cofibrations are defined similarly.
For example, if \(D\) is a poset then \(D(d,d')\otimes i\) is \(i\) if \(d\leq d'\) and empty otherwise.
If \(M\) is also combinatorial, there exists another \(S\)-enriched model category on \(S\)-enriched functors \([D,M]_{\inj}\) whose weak equivalences and cofibrations are defined pointwise.
In this case, both \([D,M]_{\inj}\) and \([D,M]_{\proj}\) will be combinatorial.

Given an \(S\)-enriched functor \(f\colon D_1\to D_2\) between small \(S\)-enriched categories and a combinatorial \(S\)-enriched model category \(M\), the pullback \(f^{\ast}\) 
admits a left adjoint \(f_!\) given by the left Kan extension of \(f^{\ast}\) and a right adjoint \(f_{\ast}\) given by the right Kan-extension.
The pair \(f_!\dashv f^{\ast}\) form a Quillen adjunction between \([D_1,M]_{\proj}\) and \([D_2,M]_{\proj}\), while \(f^{\ast}\dashv f_{\ast}\) provide a Quillen adjunction between \([D_1,M]_{\inj}\) and \([D_2,M]_{\inj}\).
In the case where \(D_2= \ast\), the derived functors corresponding to \(f_!\) and \(f_{\ast}\) are the homotopy colimits and homotopy limits respectively.

The book \cite{CatHomoTheory} discusses homotopy limits and colimits in detail both in construction and in theory, but the most relevant example of a construction is given in Sections 6.1 and 6.2 of \cite{knotLatticeInvariance}.
These sections describe homotopy colimits in general as well as provides a construction in terms of a homotopy colimit that yields a homotopy equivalent result to the definition of the knot lattice space based off of \cite{knotsAndLattice}.
In this paper, we will be using this construction as the definitions of the lattice simplicial set and the knot lattice simplicial set in Sections \ref{subsec:LatticeDef} and \ref{subsec:knotLatticeDef}.

Meanwhile \cite{HTT} provides definitions of homotopy limits and colimits internal to a general quasi-category so that they fulfill a similar role to colimits and limits in ordinary categories.
In particular, Proposition 1.2.9.3 of \cite{HTT} shows that homotopy coherent (co)cones form a quasi-category, and additionally there is a notation of an object in a quasi-category being terminal, where an appropriate notion of a mapping space to the object is always contractible.
Proposition 4.2.4.1 of \cite{HTT} confirms that this agrees with the definition of homotopy colimit given above.
Further, Proposition 4.2.2.7 \cite{HTT} shows that homotopy colimits behave functorially.

\subsection{Reedy Category Structures}

Depending on the structure of \(D\), one can construct additional model category structures on \([D,M]\) with the same weak equivalences where being cofibrant is less restrictive than in \([D,M]_{\proj}\) or fibrant is less restrictive than \([D,M]_{\inj}\).
As such if one can identify a diagram as fibrant in this model structure then one will have a weak equivalence from the limit to the homotopy limit of that diagram.
Since weak equivalences between cofibrant-fibrant objects are actually the usual homotopy equivalences in a simplicially enriched model category (see Remark
 A.3.1.8 in \cite{HTT}), this will give a homotopy inverse from the homotopy limit to the limit allowing us to lift functors defined on the limit to the homotopy limit.
This new model structure, \([D,M]_{\Reedy}\) will have the identity into \([D,M]_{\inj}\) be a left Quillen adjoint (see A.2.9.23 in \cite{HTT}) and thus there will be a week equivalence from the homotopy limit s computed with \([D,M]_{\Reedy}\) to the homotopy limit as computed with \([D,M]_{\inj}\).

\begin{definition}
A small category \(R\) is \emph{Reedy} if it comes equipped with a degree function from the objects to some ordinal \(\omega\), \(d\colon \mbox{Ob}R\to \omega\), and two wide subcategories \(R_{-}\) and \(R_{+}\) so that
\begin{enumerate}
\item every morphism in \(R_+\) raises degree
\item every morphism in \(R_{-}\) lowers degree
\item every morphism in \(R\) factors uniquely as a map in \(R_{-}\) followed by a map in \(R_{+}\)
\end{enumerate}

\end{definition}

Note that a single small category can come with multiple different Reedy structures, and the model category we describe will depend on the specific Reedy structure given.
For example Figrues \ref{subfig:Reedy1a} and \ref{subfig:Reedy1b} put different Reedy structures on the same underlying category.

Now given a Reedy category \(R\) and functor \(F\colon R \to M\) where \(M\) is a model category, there's several maps we will have to define based on \(F\) before we can state the Reedy model category structure.
In particular given an object \(r\) in \(R\), the \emph{\(r^{th}\) latching object \(L_rF\)} is \(\colim_{s\to_+ r} F\) which is taken over the slice category \(R_+/r\) containing all objects except the identity \(\id_r\).
The \emph{\(r^{th}\) matching object \(M_rF\)} is \(\lim_{r\to_{-} s}F\) which is taken over the coslice category \(r/R_{-}\) containing all objects except \(\id_r\).
There are natural maps \(L_rF\to F(r)\to M_r F\)
The following theorem provides a weaker summary of results in Chapter 14 of \cite{CatHomoTheory}, which in turn references \cite[5.2.5]{hoveyMC} and \cite[18.4.11]{hirsch}
For an intuition as to why such a result should hold, one can view the Reedy category structure as giving instructions on how to inductively lift maps into the limit with the condition on the terminal functor in \(\sSet_{\Quil}\) being cofibrant providing a condition to ensure consistency as the lifts propagate.

\begin{theorem}
Let \(R\) be a Reedy category and \(M\) a simplicially enriched category. Then there exists a model category structure on \(M^R\) so that the cofibrant objects are those which for all \(r\in R\), the map \(L_rF\to F(r)\) is a cofibration and the fibrant objects are those for which all \(r\in R\) the map \(F(r)\to M_rF\) is a fibration.
If the functor in \((\sSet_{\Quil})^R\) which is constantly the terminal object is Reedy cofibrant then the limit functor is right Quillen, and thus computes homotopy limits on Reedy fibrant functors.
\end{theorem}

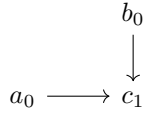
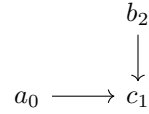
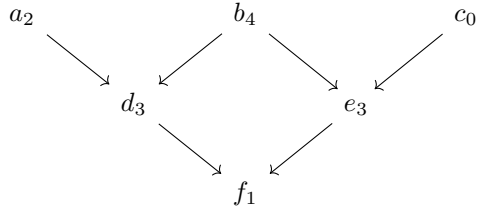
\begin{figure}
\begin{subfigure}[b]{.45\textwidth}
\[\begin{tikzcd}
 & b_0 \arrow[d] \\
a_0 \arrow[r] & c_1
\end{tikzcd}
\]
\caption{A first Reedy structure on the pullback category.}\label{subfig:Reedy1a}
\end{subfigure}
\hfill
\begin{subfigure}[b]{.45\textwidth}
\[\begin{tikzcd}
 & b_2 \arrow[d] \\
a_0 \arrow[r] & c_1
\end{tikzcd}
\]
\caption{A second Reedy structure on the pullback category.}\label{subfig:Reedy1b}
\end{subfigure}

\begin{subfigure}[b]{.9\textwidth}
\[
\begin{tikzcd}
a_2 \arrow[rd] & & b_4\arrow[ld]\arrow[rd] & & c_0 \arrow[ld] \\
& d_3 \arrow[rd] & & e_3\arrow[ld] & \\
& & f_1 & &
\end{tikzcd}
\]

\caption{A third example of a reedy structure on a category.}\label{subfig:Reedy2}
\end{subfigure}

\caption{Examples of Reedy categories where the degree is listed as a subscript on the object. These examples have all morphisms in either \(R_{-}\) or \(R_{+}\) based on how they shift degrees, as there are no notrivial compositions of \(R_{-}\) followed by \(R_{+}\).}\label{fig:ReedyExamples}

\end{figure}

\begin{example}
Both Figure \ref{subfig:Reedy1a} and Figure \ref{subfig:Reedy1b} provide Reedy category structures on the pullback category with the degree of an object denoted by subscript.
In Figure \ref{subfig:Reedy1a} all matching objects are terminal due to \(R_{-}\) being empty, and as such all pointwise fibrant functors are Reedy fibrant in this category.
However, for the constant terminal functor, the latching object at \(c\) is two points, and the map going from two points to a single point is not a monomorphism in \(\sSet_{\Quil}\).
In Figure \ref{subfig:Reedy1b}, the matching object over \(b\) is precisely the object at \(c\) with all other matching objects being terminal. 
Hence, in that Reedy model category an object is fibrant if the map from \(b\) to \(c\) is a fibration of fibrant objects and \(a\) is also fibrant.
Furthermore, for the terminal functor in \(\sSet_{\Quil}\) the only non-initial latching object is at \(c\), but that is simply the point at \(a\), which includes into the point at \(c\).
Hence all latching maps are cofibrations as needed.
\end{example}

\begin{example}\label{ex:Reedy2}
Figure \ref{subfig:Reedy2} presents a more complicated Reedy category, though one can still check that due to no nontrivial compositions of morphisms in \(R_{-}\) followed by morphisms \(R_{+}\) the unique factorization statement does hold.
The nonterminal matching maps are from 
\begin{enumerate}
\item \(a_2\) to \(f_1\),
\item  \(d_3\) to \(f_1\),
\item \(e_3\) to \(f_1\), and
\item \(b_4\) to the pullback of \(d_3\) and \(e_3\) over \(f_1\).
\end{enumerate}
The only non-initial latching maps are from 
\begin{enumerate}
\item \(c_0\) to \(e_3\),
\item \(c_0\) to \(f_1\), and
\item \(a_2\) to \(f_1\).
\end{enumerate}
All of these are cofibrant for the terminal functor in \(\sSet_{\Quil}\).
\end{example}

\section{Topological background} \label{sec:topBack}

Let \(Y\) be a rational homology sphere and \([K]\) an element of \(H_1(Y;\Z)\) which we will think of as representing the homology class of a knot.
We are going to construct some quasi-categories \(\Cc^{Y,[K]}_{\Gamma}\) an \(\Cc^{Y,[K]}_{\Jk}\) that model potential knot lattice spaces with respectively the extra data of the flip map \(\Gamma\) or the involutive structures \(\Ii\) and \(\Jj\).
However, before we construct those, we will need to understand the grading information used in the filtration.
Furthermore, constructing the surgery formula for a two handle attachment will require not only identifying a domain and codomain for the functor but also a more fine tuned understanding of how to specify a two handle attachment equipped with a \Spinc structure.
We will also need how that specification interacts with the actions on those \Spinc stuctures by the core and co-core, the restriction of \Spinc structures to the boundary, as well as conjugation of \Spinc structures.
While we are doing this with an eye towards generalized algebraic knots, this section will not directly be making such assumptions in order to better highlight how these techniques may be more widely applicable in the context of knot Floer.

Note this will parallel the more commonly discussed case of surgery on null-homologous knots, but conventions when knots are rationally null-homologous are not as well established.
In particular, while we will parallel the discussion in Section 22 of \cite{InvolMappingCone}, our notation will differ slightly to more closely match that of \cite{knotsAndLattice}.
Furthermore, a key difference as noted in remark 22.5 of \cite{InvolMappingCone} is that there are inconsistencies in the literature of how a knot associated to a doubly pointed Heegaard diagram is oriented, which lead to further inconsistencies in how the knot should act on \Spinc structures in knot Floer homology.
The action of the knot in our formulation of the knot lattice simplicial set uses the opposite orientation as used in \cite{InvolMappingCone}, and this has subtle effects on various formulas.


\subsection{Indexes for Filtrations}\label{subsec:Cosets}

To start let \([\Q:2\Z]\) represent the groupoid whose elements are cosets of \(2\Z\) in \(\Q\) and the morphisms from \(q_1+2\Z\) to \(q_2+2\Z\) are elements of \((q_2-q_1)+2\Z\).
Composition is given using addition.

\begin{definition}
The \emph{three-dimensional \Spinc cobordism category \(\Cob^{\Spin^c}(3)\)} is the category whose objects are pairs \((Y,\tf)\) where \(Y\) is a closed connected oriented three-manifold and \(\tf \in \SpincX{Y}\) and whose morphisms are given by 
\[(W,\sfk):(Y_1,\tf_1)\to (Y_2,\tf_2).\]
Here \(W\) is diffeomorphism class of  smooth cobordism from \(Y_1\) to \(Y_2\), i.e. a four manifold \(W\) whose boundary is partitioned into two pieces, one of whom has an orientation reversing diffeomorphism to \(Y_1\) and the other has an orientation preserving diffeomorphism to \(Y_2\).
Additionally, \(\sfk \in \SpincX{W}\) must satisfy \(\sfk|_{Y_1}=-\tf_1\) and \(\sfk|_{Y_2}=\tf_2\).
Composition is given by concatenating cobordisms along collared neighborhoods of the boundary.
Denote by \(\Cob^{\Spin^c}_{\Q}(3)\) the full subcategory of \(\Cob^{\Spin^c}(3)\) spanned by the rational homology spheres.
\end{definition}

Note that by assuming our objects are rational homology spheres we have that the square of a \Spinc structure's first chern class, the signature \(\sigma\), and the first betti number \(b_1\) are additive under composing cobordisms.
Furthermore, \(c_1(\sfk)^2-\sigma\) will give an integer divisible by 8 whenever \(W\) is a cobordism between integral homology spheres.
The collection of \Spinc structures on \(W\) that restrict to specific \Spinc structures on \(\bou W\) form a torsor under the action of \(H_2(W;\Z)\).
This action of \(H_2(W;\Z)\) on \(\sfk\) will change \(c_1(\sfk)^2\) by a multiple of 8.
This establishes that there is a functor \(\grf\colon \Cob^{\Spin^c}_{\Q}(3)\to [\Q:2\Z]\) so that
\begin{align*}
\grf(S^3,\tf)&:= 2\Z \\
\grf(W,\sfk)&:= \frac{c_1(\sfk)^2-\sigma(W)}{4}.
\end{align*}
Note that
\[\grf(Y_1\# Y_2,\tf_1\# \tf_2) = \grf(Y_1,\tf_1)+\grf(Y_2,\tf_2).\]

\begin{remark}
The calculation of \(\grf(\tf)\) is be reminiscent of the grading shift formula from Heegaard Floer, but off by the value of \(b_1-b_2^+\)
However, this matches the Seiberg Witten Floer Homotopy theory where \(b_2^+\) contributes to the grading shift in a morphism via a different (de)suspension operator (and in fact the one associated with classical suspension), than the one used for the \(\grf(W,\sfk)\) stated above \cite{SWFhomotopy,PeriodicFloerSpectra}.
Similarly the contribution of the \(b_1\) grading shift occurs has a different source in Seiberg Witten Floer Homotopy and thus would  make sense to model separately \cite{SWFhomotopy}.
Regardless, all of our cobordisms will be negative definite and composed entirely of two handles, and thus this distinction will not matter here.
\end{remark}

\subsection{Specifying two-handles with \Spinc structures}\label{subsec:backTopSpecify}

We are concerned here with the effect on the 3-manifold and knot given by adding a single 2-handle to \(Y\) along \(K\) to give a cobordism \(W\colon Y \to Y'\), and thus we first need to be able to specify a choice of 2-handle attachment. 
This handle attachment is determined by a choice of framing, i.e. an identification up to isotopy of a tubular neighborhood \(N\) of \(K\) with the attaching region of the 2-handle \(D^2\times S^1\), which only needs to be done up to isotopy.
This in turn is determined by the linking of \(K\) and \(\lambda\), where \(\lambda\) is the image of \(p\times S^1\) for some nonzero \(p\in D^2\).
Because \(K\) is only rationally null-homologous this linking not necessarily given by an integer but an element of the coset \(\lkc([K],[K])\in \Q/\Z\), where \(\lkc\) is the linking form on \(H_1(Y;\Z)\).

We will now translate this to information about surfaces in \(W\) and its second homology classes.
Let \(\Sigma_1\) be the rational second homology class in \(H_1(Y,K;\Q)\) with boundary \([K]\), \(C\) the core of the two handle oriented so that \(\bou C\) is \(K\) with the opposite orientation, and \(D\) the cocore oriented so that the intersection of \(C\) with \(D\) is positively oriented. 
The boundary of \(D\) is the dual knot \(\mu\).
Together \([C]\) and \(\Sigma_1\) form a class \(\Sigma \in H_2(W;\Q)\), and the value of \(\Sigma^2\) will belong to the coset \(\lkc([K],[K])\).
Furthermore we will represent the cobordism \(W\) by \(W_{\Sigma^2}(K)\) and the surgered three-manifold by \(Y_{\Sigma^2}(K)\)
Note that \(\Sigma^2\neq 0\) implies that the framing \(\lambda\) is not parallel to a rational seifert surface and thus \(Y_{\Sigma^2}(K)\) is a rational homology sphere.
For ease of notation, we will often denote \(\Sigma^2\) by \(r\), and since all of our surgeries come from two-handle attachments, whenever discussing \(Y_r(K)\) it will be a part of the assumption that \(r\in \lkc([K],[K])\).

The affine modeling of \(\SpincX{W_{\Sigmasq}(K)}\) on second cohomology respects the long exact sequence, so we a short exact sequence of affine spaces given by
\[0 \to \Z \to \SpincX{W_{\Sigmasq}(K)} \to \SpincX{Y}\to 0,\]
where the \(\Z\) represents the action of the cocore of the two-handle.
Let 
\[\hat{A}\colon \SpincX{W_{\Sigmasq}(K)} \to \Q\]
be defined by
\[\hat{A}(\sfk):= \frac{c_1(\sfk)(\Sigma) + \Sigmasq }{2}.\]
Equivalently if \(\hat{A}(\sfk)=i\) then \(c_1(\sfk)[\Sigma]=2i-\Sigma^2\).
Note that
\begin{align*}
\hat{A}(\sfk+n[D]) &= \frac{c_1(\sfk+n[D])(\Sigma)+\Sigmasq}{2} \\
&\frac{(c_1(\sfk)+2n\PD[D])(\Sigma) +\Sigmasq}{2} \\
&= \frac{c_1(\sfk)(\Sigma)+\Sigmasq}{2}+n\\
&= \hat{A}(\sfk)+n.
\end{align*}
As such, \(\hat{A}\) detects the action of the cocore and a choice of \(\sfk \in \SpincX{W_{\Sigmasq}(K)}\) are determiend by \(\sfk|_{Y}\) and \(\hat{A}\).

For \(\tf \in \SpincX{Y}\), define 
\begin{align*}
\Ac(Y,K,\tf)&:= \{\hat{A}(\sfk)\, |\, \sfk \in \SpincX{W_{\Sigmasq}(K)}\, \sfk|_{Y}= \tf\}\\
\Ac(Y,K) &= \{(\tf,i) \in \SpincX{Y}\times \Q\, |\, i \in \Ac(Y,K,\tf)\}.
\end{align*}
The above discussion guarantees that the map 
\[\phi\colon \SpincX{W_{\Sigmasq}(K)}\to \Ac(Y,K)\]
given by \(\phi(\sfk)= (\sfk|_Y,\hat{A}(\sfk))\) is a bijection.
However, we still need to identify the sets \(\Ac(Y,K,\tf)\) in order to know when \((\tf,i)\) actually corresponds to a \Spinc strcutre over \(W_{\Sigmasq}(K)\).

Prima facie, \(\Ac(Y,K,\tf)\) is a coset of \(\Z\) in \(\Q\) which may also depend on \(\Sigmasq\), but the following propositon shows that it does not.
In fact, the coset can be recovered from the \(\grf\) functor that we already will be tracking.
We have chosen to represent the pieces that go in to \(\Ac(Y,K,\tf)\) using varaitions on the letter \(A\), precisely because if \(\grf\) represents coset containing the even portion of the maslov grading, \(\Ac(Y,K,\tf)\) will represent the coset of \(\Z\) containing the Alexander grading. 

\begin{proposition}\label{prop:AlexCoincide}
If \(\Sigmasq\neq 0\), the cosets \(\Ac(Y,K,\tf)\) and \(\frac{\grf(Y,\tf)-\grf(Y,\tf+[K])}{2}\) of \(\Z\) in \(\Q\) coincide.
\end{proposition}
\begin{proof}
Let \(X\) be some simply connected 4-manifold that \(Y\) bounds, let \(F\) be a surface in \(X\) that \(K\) bounds and let \(\tilde{F}\) be \(F\) capped off with \(-\Sigma_0\) in \(X\).
Let \(\sfk_X\) be a \Spinc structure on \(X\) that restricts to \(\tf\) on \(Y\) and let \(\sfk_W\) be defined similarly but with \(W\) instead of \(X\).
By \(c_1(\sfk_X\#\sfk_W)\) being a characteristic cohomology class, the following must be an integer
\begin{align*}
\frac{c_1(\sfk_X\#\sfk_W)^2-c_1(\sfk_X\#\sfk_W + [F\sqcup_K C])^2}{8} &= \frac{c_1(\sfk_X)^2-c_1(\sfk_X+[F])^2}{8} \\
& - \left( \frac{c_1(\sfk_W +[C])^2 -c_1(\sfk_W)^2}{8}\right) \\
&= \frac{\grf(\sfk)-\grf(\sfk+[K])}{2} \\
&- \left( \frac{c_1(\sfk_W +[C])^2 -c_1(\sfk_W)^2}{8}\right).
\end{align*}
As such it remains to show that
\[ \hat{A}(\sfk_W) = \frac{c_1(\sfk_W +[C])^2 -c_1(\sfk_W)^2}{8}.\]

In calculating \(c_1(\sfk_W+[C])^2\) we are doing so over \(\Q\) using the identification of \(H^2(W_{\Sigmasq}(K);\Q)\) with \(H^2(W_{\Sigmasq}(K), \bou W_{\Sigmasq}(K);\Q)\) that leverages that both \(Y\) and \(Y_{\Sigmasq}(K)\) are rational homology spheres, a fact that comes from the assumption that \(\Sigmasq\neq 0\).
In this identification, the action of \([C]\) becomes identified with  the action of \(\Sigma\).
As such,
\begin{align*}
\frac{c_1(\sfk_W +[C])^2 -c_1(\sfk_W)^2}{8} &= \frac{\left(c_1(\sfk_W)+2\PD[\Sigma]\right)^2 -c_1(\sfk_W)^2}{8} \\
&=\frac{4c_1(\sfk_W)[\Sigma] +4\Sigmasq}{8} \\
&=\hat{A}(\sfk_W).
\end{align*}

\end{proof}
\begin{remark}
The statement is also true if \(\Sigmasq =0\), especially since in that case our knot was null homologous, and both cosets must be \(\Z\).
However, since we care primarily about \(\Sigmasq<0\) that is outside of our focus.
\end{remark}

Finally, we will use our description of \(\SpincX{W_{\Sigmasq}(K)}\) in terms of \(\Ac(Y,K)\) to understand \(\SpincX{Y_{\Sigmasq}(K)}\).
In particular a similar argument as used on \(Y\) works on \(Y_{\Sigmasq}(K)\) and we have a short exact sequence of affine spaces
\[ 0\to \Z \to \SpincX{W_{\Sigmasq}(K)}\to \SpincX{Y_{\Sigmasq}(K)}\to 0,\]
where this time \(\Z\) represents the action of the core rather than the cocore.
Here note that for \(\sfk \in \SpincX{W_{\Sigmasq}(K)}\)
\begin{align*}
(\sfk+n[C])|_Y&= \sfk|_Y+n[K]\\
\hat{A}(\sfk+n[C]) &= \frac{c_1(\sfk+n[C])[\Sigma] +\Sigmasq}{2}\\
&= \frac{c_1(\sfk)+2n\PD[C]\cdot[\Sigma] +\Sigmasq}{2} \\
&=\frac{c_1(\sfk)+\Sigmasq}{2}+n\Sigmasq \\
&=\hat{A}(\sfk)+n\Sigmasq.
\end{align*}
As such, \(\phi\) descends to a well defined map 
\[\phi_{\Sigmasq} \colon\SpincX{Y_{\Sigmasq}(K)} \to \Ac(Y,K)/\sim\]
where \(\sim\) is the relation given by
\[(\tf,i) \sim (\tf+n[K],i+n\Sigmasq).\]
We will denote elements of \(\Ac(Y,K)\) with parenthetical \((\tf,i)\) and drop the paratheticals as they are used as in puts in other functions, while elements of \(\Ac(Y,K)/\sim\) with square brackets that are never dropped \([\tf,i]\) to emphasize that the later is an equivalence class, though we will also attempt to provide enough context to disambiguate whether we are representing an element of \(\SpincX{W_{\Sigmasq}(K)}\) or \(\SpincX{Y_{\Sigmasq}(K)}\).

\subsection{A diagram of \Spinc structures}\label{subsec:Ic}
The \Spinc structures over \(W_{\Sigmasq}(K)\) can be arranged into a small category \(\Ic_{Y,K,\Sigmasq}\), which will not only be useful in describing the functor \(\Xb_{\Sigmasq}\) and its specializations but also has an interpretation in terms of the interactions between knot Floer homology and Heegaard Floer homology.
In particular the objects of \(\Ic_{Y,K,\Sigmasq}\) are a disjoint union of two copies of \(\Ac(Y,K)\) using the the additons of \(a\) and \(b\) to the tupple to distinguish them. 
For morphisms we will have for every \((\tf,i) \in \Ac(Y,K)\),
\begin{align*}
\lambda_{\tf,i}\colon& (\tf,i,a)\to (\tf,i,b) \\
\rho_{\tf,i}\colon &(\tf,i,a)\to (\tf+[K],i+\Sigmasq,b).
\end{align*}
Additionally note that \(\Ic_{Y,K,\Sigmasq}\) decomposes as a direct sum of the categories \(\Ic_{Y,K,\Sigmasq,[\tf,i]}\).

The \((\tf,i,b)\) should be thought of as recording the different Heegaard Floer morphisms for each \((W_{\Sigmasq}(K),\sfk)\). 
Meanwhile the \((\tf,i,a)\) can be (loosely) thought of as providing a particular way to manifest the relation between \((\tf,i,b)\) and \((\tf+[K],i+\Sigmasq,b)\).
The construction of \(\Xb_{\Sigmasq}\) and its specializations will directly use this category, so when we discuss various actions on \(\SpincX{W_{\Sigmasq}(K)}\), it will help for us to view those as actions on \(\Ic_{Y,K,\Sigmasq}\) 
As an example, not only can the cocore \(D\) act on \(\SpincX{W_{\Sigmasq}(K)}\), but it also acts on \(\Ic_{Y,K,\Sigmasq}\) inducing a shifting automorphism \(\Xi\) which simply changes every index of \(i\) to \(i+1\), and this action takes \(\Ic_{Y,K,\Sigmasq,[\tf,i]}\) to \(\Ic_{Y,K,\Sigmasq,[\tf,i]+[\mu]}\).

\subsection{The specified grading shifts}\label{subsec:gradShift}

At this point we will attempt to understand the functions \(\grf(\phi^{-1}(\tf,i))\) and \(\grf(\phi_{\Sigmasq}^{-1}([\tf,i]))\) when \(\Sigmasq<0\).
We will denote the later as \(\grf([\tf,i])\)  taking values in \([\Q:2\Z]\), while we will find the former depends only on \(i\) and \(\Sigmasq\) but not \(\tf\) and thus write it as \(\grf_{\Sigmasq}(i)\).
The later will be directly necessary to identify the codomain of our functors, but given the significance of each \((\tf,i)\) in \(\Ic_{Y,K,\Sigmasq}\), we will end up directly using each \(\grf_{\Sigmasq}(i)\).
To perform this calculation first note that for \(K \in H^2(W_{\Sigmasq}(K);\Q)\), \(K^2= \frac{K(\Sigma)^2}{\Sigmasq}\).
Thus, for \(\sfk\in \SpincX{W_{\Sigmasq}(K)}\) with \(\hat{A}(\sfk)=i\), we have
\begin{align*}
\grf_{\Sigmasq}(i)&:=\grf(\sfk) \\
&= \frac{\frac{c_1(\sfk)[\Sigma]^2}{\Sigmasq} +1}{4} \\
&=\frac{c_1(\sfk)[\Sigma]^2+\Sigmasq}{4\Sigmasq}\\
&=\frac{\left(2i-\Sigmasq\right)^2+\Sigmasq}{4\Sigmasq}.
\end{align*}

\subsection{Action of conjugation on specification}\label{subsec:SpecConj}

We will need to understand how two different types of conjugation of \Spinc structures affect our method of specification.
In particular, there will be \emph{conjugation} \(\overline{\sfk}\) and \emph{\(C\)-translated conjugation} \(\overline{\sfk+[C]}\).

Starting with the usual conjugation, first observe that conjugation respects restriction so \(\overline{\sfk}|_Y=\overline{\sfk|_Y}\).
Additionally, for \(\sfk \in \SpincX{W_{\Sigma^2}(K)}\)
\begin{align*}
\hat{A}(\overline{\sfk}) &= \frac{c_1(\overline{\sfk})[\Sigma]+\Sigmasq}{2} \\
&= \frac{-c_1(\sfk)[\Sigma]+\Sigmasq}{2} \\
&= \frac{-(2\hat{A}(\sfk)-\Sigmasq)+\Sigma^2}{2}\\
&=-\hat{A}(\sfk)+\Sigmasq.
\end{align*}
Therefore, conjugation takes the \Spinc structure specified by \((\tf,i)\) to \((\overline{\tf},-i+\Sigma^2)\).

Now consider \(C\) translated conjugation.
This also plays well with restriction to \(Y\) in that \(\overline{\sfk+[C]}|_Y= \overline{\sfk|_Y +[K]}\).
Furthermore,
\begin{align*}
\hat{A}(\overline{\sfk+[C]}) &= \frac{c_1(\overline{\sfk+[C]})[\Sigma] +\Sigmasq}{2} \\
&=\frac{(-c_1(\sfk)-2\PD[C])(\Sigma) +\Sigmasq}{2} \\
&=\frac{-(2\hat{A}(\sfk)-\Sigma^2)-2\Sigmasq+\Sigmasq}{2} \\
&= -\hat{A}(\sfk).
\end{align*}
Hence, \(C\)-translated conjugation takes the \Spinc structure specified by \((\tf,i)\) to the \Spinc structure specified by \((\overline{\tf+[K]},-i)\).

Together these two types of conjugation define an automorphism \(\psi\) of \(\Ic_{Y,K,\Sigmasq}\), where
\begin{align*}
\psi(\tf,i,a)&=(\overline{\tf+[K]},-i,a) \\
\psi(\tf,i,b)&=(\overline{\tf},\Sigmasq-i,b)\\
\psi(f_{\tf,i})&=g_{\overline{\tf+[K]},-i} \\
\psi(g_{\tf,i})&=f_{\overline{\tf+[K]},-i}.
\end{align*}

Note that because \(c_1(\sfk)^2=c_1(\overline{\sfk})^2\), we have that for all cobordisms in \(\Cob^{\Spin^c}_{\Q}\), \(\grf(Y,\overline{\tf})=\grf(Y,\tf)\).
Furthermore, acting by \([K]\)-translated conjugation on \((Y,\tf)\) will take the pair \(\left(\grf(Y,\tf),\grf(Y,\tf+[K])\right)\) to the pair
\[\left(\grf(Y,\overline{\tf+[K]}),\grf(Y,\overline{\tf})\right)=\left(\grf(Y,\tf+[K]),\grf(Y,\tf)\right)\].

\section{Quasi-Categories to model knot lattice} \label{sec:quasicats4CFK}

We will now define quasi-categories \(\Cc_{Y,[K]}^{\Gamma}\) and \(\Cc_{Y,[K]}^{\Jk}\) to model the additional information needed for surgery and the involutive data.
Note that constructing the quasi-category \(\Cc_{Y,[K]}^{\Jk}\) will be much more intensive than constructing \(\Cc_{Y,[K]}^{\Gamma}\), both due to modeling the symmetries of the involutive data and because the involutions live in different categories.

\subsection{Quasicategories for filtered and doubly-filtered spaces}\label{subsec:FilteredsSet}

Before constructing \(\Cc_{Y,[K]}^{\Gamma}\) and \(\Cc_{Y,[K]}^{\Jk}\), we will need to construct two more specific quasi-categories \(\Cc_{Y}\) and \(\Cc_{Y,[K]}\) which will model the singly filtered and doubly-filtered spaces that lattice homology and knot lattice homology produces.

In particular, from our discussion in Section \ref{subsec:Cosets}, the functor \(\grf\) allows us to associate to a rational homology three-sphere with \Spinc structure \((Y,\tf)\) a coset of \(2\Z\) in \(\Q\).
Note that as a subset of \(\Q\), the set \(\grf(Y,\tf)\) inherits a poset structure, which can be interpreted as a small category enriched over \(\sSet_{\Quil}\) discretely.

Given a poset \(P\) consider the model category \(\Filt_{P}:=[P^{op},\sSet_{\Quil}]_{\proj}\).
Note that we are taking contravariant functors, which may reflect the opposite poset structure than one would expect but is more natural to the context of Heegaard Floer homology.
We will often refer to objects in this category as \(X_{\ast}\) to emphasize the diagrammatic nature, where \(X_n\) represents the object at height \(n\) and \(X:=\colim_{P}X_{\ast}\).
While as a category \(\Filt_{P}\) contains all diagrams of shape \(P^{op}\), the cofibrant objects in this model category are particularly the simplicial sets with a height-induced filtration of shape \(P\), i.e. admitting the following description.
For a filtered simplical set \(X_{\ast}\) each simplex \(\sigma\) of \(X\) has been assigned a height \(h(\sigma)\in P\) so that for any other simplex \(\sigma'\) in the boundary of \(\sigma\), \(h(\sigma)\leq h(\sigma')\).
To recover the diagram associated to this filtration one takes super level sets and the inclusions between them.
Morphisms in this context are filtered if they are increasing with respect to the height function, i.e. \(f\colon X\to Y\) is filtered if for simplices \(\sigma\in X\), \(h_X(\sigma)\leq h_Y(f(\sigma))\).

In the case of \(\Filt_{P\times Q}\), we can view our height function on cofibrant objects as coming from two height functions: one to \(P\) and one to \(Q\).
A morphism of underlying simplicial sets \(f\colon X \to Y\) is in \(\Filt_{P\times Q}\) if it is filtered with respect to both height functions.
Hence, letting \(M^{\circ}\) represent the full subcategory of fibrant-cofibrant objects in a model category \(M\), \(\Filt_{P\times Q}^{\circ}\) is the pullback of \(\Filt_P^{\circ}\) and \(\Filt_Q^{\circ}\) over \(\sSet_{\Quil}^{\circ}\).
As such, we call objects in \(\Filt_{P\times Q}^{\circ}\) doubly filtered simplicial sets.

\begin{remark}
If one prefers one can instead use another simplicially enriched cartesian model category eqivalent to \(\sSet_{\Quil}\) such as \(\Top_{\Quil}\).
In fact, the lattice and knot lattice spaces are best described as cubical sets, and cubical sets with connections (an extra form of degeneracy) form a cartesian model category Quillen equivalent to \(\sSet_{\Quil}\).
However, in the face of \(\infty\)-category theory’s current preference for simplicial sets, using cubical sets would be more trouble than it is currently worth.
\end{remark}
\begin{remark}
Additionally one could instead jump directly to using infinity categories with \([N(P^{op}),\Rf\sSet_{\Quil}^{\circ}]\), and Proposition 4.2.4.4 of \cite{HTT} ensures this is weakly equivalent as a quasi-category to \(\Rf\left([P^{op},\sSet_{\Quil}]^{\circ}\right)\), which we will be using.
Our objects of interest all land in \([P^{op},\sSet_{\Quil}]^{\circ}\) and the model category structure provides extra leverage for constructions so this is what we will use.
However, this observation is relevant interpolating the discussion here with Zemke's work in \cite{ZemkeWhat}, which uses \(A_{\infty}\)-modules, a structure behaving much like homotopy coherent diagrams, in a way that does not distinguish between the morphisms from the poset and the other morphisms we would like to model in \(\Cc_{Y,[K]}^{\Gamma}\) or \(\Cc_{Y,[K]}^{\Jk}\).
\end{remark}

We will now discuss some Quillen adjunctions on the model categories \(\Filt_{[q_1]}\) and \(\Filt_{[q_1,q_2]}\) which will be helpful for us later.
The discussion here will parallel the discussion in Section 3 of \cite{knotLatticeInvariance}.
However, unlike in \cite{knotLatticeInvariance} to verify that these functors play well with the model category structure, we have used defined them using Kan extensions, and thus need to then reverify that they do indeed play as intended.
In particular given \(q_1,q_2,q_3,q_4\in \Q\), we have that the action of \(q_3\) on \(\Q/2\Z\) and \((q_3,q_4)\) on \(\Q\times \Q/(2\Z\times 2\Z)\) provide functors which for \(X_{\ast} \in \Filt_{[q_1]}\) and \(Y_{\ast\ast} \in \Filt_{[q_1,q_2]}\),
\begin{align*}
X[q_3]_q &= X_{q-q_3} \\
Y[q_3,q_4]_{q,q'}&= Y_{q-q_3,q'-q_4}.
\end{align*}
Note that this is the pushforward rather than the pullback, and if \(h_X\) was the height function on \(X_{\ast}\) then the height function on \(X_{\ast}[q_3]\) is \(h_X+q_3\).

Given cosets \([q_1],[q_2]\) cosets of \(2\Z\) in \(\Q\), then the canonical projections \(p_1,p_2\) from \([q_1]\times[q_2]\) to \([q_1]\) and \([q_2]\) respectively induce left Quillen functors \(p_{1,!},p_{2,!}\) using their left-Kan extensions.
There is a left Quillen bifunctor
\[\times\colon \Filt_{[q_1]}\times \Filt_{[q_2]}\to \Filt_{[q_1,q_2]}\]
defined by \((X\times Y)_{(q_1,q_2)}=X_{q_1}\times X_{q_2}\).
By composing \(\times\) with \(+_!\) we get a new left Quillen bifunctor
\[\otimes \Filt_{[q_1]}\times \Filt_{[q_2]}\to \Filt_{[q_1+q_2]}.\]
Note that one can equivalently define this functor for doubly filtered simplical sets to get
\[\otimes \Filt_{[q_1,q_2]}\times \Filt_{[q_3,q_4]}\to \Filt_{[q_1+q_3,q_2+q_4]}.\]
Given \(i \in \frac{[q_1]-[q_2]}{2}\) let \(A_i\colon [q_1] \to [q_1,q_2]\) given by \(A_i(q)=(q,q-2i)\).
The right Quillen functor \(A_i^{\ast}\) can be seen as picking out the Alexander grading \(i\) portion of \([q_1,q_2]\).

Finally, define
\[\sigma\colon \Filt_{[q_1,q_2]} \to \Filt_{[q_2,q_1]}\]
as induced by the map \((\Q\times \Q)/(2\Z\times 2\Z)\) swapping the first and second factors.
In terms of height functions on fitlered simplicial sets, \(\sigma\) presrves the underlying simplicial set but turns the height function \((h_1,h_2)\) into \((h_2,h_1)\).
Given \(X_{\ast\ast}\in \Filt_{[q_1,q_2]}\) and \(Y_{\ast\ast}\in \Filt_{[q_2,q_1]}\), a morphism \(f\colon X\to Y\) of the underlying simplicial sets is called \emph{skew-filtered} if \(f\colon X_{\ast\ast}\to \sigma(Y_{\ast\ast})\) is filtered.

The following propositions allow us to similarly interpret the other functors on filtered simplicial sets in terms of the height functions on those simplicial sets.
For example, 
\begin{proposition}\label{prop:forgeth1}
When evaluated on a doubly filtered simplicial set \(X_{\ast\ast}\in \Filt_{[q_1,q_2]}\) with height functions \(h_1\) and \(h_2\), the functors \(p_{1,!}\) and \(p_{2,!}\) have the effect of preserving the underlying simpicial set remembering respectively \(h_1\) and \(h_2\).
\end{proposition}
\begin{proof}
We will prove this for \(p_{1,!}\) with \(p_{2,!}\) going similarly.
Writing \(p_{1,!}\) as a coend and using the definition of the hom space of \([q_1]\) we have
\[p_{1,!}(X)_q= \coprod_{\substack{q_1\geq q_3\geq q \\ q_2 \geq q_4}}X_{q_3,q_4} \rightrightarrows \coprod_{q_1\geq q,q_2} X_{q_1,q_2},\]
where the two maps would be the inclusion of \(X_{q_3,q_4}\) into \(X_{q_1,q_2}\) and sending \(X_{q_3,q_4}\) to itself.
However, note that all each \(X_{q_1,q_2}\) will become identified with its image in \(X_{q,q_2}\), and the rest of the identifications realize \(p_{1,!}(X)_q=\colim_{[q_2]}X_{q,\ast}\).
That this underlying simplicial set of \(p_{1,!}(X)\) is the same underlying simplicial set of \(X_{\ast\ast}\) comes down to colimits commuting. 
Additionally, for a given simplex \(\sigma\) of \(X\), the inclusion \(i_{q_1,q_1}\colon X_{q_1,q_2} \to X\) has \(\sigma\) in its image if and only if \((q_1,q_2)<(h_1,h_2)\), and \(\sigma\) will be in the image of \(\colim_{[q_2]}X_{q,\ast}\) if and only if \(q<h_1\) as needed.
\end{proof}

\begin{proposition}\label{prop:heightSum}
For \(X_{\ast}\in \Filt_{[q_1]}\) and \(Y_{\ast} \in \Filt_{[q_2]}\) filtered spaces with height functions \(h_X\) and \(h_Y\),\(X_{\ast}\otimes Y_{\ast}\) has underlying simplicial set \(X\times Y\) and the height \(h\) on a simplex is
\[h(\sigma_1,\sigma_2)=h_{X}(\sigma_1)+h_Y(\sigma_2).\]
 Similarly, for doubly filtered simplicial sets the height function on \(X_{\ast\ast}\otimes Y_{\ast\ast}\) is induced by adding the height functions on the product \(X\times Y\). 
\end{proposition}
\begin{proof}
We focus on the singly filtered case as the doubly filtered case follows similarly.
We break this down into two statments, one about \(\times\) having underlying siplicial set \(X\times Y\) and height \((h_X,h_Y)\) and one about \(+_!\) preserving the underlying simplicial set taking a height function \((h_1,h_2)\) to \(h_1+h_2\).
The first statement follows from the colimits in question being directed and the second follows similarly to Proposition \ref{prop:forgeth1}, as it comes from a larger statement that for left kan extensions of poset maps \(\phi\)  the height function changes by applying \(\phi\) to the previous height funciton.
For \(Z_{\ast} \in \Filt_{[q_1,q_2]}\) the coend to describe \(+_!(Z)_{q}\) will be the same as taking the coproduct of \(Z_{q_1,q_2}\) with \(q_1+q_2=1\) with their overlap identified together as it is in \(Z\). 
\end{proof}

\begin{proposition}
Given a doubly filtered simplicial set \(X_{\ast\ast}\in \Filt_{[q_1,q_2]}\) with height functions \(h_{X,1}\) and \(h_{X,2}\), \(A_i^{\ast}(X_{\ast\ast})\) will be a filtered simplicial set with height function \(h(\sigma)=\min\{h_{X,1},h_{X,2}+2i\}\).
\end{proposition}
\begin{proof}
As \(A_i^{\ast}\) is right rather than left Quillen it is not guarunteed to preserve cofibrancy in the projective model structure. 
However, it is preserved that all of the structure maps of the diagram are monomorphisms and thus some filtered space, and since \(A_i\) is coinitial in \([q_1,q_2]\), \(A_i^{\ast}(X_{\ast\ast})\) will have the same underlying simplicial set as \(X\) following a similar argument to that used for \(p_{1,!}\) in  Proposition \ref{prop:forgeth1}.
Additionally, a simplex \(\sigma\) is in \(A_i^{\ast}(X_{\ast\ast})_q\) if \(h_{X,1}(\sigma)\geq q\) and \(h_{X,2}(\sigma)\geq q-2i\).
However, this is equivalent to both \(h_{X,1}(\sigma)\geq q\) and \(h_{X,2}(\sigma)+2i\geq q\) or \(\min\{h_{X,1}(\sigma),h_{X,2}(\sigma)+2i\}\geq q\), as needed.
\end{proof}

\begin{proposition}
Given a singly filtered simplicial set \(X_{\ast} \in \Filt_{[q_1]}\) with height function \(h\) and \(i \in \left[\frac{q_1-q_2}{2}\right]\), we have \(A_{i,!}(X_{\ast})\) has the same underlying simplicial set as \(X_{\ast}\) and has height function \(\tilde{h}\) given by
\[\tilde{h}(\sigma)= (h(\sigma),h(\sigma)-2i).\]
\end{proposition}
\begin{proof}
That the underlying simplicial set is the same relies primarily on \(A_i\) being a co-initial sequence, so all simplices are represented.
The argument for why \(\tilde{h}(\sigma)=A_i(h(\sigma))\) follows similarly to previous proofs such as that of Proposition \ref{prop:forgeth1}, though we will produce some of the argument here.
In particular,
\[A_{i,!}(X_{\ast})_{q_1,q_2}:= \coprod_{\substack{q_1\leq q_3\leq q_4\\ q_2\leq q_3-2i\leq q_4-2i}}X_{q_4}\rightrightarrows \coprod_{\substack{q_1\leq q_3 \\ q_2\leq q_3-2i}} X_{q_3}\]
with the two maps of the coend being the inclusion of \(X_{q_4}\) into \(X_{q_3}\) and identifying \(X_{q_4}\) with itself.
This will lead to a simplex \(\sigma\) appearing in \(A_{i,!}(X_{\ast})_{q_1,q_2}\) precisely when \((q_1,q_2)\leq (h(\sigma),h(\sigma)-2i)\), as needed.
\end{proof}

Finally note that the persistence chain complex \(PC_{\bullet}\) gives a functor from \(\Filt_{[q]}\) to graded \(\Z[U]\) chain complexes with grading in the coset \(q_1+\Z\) with \(U\) having grading \(-2\).
In particular the simplicial chain complex associated to a simplicial set yields a diagram of \(\Z\)-graded \(\Z\)-chain complexes of shape \(U\).
Taking the direct sum of all constituent chain complexes and letting \(U\) act by inclusion yields a \(\Z[U]\)-chain complex with grading the sum of the \(\Z\) grading and the index coming from \([q]\).
Similarly one gets a functor \(PC_{\bullet\bullet}\) from \(\Filt_{[q_1,q_2]}\) to \(\Z[\Uc,\Vc]\)-chain complexes bigraded in gradings \((q_1,q_2)+\Z\) where \(\Uc\) has grading \((-2,0)\) and \(\Vc\) has grading \((0,-2)\).
After taking persistence \(A_i^{\ast}\) becomes equivalent to the functor taking the Alexander grading \(i\) portion of \(PC_{\bullet\bullet}\), \(p_{1,!}\) becomes equivalent to the functor setting \(\Vc=1\) and forgetting the second grading, and \(p_{2,!}\) becomes equivalent to the functor setting \(\Uc=1\) and forgetting the first grading.
These will not be the focus of this paper, but they will be necessary for comparing to previous work and may be useful for those more familiar with Heegaard Floer homology and Knot Floer homology.

Having established our simplicial model categories for filtered simplicial sets, we will now bundle those model categories together to give our quasi-categories.
In particular, given a three-manifold \(Y\) and a first homology class \([K]\) in \(H^1(Y;\Z)\) define
\begin{align*}
\Cc_Y&:= \prod_{\tf\in \SpincX{Y}}\Rf \Filt_{\grf(\tf)}\\
\Cc_{Y,[K]}&:= \prod_{\tf\in\SpincX{Y}}\Rf\Filt_{\grf(\tf),\grf(\tf+\PD[K])}.
\end{align*}
Note that we will be treating \(K\) as if it is a knot, but we emphasize that this category only depends on its first homology class.
We will denote objects in \(\Cc_Y\) generally by \(X\) with \(X_{\ast}^{\tf}\) referring to the filtered space associated to the \Spinc structure \(\tf\) on \(Y\), and similarly for \(X\) in \(\Cc_{Y,[K]}\), \(X_{\ast\ast}^{\tf}\) refers to the doubly-filtered space associated with \(\tf\).

Given two three manifolds equipped with first homology classes \((Y_1,[K_1])\) and \((Y_2,[K_2])\), there exist functors
\begin{align*}
\otimes \colon \Cc_{Y_1}\times \Cc_{Y_2}&\to \Cc_{Y_1\# Y_2} \\
\otimes \colon \Cc_{Y_1,[K_1]}\times \Cc_{Y_2,[K_2]}&\to \Cc_{Y_1\# Y_2,[K_1\#K_2]}
\end{align*}
given by modeling \(\SpincX{Y_1\#Y_2}\) on \(\SpincX{Y_1}\times \SpincX{Y_2}\) and over the \Spinc structure \((\tf_1,\tf_2)\), we can apply the functor \(\otimes\) defined above.

\subsection{A Quasi-Category Modeling the Flip Map}\label{subsec:quasicatFlipMap}

We can define three functors \(p_1\) and \(p_2\) from \(\Cc_{Y,[K]}\) to \(\Cc_Y\), by

\begin{align*}
p_1(X)_{\ast}^{\tf}&:= p_{1,!}(X_{\ast\ast}^{\tf}) \\
p_2(X)_{\ast}^{\tf}&:= p_{2,!}\left(X_{\ast\ast}^{\tf-[K]}\right).
\end{align*}

As such, define \(\Cc_{Y,[K]}^{\Gamma}\) by the following pullback diagram, and \(\Cc_{Y,[K]}^{\Gamma,2}\) as its homotopy pullback in \(\sSet^+_{\Joyal}\), i.e. apply the \((-)^{\sharp}\) adjunction to mark, find the homotopy pullback in \(\sSet^+_{\Joyal}\) then apply the other \((-)^{\sharp}\) to forget the markings.
In addition to being able to define the homotopy pullback in \(\sSet^+_{\Joyal}\) due to that model structure being enriched in \(\sSet_{\Quil}\), the effect of this transfer is that our choices of data in \(\Cc_{Y,[K]},\left[\bou \Delta^1,\Cc_Y\right],\) and \(\left[\Delta^1,\Cc_Y\right]\) do not need to commute with \((p_2,p_1)\) and \(i^{\ast}\), but there do need to be equivalences realizing the relations that would have been imposed in the ordinary pullback.
\[\begin{tikzcd}
\Cc_{Y,[K]}^{\Gamma} \arrow[r] \arrow[d] & \left[\Delta^1,\Cc_Y\right] \arrow{d}{i^{\ast}} \\
\Cc_{Y,[K]} \arrow{r}{(p_2,p_1)} & \left[\bou \Delta^1,\Cc_Y\right].
\end{tikzcd}
\]

\begin{proposition}\label{prop:CykGamma} 
The simplicial sets \(\Cc_{Y,[K]}^{\Gamma}\) and \(\Cc_{Y,[K]}^{\Gamma,2}\) are quasi-categories and weakly equivalent in \(\sSet_{\Joyal}\).
\end{proposition}
\begin{proof}
There is a Reedy structure under which the diagram defining the symplical set \(\Cc_{Y,[K]}^{\Gamma}\) is Reedy fibrant in \(\sSet^+_{\Joyal}\) (Reedy structures are discussed in Section \ref{subsec:modelFunc}).
The Reedy structure in question assigns \(\Cc_{Y,[K]}\) degree 0, \([\bou\Delta^1,\Cc_Y]\) degree 1 and \([\Delta^1,\Cc_Y]\) degree 2.
Since the only nontrivial map in \(R_{-}\) is \(i^{\ast}\) the only condition on matching maps for this diagram to be Reedy fibrant is that \(i^{\ast}\) be a fibration, which it is because it is the pullback hom of the monomorphism \(\bou\Delta^1\to \Delta^1\) and the fibration \(\Cc_Y\to \bullet\).
This guaruntees that the homotopy pullback and the ordinary pullback in \(\sSet^+\) are weakly equivalent and the ordinary pullback is fibrant in \(\sSet^+_{\Joyal}\).
 because \((-)^{\sharp}\) is a right Quillen adjoint and both the homotopy pullbakc and the ordinary pullback are fibrant, the weak equivalence is  preserved by forgetting the markings in passing to \(\sSet_{\Joyal}\)
\end{proof}

Objects in \(\Cc_{Y,[K]}^{\Gamma}\) have the form \((X,\Gamma_X)\), where \(X\) is the projection into \(\Cc_{Y,[K]}\) and  \(\Gamma_X\) is the portion coming from \([\Delta^1,\Cc_Y]\) and provides a collection of filtered maps
\[\Gamma_X^{\tf}\colon \pi_{2,\ast}^l\left(X_{\ast\ast}^{\tf}\right) \to \pi_{1,\ast}^l(X_{\ast\ast}^{\tf+\PD[K]}).\]
However, we will often leave \(\Gamma_X\) implicit and just write \(X\).

\begin{remark}
A similar argument would work if instead of \(\Delta^1\) one chose \(N(\Jc)\) where \(\Jc\) is the walking isomorphism category.
This would force \(\Gamma\) to not only give a filtered map but a filtered homotopy equivalence, which is actually guarunteed in the cases coming from knot lattice spaces.
However, that will not actually be needed for constructing the basic dual knot surgery functor, and working in that category requires tracking more information.
\end{remark}

\begin{remark}
Because Proposition \ref{prop:CykGamma} ensures \(\Cc_{Y,[K]}^{\Gamma}\) and \(\Cc_{Y,[K]}^{\Gamma,2}\) are homotopy equivalent as quasi-categories, one can allow for \(\Gamma^{\tf}\) to be defined using homotopy equivalent models for \(p_{1,!}^{\tf}\) and \(p_{2,!}^{\tf}\).
Not only does this allow one to potentially use simpler models in the calculation, but may be helpful in interpolating with models of the surgery formula that do not place such restrictions.
For example, the \(A_{\infty}\) modules in \cite{ZemkeWhat} includes a place to include what is practice is homotopy equivalent to \(p_{1,!}\) and \(p_{2,!}\) and encodes the maps from the surgery formula that require the input of \(\Gamma\), but the homotopy equivalence is not built into the data, let alone the exact forms of \(p_{1,!}\) and \(p_{2,!}\).
As such, I would not be surprised if, after applying \(PC_{\bullet}\) and converting from the language of quasi-categories to the language of \(A_{\infty}\)-modules that there was a functor from \(\Cc_{Y,[K]}^{\Gamma,2}\) to Zemke's modules.

However, allowing simpler models for individual objects in the calculation increases the complexity of the overall model, since one also has to keep track of how all these models are related. As such, for the purposes of defining our functor, we will focus on \(\Cc_{Y,[K]}^{\Gamma}\). 
\end{remark}

\subsection{A Quasi-Category modeling involutive data}\label{subsec:quasicatInvol}

There will be two main struggles in modeling the involutive information we would like.
First, the maps \(\Ii\) and \(\Jj\) (to be introduced in Sections \ref{subsec:LatticeDef} and \ref{subsec:knotLatticeDef}) that we would like to call involutions are not true endomorphisms due to both shuffling the \Spinc structures, and second \(\Jj\) is skew-filtered rather than doubly-filtered.
Note that while there are alternative methods of dealing with the shuffling of \Spinc structures, the skew filtered maps are more of a fundamental issue, as it requires one to track the action of the functor \(\sigma\).
As such, in order to capture these we cannot simply use diagrams of shape \(N(C_2)\), where \(C_2\) is the cyclic group of order two interpreted as a category with one element, but need to include the action of a functor in our objects.
Second, our desired involutions will naturally live in different categories with \(\Ii\) needing to land in some \(\Filt_{\grf(Y,\tf)}\) and \(\Jj\) needing to land in some \(\Filt_{\grf(Y,\tf)\times \grf(Y,\tf+[K])}\).
While we have already run into this in Section \ref{subsec:quasicatFlipMap} in tracking an object from \(\Cc_{Y,[K]}\) and a map from \(\Cc_{Y}\) the complexity of the relations that need tracking has increased significantly.

For both \(\Ii\) and \(\Jj\) the failure to be an endomorphism can be expressed in terms of an automorphism of the underlying category, one that shuffles the \Spinc structures appropriately and swaps filtrations for \(\Jj\).
These are made explicit in the following automorphisms of \(\Cc_Y\) and \(\Cc_{Y,[K]}\), which are automorphisms because of the comments at the end of Section \ref{subsec:SpecConj} on how \(\grf\) interacts with conjugation.
Here,
\begin{align*}
\Psi_1\colon& \Cc_Y \to \Cc_Y \\
\Psi_2\colon& \Cc_{Y,[K]} \to \Cc_{Y,[K]}\\
\Psi_1(X)_{\ast}^{\tf} &=X_{\ast}^{\overline{\tf}} \\
\Psi_2(X)_{\ast\ast}^{\tf}&=\sigma\left(X_{\ast\ast}^{\overline{\tf+[K]}}\right)
\end{align*}
where \(\sigma\) is the functor defined in Section \ref{subsec:FilteredsSet}.
In interpreting \(\Ii\) and \(\Jj\) as endomorphism-like, we are recognizing that while they is not in \(\Hom(X,X)\), they are in \(\Hom(X,\Psi_i(X))\), which by the action of \(\Psi_i\) is also \(\Hom(\Psi_i(X),X)\).
As such, \(\Ii^2\) can be interpreted as \(\Psi_1(\Ii)\Ii\), which we can require to be homotopic to the identity, and \(\Psi_1\) will take this homotopy to a homotopy relating \(\Ii\Psi_1(\Ii)\) to the identity.

More specifically, let \(\Jc_i\) represent the category with objects \(a\) and and an object \(b_i\) where \(i \in \{1,2\}\) and morphisms \(f_i\colon a\to b_i\) and \(g_i\colon b_i\to a\).
While \(\Jc_1\) and \(\Jc_2\) are naturally isomorphic, we will also be considering the colimit of \(\Jc_1\) and \(\Jc_2\) under the inclusion of \(a\), which is 3 objects with an isomorphism between each pair, and keeping names separate now will be notationally convenient.
Consider then the functor \(\Jc_i\to C_2\) which maps the non-identity maps in \(\Jc_i\) to the non-identity maps in \(C_2\), the group with 2 elements interpreted as a category.
This acts similarly to the to the universal cover of \(C_2\), and in particular, there is a free action of \(C_2\) on \(\Jc_i\) that when \(\Jc_i\) is quotiented by this action we get \(C_2\) again.
As such an equivariant functor on \(\Jc_i\), or on \(N(\Jc_i)\) in the \(\infty\)-categorical setting, to a (quasi-)category with an action by \(C_2\), such as \(\Cc_Y\) with \(\Psi_1\) or \(\Cc_{Y,[K]}\) with \(\Psi_2\), provides one way to model such an almost-involution.

This is additionally motivated by \(\iota\) and \(\iota_K\) in Heegaard Floer homology and knot Floer homology being defined using (first-order) naturality of Heegaard Floer homology and Knot Floer homology as well as the presentation of each exhibiting such a symmetry.
As such, up to care to properly model \(\iota_K^2\) not always squaring to the identity, modeling the involutive maps using symmetric functor seems reasonable, as higher order naturality statements would provide the extension of this functor across the higher order simplices on \(N(\Jc_i)\).
However, to do so we need to make sure the simplicial set of such symmetric functors is indeed a quasi-category.

\begin{proposition}
Let \(G\) be a simplicial group, in interpreted as a simplically enriched category with one object.
Let \(X\) be a simplicial set with a free \(G\) action and \(C\) a quasi-category with a \(G\) action.
Then the simplicial set of equivariant functors \([X,C]_G\) is a quasi-category.
\end{proposition}
\begin{proof}
The model category \(\sSet_{\Joyal}\) is cofibrantly generated and monoidal and thus enriched in itself.
Furthermore all objects are cofibrant in \(\sSet_{\Joyal}\), so in particular the group \(G\) is cofibrant in \(\sSet_{\Joyal}\).
As such, the model structure \([G,\sSet_{\Joyal}]_{\proj}\) exists and is enriched over \(\sSet_{\Joyal}\).
In a cofibrant model category, the cofibrations are retracts of cell-complexes built out of the generating cofibrations.
Because \(X\) has a free \(G\) action, the cell attachments of each new simplex can be done orbit by orbit, with each orbit attachment being a generating cofibration of \([G,\sSet_{\Joyal}]_{\proj}\), and thus \(X\) with this \(G\) action is cofibrant in \([G,\sSet_{\Joyal}]_{\proj}\).
Additionally \(C\) is fibrant in \([G,\sSet_{\Joyal}]_{\proj}\).
By \([G,\sSet_{\Joyal}]_{\proj}\) being enriched in \(\sSet_{\Joyal}\) the hom of a cofibrant object with a fibrant object will be fibrant in \(\sSet_{\Joyal}\)and thus a quasicategory.
\end{proof}

We will now define \(\Cc_{Y,[K]}^{\Jc}\) as the limit of a diagram.
Here \(\Cc_Y\) and \(\Cc_{Y,[K]}\) in the context of quasi-categories with a \(C_2\) action will have the actions of \(\Psi_1\) and \(\Psi_2\) respectively.
Furthermore let \(j_1\) be canonical inclusion of \([N(\Jc_1),\Cc_Y]_{C_2}\) into \([N(\Jc_1),\Cc_Y]\) and \(j_2\) similarly for \([N(\Jc_2),\Cc_{Y,[K]}]_{C_2}\).
Let \(\Jc_i^{\ast}\) represent the pullback of the inclusion of \(N(\Jc_i)\) into \(N(\tilde{\Jc})\), and let \(a^{\ast}\) represent the pullback of the inclusion of \(a\) into \(N(\Jc_i)\).
Then we define \(\Cc_{Y,[K]}^{\Jk}\) as the the limit of the following diagram, and define \(\Cc_{Y,[K]}^{\Jk,2}\) to be the homotopy limit of the same diagram (in \(\sSet^+_{\Joyal}\).

\[
\begin{tikzcd}
& & {[N(\Jc_2),\Cc_{Y,[K]}]_{C_2}} \ar[d,"p_1j_2"] \\
 & {[N(\tilde{\Jc}),\Cc_Y]} \ar[d,"\Jc_1^{\ast}"]  \ar[r,"\Jc_2^{\ast}"] & {[N(\Jc_2),\Cc_Y]} \ar[d,"a^{\ast}"] \\
{[N(\Jc_1),\Cc_Y]_{C_2}} \arrow[r,"j_1"] & {[N(\Jc_1),\Cc_Y]} \arrow[r,"a^{\ast}"] & \Cc_Y
\end{tikzcd}
\]



Define \(\Cc_{Y,[K]}^{\Jk,3}\) and \(\Cc_{Y,[K]}^{\Jk,4}\) as the pullback and homotopy pullback (in \(\sSet_{\Joyal}^+\) of the following diagram
\[
\begin{tikzcd}
 & {[N(\Jc_2),\Cc_{Y,[K]}]_{C_2}} \ar[d,"a^{\ast}p_1j_1"] \\
{[N(\Jc_1),\Cc_Y]_{C_2}} \ar[r,"a^{\ast}j_1"] & \Cc_Y
\end{tikzcd}
\]

\begin{proposition}\label{prop:CJkquasicat}
The simplicial sets \(\Cc_{Y,[K]}^{\Jk},\Cc_{Y,[K]}^{\Jk,2},\Cc_{Y,[K]}^{\Jk,3},\) and \(\Cc_{Y,[K]}^{\Jk,4}\) are all quasicategories and weakly equivalent.
\end{proposition}
\begin{proof}
To see that morphism \(a^{\ast}j_1\) is a fibration, note that it can also be factored as \(k^{\ast}i_1^{\ast}\), where \(i_1\) represents the inclusion of the objects of \(N(\Jc_1)\) into \(N(\Jc_1)\) as simplicial sets with a \(C_2\) action and then the pullback to \(a\) without the \(C_2\) action.
Note that here \(k^{\ast}\) is an isomorphism, since the objects of \(N(\Jc_1)\) can be thought of as \(a\otimes C_2\), and an equivariant functor out of this is the same as a functor from \(a\).
Additionally \(i_1\) is a cofibration in \([C_2,\sSet_{\Joyal}]_{\proj}\) by the same argument that a simplical set with a free action is cofibrant.
As such both \(i_1^{\ast}\) and \(k^{\ast}\) are fibrations.

The maps \(a^{\ast}\) are also both fibrations.
Finally we need the map from the pushout of \(N(\Jc_1)\) and \(N(\Jc_2)\) along \(a\) to include into \(N(\tilde{\Jc})\).
However, this sit true since the simplices of the pushout are functors from \([n]\) to \(\tilde{\Jc}\) that factor through either \(\Jc_1\) or \(\Jc_2\), which naturally inject to functors from \([n]\) to \(\tilde{\Jc}\).

A proof that \(\Cc_{Y,[K]}^{\Jk,3}\) and \(\Cc_{Y,[K]}^{\Jk,4}\) are weakly equivalent quasicategories follows then similarly to Proposition \ref{prop:CykGamma}.
Similarly, there is a Reedy structure under which \(\Cc_{Y,[K]}^{\Jk}\) is described by a Reedy fibrant diagram and thus is a quasi-category and weakly equivalent to \(\Cc_{Y,[K]}^{\Jk}\), in particular, using the Reedy Structure described in Example \ref{ex:Reedy2} and Figure \ref{subfig:Reedy2}, we will show that this diagram is Reedy fibrant.

Finally,we would like to show that \(\Cc_{Y,[K]}^{\Jk}\) and \(\Cc_{Y,[K]}^{\Jk,3}\) are weakly equivalent.
We'll note that because \(N\) is the right adjoint of a Quillen equivalence, that while it doesn't preserve colimits the canonical map from \(\colim NF\) to \(N(\colim F)\) must be a weak equivalence (since it is an equivalence in the homotopy category).
Let \(X\) be the pushout of the \(N(\Jc_i)\) over \(a\).
Then, the map from \(X\) to \(N(\tilde{\Jc})\) is a weak equivalence forcing the corresponding map from \([N(\tilde{\Jc}),\Cc_Y]\) to \([X,\Cc_Y]\) to be a weak equivalence.
This means that the current homotopy limit used to describe \(\Cc_{Y,[K]}^{\Jc,2}\)  is weakly equivalent to limit over a diagram where instead of \([N(\tilde{\Jc}),\Cc_Y]\), you used \([X,\Cc_Y]\).
The same arguments show that this diagram is also Reedy fibrant and thus the homotopy limit over this diagram is weakly equivalent to the limit over this diagram.
However, inclusion of \([X,\Cc_Y]\) in the limit does not record any new information, as \([X,\Cc_Y]\) is the pullback of the objects below it, hence the limit here is isomorphic to the option where that object in the diagram is dropped.
The limit over this diagram in turn is isomorphic to the limit over the diagram used for \(\Cc_{Y,[K]}^{\Jk,3}\).
\end{proof}

\begin{remark}
In practice \(\Cc_{Y,[K]}^{\Jk}\) being weakly equivalent to \(\Cc_{Y,[K]}^{\Jk,2}\) means the functor we describe can be extended over \(\Cc_{Y,[K]}^{\Jk,2}\) and in particular will allow for different portions of the model to be simplified before use so long as one keeps track of all the equivalences relating these simplifications.
Additionally \(\Cc_{Y,[K]}^{\Jk}\) being equivalent to \(\Cc_{Y,[K]}^{\Jk,3}\) may prove useful in reducing the data needed to record an object, and that essentially it suffices to track each involutive structure separately.
However, it will be easier to construct our functors using \(\Cc_{Y,[K]}^{\Jk}\), so that is what we will use here.
It contains enough information to make its lift explicit without tracking too many additional homotopy coherent relations involved in the homotopy limit.
\end{remark}

\begin{proposition}
There is a functor \(\Fc\colon \Cc_{Y,[K]}^{\Jk}\to \Cc_{Y,[K]}^{\Gamma}\).
\end{proposition}
\begin{proof}
We will provide such a functor using the universal property of limits, providing functors from \(\Cc_{Y,[K]}^{\Jk}\) to \(\Cc_{Y,[K]}\) and \([\Delta^1,\Cc_Y]\) that agree on \(\Cc_Y\).
Our functor \(\Fc_1\) from \(\Cc_{Y,[K]}^{\Jk}\) to \(\Cc_{Y,[K]}\) will be the projection to \([N(\Jc_2),\Cc_{Y,[K]}]_{C_2}\) followed by restriction to the object \(a\).
Our functor \(\Fc_2\) from \(\Cc_{Y,[K]}^{\Jk}\) to \([\Delta^1,\Cc_{Y}]\) will be the projection to \([N(\tilde{\Jc}),\Cc_Y]\) followed by restriction to the morphism \(f_1g_2\) in \(N(\tilde{\Jc})\) followed by the automorphism \(\Psi_1\).
The domain of this morphism \(f_1g_2\) is \(b_2\), which in \([N(\Jc_2),\Cc_{Y,[K]}]_{C_2}\) is forced by the symmetry to be \(p_1\Psi_2F_1\), while the codomain is \(b_1\), which in \([N(\Jc_1),\Cc_Y]_{C_2}\) is forced by the symmetry to be \(\Psi_1p_1F_1\).
As such \(\Fc_2\) restricted to the domain and codomain yields \((\Psi_1p_1\Psi_2F_1,\Psi_1^2p_1F_1)=(p_2F_1,p_1F_1)\).
One can check directly that \(\Psi_1p_1\Psi_2= p_2\).

\end{proof}

Given three-manifolds equipped with first homology classes \((Y_1,[K_1])\) and \((Y_2,[K_2])\) there exists a functor 
\[\otimes\colon \Cc^{\Jk}_{Y_1,[K_1]}\times \Cc_{Y_2,[K_2]}^{\Jk}\to \Cc_{Y_1\#Y_2,[K_1\# K_2]}^{\Jk}\]
extending that constructed in Section \ref{subsec:FilteredsSet}.
In particular, each of the homotopy coherent diagrams to \(\Cc_{Y_i}\) and \(\Cc_{Y_i,[K_i]}\) combine to give homotopy coherent diagrams to the product on which we can then postcompose with \(\otimes\).
The main question is if the homotopy coherent diagrams that need to be symmetric are.
However, this comes down to the fact that the actions of conjugation and translated conjugation on \(\Cc_{Y_1\#Y_2}\) and \(\Cc_{Y_1\# Y_2,[K_1\# K_2]}\) are compatible with the decomopsition of \(\SpincX{Y_1\# Y_2}\) as \(\SpincX{Y_1}\times \SpincX{Y_2}\).

\section{Functors for Surgery}\label{sec:InfFunc4Surgery}
Here we build functors for the infinity categories of Section \ref{subsec:quasicatFlipMap} and \ref{subsec:quasicatInvol} that we will verify model surgery.
For this section \(Y\) will be a three-manifold, \(K\) a knot inside it \(\Sigmasq \in \lkc([K],[K])\), and \(\mu\) the dual knot in \(Y_{\Sigmasq}(K)\).
With that data fixed, we will use \(\Ic\) and \(\Ic_{[\tf,i]}\) as shorthands for \(\Ic_{Y,K,\Sigmasq}\) and \(\Ic_{Y,K,\Sigmasq,[\tf,i]}\) to cut down on notational clutter.
Our construction will be broken into the following stages with each construction encoding new information on top of the previous ones.
\begin{align*}
\Xb_{\Sigmasq}\colon &\Cc_{Y,[K]}^{\Gamma} \to \Cc_{Y_{\Sigmasq}(K)} \\
\XK_{\Sigmasq}\colon &\Cc_{Y,[K]}^{\Gamma}\to \Cc_{Y_{\Sigmasq}(K),[\mu]}^{\Gamma} \\
\XKI_{\Sigmasq}\colon &\Cc_{Y,[K]}^{\Jk}\to \Cc_{Y_{\Sigmasq}(K),[\mu]}^{\Jk}.
\end{align*}
The constructions of \(\Xb_{\Sigmasq}\) and \(\XK_{\Sigmasq}\) are in line with known approaches though discussed using the framework of \(\infty\)-categories and done on the level of filtered spaces rather than chain complexes.
The construction of \(\XKI_{\Sigmasq}\) agrees with that of \cite{InvolMappingCone} and \cite{InvolDualKnot}, where those constructions are defined and to the extent that matching would be possible.
See Section \ref{subsec:CompareForm} for more details.

Additionally instead of expressing the end results as mapping cones of a single map, we will express them as homotopy colimits of homotopy coherent diagrams on \(\Ic_{Y,[K],\Sigmasq}\).
This is both because unlike chain complexes or Spectra, \(\sSet_{\Quil}\) is not additive and does not allow us to add maps together to get a single map and because we believe that tracking this structure is more true to what is going on even if collapsing to a single mapping cone were possible.

\subsection{The functor \(\Xb_{\Sigmasq}\)}\label{subsec:Xb}

To construct \(\Xb_{\Sigmasq}\) we will first construct functors
\[\XD_{\Sigmasq}^{[\tf,i]}\colon \Cc_{Y,[K]}^{\Gamma} \to [N(\Ic_{[\tf,i]}),\Rf\Filt_{\grf(Y_{\Sigmasq}(K),[\tf,i])}].\]
To do this we will have to understand not only the functors described in Section \ref{subsec:FilteredsSet} but also how they are related.

\begin{proposition}
Given cosets \([q_1],[q_2]\in \Q/2\Z\), for all \(i \in \frac{[q_1]-[q_2]}{2}\) are natural transformations
\begin{align*}
\eta_1\colon &A_i^{\ast} \to p_{1,!}\\
\eta_2\colon &A_i^{\ast} \to [2i]p_{2,!}
\end{align*}
On filtered simpicial sets this natural transformation is the identity on the underlying simplicial set.
\end{proposition}
\begin{proof}
Unpacking the definition of \(p_{1,!}\) as a left Kan extension and thus as a particular coend, we can see that there is a conanical map from \(X_{q,q-2i}\) to \(p_{1,!}(X_{\ast})_q\) for all \(q\), and these maps are compatible with the corresponding inclusions.
For the cofibrant objects we care about this is realized in particular from the inequality \(\min\{h_1,h_2+2i\}\leq h_1\).
Similarly \(\eta_2\) comes from the inequality \(\min\{h_1,h_2+2i\}\leq h_2+2i\).
\end{proof}

As such we can define
\begin{align*}
\XD_{\Sigmasq}^{[\tf,i]}(X)(\tf,i,a)&=A_i^{\ast}(X^{\tf}_{\ast})[\grf_{\Sigmasq}(i)] \\
\XD_{\Sigmasq}^{[\tf,i]}(X)(\tf,i,b)&=p_{1,!}(X^{\tf}_{\ast})[\grf_{\Sigmasq}(i)]\\
\XD_{\Sigmasq}^{[\tf,i]}(X)(\lambda_{\tf,i})&=\eta_1[\grf_{\Sigmasq}(i)]\\
\XD_{\Sigmasq}^{[\tf,i]}(X)(\rho_{\tf,i})&=\Gamma^{\tf}\circ \eta_2[\grf_{\Sigmasq}(i)].
\end{align*}
To reduce notational clutter \(\XD_{\Sigmasq}^{[\tf,i]}(X)(\tf,i,a)\) is often denoted \(A_{\tf,i}\), and the filtered simplical set \(\XD_{\Sigmasq}^{[\tf,i]}(X)(\tf,i,b)\) is denoted \(B_{\tf,i}\) with the fact that this is functorial in \(X\) and dependent on \(\Sigmasq\) left implicit.
In fact, while \(\Sigmasq\) is necessary for the grading shifts here, that is the only place where \(A_{\tf,i}\) and \(B_{\tf,i}\) depend on \(\Sigmasq\).
See Figure \ref{fig:XDr} for a diagram representing this functor.

There are two things to note with \(\rho_{\tf,i}\).
First, one can check directly from the formula for \(\grf_{\Sigmasq}(i)\) given in Section \ref{subsec:gradShift} that
\[\grf_{\Sigmasq}(i)+2i = \grf_{\Sigmasq}(i+\Sigmasq),\]
as needed for \(\eta_2[\grf_{\Sigmasq}(i)]\) to end at the correct place.
Second, the composition \(\Gamma\circ \eta_2[\grf_{\Sigmasq}(i)]\) is only defined in a quasi-category up to a choice for liftng an inner horn inclusion.
We can either be fine with that ambiguity, knowing that all options are homotopic, or use that there is a canonical composition in the model category.

\begin{figure}
\begin{subfigure}[b]{.9\textwidth}
\[
\begin{tikzcd}
\cdots \ar[rd]&A_{\tf,i} \ar[d,"\eta_1"]\ar[dr,"\Gamma^{\tf}\circ\eta_2" description] & A_{\tf+[K],i+\Sigmasq}\ar[d,"\eta_1"]\ar[dr,"\Gamma^{\tf+[K]}\circ\eta_2" description] & \cdots \ar[d] \\
&B_{\tf,i} &B_{\tf+{[K]},i+\Sigmasq} & B_{\tf+2{[K]},i+2\Sigmasq}
\end{tikzcd}
\]
\caption{The functor \(\XD_{\Sigmasq}^{[\tf,i]}\)}\label{subfig:XDr}
\end{subfigure}
\hfill
\begin{subfigure}[b]{.9\textwidth}
\[
\begin{tikzcd}
\cdots \ar[rd]&A_{\tf,i}^{\mu} \ar[d,"\eta_1"]\ar[dr,"\Gamma^{\tf}\circ\eta_2" description] & A_{\tf+[K],i+\Sigmasq}^{\mu}\ar[d,"\eta_1"]\ar[dr,"\Gamma^{\tf+[K]}\circ\eta_2" description] & \cdots \ar[d] \\
&B_{\tf,i}^{\mu} &B_{\tf+[K],i+\Sigmasq} & B_{\tf+2[K],i+2\Sigmasq}^{\mu}
\end{tikzcd}
\]
\caption{The functor \(\XKD_{\Sigmasq}^{[\tf,i]}\).}\label{subfig:XKDr}
\end{subfigure}
\caption{Diagrams expressing \(\XD_{\Sigmasq}^{[\tf,i]}\) and \(\XKD_{\Sigmasq}^{[\tf,i]}\). Note that the superscript \(\mu\) in Subfigure \ref{subfig:XKDr} represent using the double filtration whose first height function comes from that in Figure \ref{subfig:XDr} and the second height function comes from the same diagram but where \(i\) has been shifted to \(i+1\).
In particular \(p_{1,!}\) applied to Subfigure \ref{subfig:XKDr} gives Subfigure \ref{subfig:XDr}.
We have supressed including the grading shifts on the morphisms to reduce notational clutter.}\label{fig:XDr}
\end{figure}

Due to \(\Ic_{[\tf,i]}\) having all non-trivial morphisms atomic, this suffices to define the desired functor \(\XD^{[\tf,i]}_{\Sigmasq}\), and we can define
\[ \Xb_{\Sigmasq}:=\prod_{[\tf,i]\in \SpincX{Y_{\Sigmasq}(K)}}\hocolim\XD^{[\tf,i]}.\]

\subsection{The functor \(\XK_{\Sigmasq}\)}\label{subsec:XK}

Having established a functor \(\Xb_{\Sigmasq}\) we wish to upgrade this to a functor \(\XK_{\Sigmasq}\).
To do this note the symmetry \(\Xi\) of \(\Ic\) described in Section \ref{subsec:Ic} and consider \(\Xi^{\ast}\XD_{\Sigmasq}^{[\tf,i]}\), which now maps \(N(\Ic_{[\tf,i-1]})\) to \(\Rf\Filt_{\grf(Y_{\Sigmasq}(K),[\tf,i])}\)
Further note that on the level of underlying simplicial sets the diagrams \(\Xi^{\ast}\XD_{\Sigmasq}^{[\tf,i]}\) and \(\XD^{[\tf,i-1]}_{\Sigmasq}\) are the same.
In particular every object that is indexed by \(\tf\) has underlying simplicial set the same as \(X_{\ast\ast}^{\tf}\) and every \(\lambda_{\tf,i}\) is the identity on this underlying simplicial set, while every \(\rho_{\tf,i}\) acts as \(\Gamma^{\tf}\).
As such, together \(\XD_{\Sigmasq}^{[\tf,i-1]}\) and \(\Xi_{\ast}\XD_{\Sigmasq}^{[\tf,i]}\) form a single functor 
\[\XKD^{[\tf,i-1]}_{\Sigmasq}\colon \Cc_{Y,[K]}^{\Gamma} \to [N(\Ic_{[\tf,i-1]}),\Rf\Filt_{\grf(Y_{\Sigmasq}(K),[\tf,i-1])\times\grf(Y_{\Sigmasq}(K),[\tf,i])}].\]
In particular, we can define \(A_{\tf,i}^{\mu}\) to be the simplicial set with double filtration coming from \(A_{\tf,i}\) and \(A_{\tf,i+1}\), and let \(B_{\tf,i}^{\mu}\) be the simplicial set with double filtration coming from \(B_{\tf,i}\) and \(B_{\tf,i+1}\).
A diagram expressing the form of \(\XKD^{[\tf,i]}\) is shown in Figure \ref{fig:XDr}.

Now define,
\[\XK^1_{\Sigmasq}:= \prod_{[\tf,i]\in \SpincX{Y_{\Sigmasq}(K)}}\hocolim\XKD^{[\tf,i]},\]
which will provide a functor \(\XK^1_{\Sigmasq}\) to \(\Cc_{Y_{\Sigmasq}(K),[\mu]}\).
Furthermore,  there is a built in natural transformation \(\tilde{\Gamma}_{\Ic}^{[\tf,i]}\) from \(p_2\XKD^{[\tf,i-1]}\) to \(\Xi^{\ast}p_1\XKD^{[\tf,i]}\), as these are in fact the same functor.
Additionally there is an identification \(\xi\) of the filtered simplicial set \(\hocolim \Xi^{\ast}p_1\XKD^{[\tf,i]}\) with \(\hocolim p_1\XKD^{[\tf,i]}\) that follows from \(\Xi^{\ast}\) being an automorphism of the underlying diagram.
Together we have a natural transformation \(\tilde{\Gamma}^{[\tf,i]}\colon p_2\XK^1 \to p_1\XK^1\) as needed for our compatible functor
\[\XK_{\Sigmasq}^2\colon \Cc_{Y,[K]}^{\Gamma} \to [\Delta^1,\Cc_{Y_{\Sigmasq}(K)}]\]
Together \(\XK_{\Sigmasq}^1\) and \(\XK_{\Sigmasq}^2\) form the desired functor \(\XK_{\Sigmasq}\).

\subsection{The functor \(\XKI_{\Sigmasq}\)}\label{subsec:XKI}

By precomposing \(\XK_{\Sigmasq}\) with the functor \(\Fc\colon \Cc_{Y,[K]}^{\Jk}\to \Cc_{Y,[K]}^{\Gamma}\), we now have a functor from \(\Cc_{Y,[K]}^{\Jk}\) to  \(\Cc_{Y_{\Sigmasq}(K),[\mu]}^{\Gamma}\) that we would like to lift to have codomain \(\Cc_{Y_{\Sigmasq}(K),[\mu]}^{\Jk}\).
The key step of this will be finding a functor \(\XKI^1_{\Sigmasq}\) to \([N(\Jc_1),\Cc_{Y_{\Sigmasq}(K)}]_{C_2}\) compatible with \(\XK_{\Sigmasq}\), as the symmetry used in Section \ref{subsec:XK} will allow us to use this functor as a blueprint for what is needed at \([N(\tilde{\Jc}),\Cc_{Y_{\Sigmasq}(K)}]\) and \([N(\Jc_2),\Cc_{Y_{\Sigmasq}(K),[\mu]}]_{C_2}\).
To build \(\XKI_{\Sigmasq}^1\) up we will start by relating pieces of the diagrams \(\XD_{\Sigmasq}\) then show those pieces can be stitched together.
In this process we will be using the following notation for the canonical projections from \(\Cc_{Y,[K]}^{\Jk}\)
\begin{align*}
\pi_{1}\colon &\Cc_{Y,[K]}^{\Jk} \to [N(\Jc_1),\Cc_Y]_{C_2} \\
\pi_{\tilde{\Jc}}\colon&  \Cc_{Y,[K]}^{\Jk} \to [N(\tilde{\Jc}),\Cc_Y] \\
\pi_{2} \colon & \Cc_{Y,[K]}^{\Jk} \to [N(\Jc_2),\Cc_{Y,[K]}]_{C_2},
\end{align*}
and when we want to postcompose these with the projection to a specific \Spinc structure in \(\Cc_Y\) or \(\Cc_{Y,[K]}\) we will superscript with that \Spinc structure.
For example \(\pi_1^{\tf}\) has codomain \([N(\Jc_1),\Filt_{\grf(Y,\tf)}]\).

\begin{figure}
\begin{subfigure}[b]{.9\textwidth}
\[
\begin{tikzcd}
{A_i^{\ast}(X_{\ast\ast}^{\tf})[\grf_{\Sigmasq}(i)]} \ar[rr,"{\eta_1[\grf_{\Sigmasq}(i)]}"] \ar[d,leftrightarrow,"A_i^{\ast}\pi_2"] && {p_{1,!}(X_{\ast\ast}^{\tf})[\grf_{\Sigmasq}(i)]} \ar[d,leftrightarrow,"\pi_1"] \\
{A_{-i}^{\ast}\left(X_{\ast\ast}^{\overline{\tf+[K]}}\right)[\grf_{\Sigmasq}(-i)] }\ar[rr,"{\Gamma^{\overline{\tf}}\circ\eta_2[\grf_{\Sigmasq}(-i)]}"] & &{p_{2,!}(X_{\ast\ast}^{\overline{\tf}})[\grf_{\Sigmasq}(\Sigmasq-i)]}
\end{tikzcd}
\]
\caption{Diagram to make homotopy coherent.}\label{subfig:hcFillIn}
\end{subfigure}
\hfill

\begin{subfigure}[b]{.9\textwidth}
\[
\begin{tikzcd}
{A_i^{\ast}(X_{\ast\ast}^{\tf})} \ar[r,"{\eta_1}"] \ar[d,leftrightarrow,"A_i^{\ast}\pi_2 "] & p_{1,!}(X_{\ast\ast}^{\tf})\ar[r,"\id"]\ar[d,leftrightarrow,"p_{1,!}\pi_2^{\tf}"]& {p_{1,!}(X_{\ast\ast}^{\tf})} \ar[d,leftrightarrow,"\pi_1^{\tf}"] \\
{A_{i}^{\ast}\left(\Psi_2X_{\ast\ast}^{\overline{\tf+[K]}}\right) }\ar[r,"{\eta_2}"] & p_{2,!}\left(X_{\ast\ast}^{\overline{\tf}}\right)\ar[r,"\Gamma^{\overline{\tf}}"]&{p_{2,!}\left(X_{\ast\ast}^{\overline{\tf}}\right)}
\end{tikzcd}
\]
\caption{A useful factorization.}\label{subfig:hcFillFactor}
\end{subfigure}
\caption{Subfigure \ref{subfig:hcFillIn} depicts the diagram Lemma \ref{lem:XKIpiece} ensures can be filled in with homotopy coherent data with the grading shifts used in \(\Xb_{\Sigmasq}\). Note that \(A_i^{\ast}\pi_2\) and \(\pi_1\) indicate that the are modeling the entire homotopy coherent diagrams therein and not just a specific choice of morphism. That the lower left corner agrees with \(A_i^{\ast}\Psi_2(X)^{\tf}[\grf_{\Sigmasq}(i)]\) is confirmed in Lemma \ref{lem:AiPsi2check}. Subfigure \ref{subfig:hcFillFactor} depicts a factorization of this diagram (with grading shifts supressed to reduce notational clutter).}\label{fig:hcFillIn}
\end{figure}

\begin{lemma}\label{lem:XKIpiece}
For every \((\tf,i)\in \Ac(Y,K)\) there exists a functor
\[\XKID^{\tf,i}_{\Sigmasq}\colon \Cc_{Y,[K]}^{\Jk}\to [N(\Jc_1)\times \Delta^1,\Rf\Filt_{\grf(Y_{\Sigmasq}(K),[\tf,i])}]\]
filling in the higher homotopy coherent relations of Figure \ref{subfig:hcFillIn}.
Precisely,
\begin{enumerate}
\item restricting to \(N(a) \times \Delta^1\) yields the same result as restricting \(\XD_{\Sigmasq}^{[\tf,i]}\) to \(\lambda_{\tf,i}\).
\item restricting to \(N(b)\times \Delta^1\)  yields the same result as restricting \(\XD_{\Sigmasq}^{[\overline{\tf},\Sigmasq-i]}\) to \(\rho_{\overline{\tf+[K]},-i}\).
\item restricting to \(N(\Jc_1) \times \{0\}\) yields the same result as \(A_i^{\ast}\pi_{2}^{\tf}[\grf_{\Sigmasq}(i)]\) 
\item restricting to \(N(\Jc_1)\times\{1\}\) yields the same result as \(\pi_{1}^{\tf}[\grf_{\Sigmasq}(i)]\).
\end{enumerate}
\end{lemma}
\begin{proof}
Note that the left square of Figure \ref{subfig:hcFillFactor} is composed entirely of maps that are the identity on the underlying simplicial sets.
As such any filtered simplicial maps and homotopies we use to fill in the right square, i.e. where restricting to \(N(\Jc_1)\times \{0\}\) yields \(p_{1,!}\pi_2^{\tf}\), will still be filtered if we change \(N(\Jc_1)\times \{0\}\) to yield \(A_i^{\ast}\pi_2^{\tf}\).
There also exists a functor \(\phi\colon \Jc_1\times [1]\to \tilde{\Jc}\) which sends
\begin{enumerate}
\item \(\Jc_1\times \{0\}\) isomorphically to \(\Jc_2\) in \(\tilde{\Jc}\) (with \(a\) mapped to \(a\))
\item \(\Jc_1\times \{1\}\) isomorphically to \(\Jc_1\) in \(\tilde{\Jc_1}\) (with \(a\) mapped to \(a\)), \item \(a\times [1]\) to \(\id\)
\item \(b_1\times \{1\}\) to \(f_1g_2\).
\end{enumerate}
Note that by these conditions \(N(\phi)^{\ast}\pi_{\tilde{\Jc}}^{\tf}\) will  satisfy the needed properties.
\end{proof}

\begin{lemma}\label{lem:AiPsi2check}
As functors on \(\Cc_{Y,[K]}\),
\[(A_i^{\ast}\Psi_2(X)^{\tf})[\grf_{\Sigmasq}(i)]=A_{-i}^{\ast}\left(X_{\ast\ast}^{\overline{\tf+[K]}}\right)[\grf_{\Sigmasq}(-i)]\]
\end{lemma}
\begin{proof}
The functor \(\Psi_2\) shuffles the \Spinc structures in the indicated way, and then it acts by swapping the order the height functions.
As such, if \(X_{\ast\ast}^{\overline{\tf+[K]}}\) has height functions \(h_1\) and \(h_2\) then \((A_i^{\ast}\Psi_2(X)^{\tf})[\grf_r(i)]\) has height function
\[h= \min\{h_2,h_1+2i\}+\grf_{\Sigmasq}(i),\]
while \(A_{-i}^{\ast}\left(X_{\ast\ast}^{\overline{\tf+[K]}}\right)[\grf_{\Sigmasq}(-i)]\) has height function
\begin{align*}
h'&= \min\{h_1,h_2-2i\} +\grf_{\Sigmasq}(-i)\\
&=\min\{h_2,h_1+2i\}-2i +\grf_{\Sigmasq}(-i).
\end{align*}
Moreover, 
\begin{align*}
\grf_{\Sigmasq}(i)-\grf_{\Sigmasq}(-i) &= \frac{(2i-\Sigmasq)^2+\Sigmasq}{4\Sigmasq}-\frac{(-2i-\Sigmasq)^2+\Sigmasq}{4\Sigmasq} \\
&=\frac{(2i-\Sigmasq)^2-(-2i-\Sigmasq)^2}{4\Sigmasq}\\
&=\frac{-4i\Sigmasq-4i\Sigmasq}{4\Sigmasq}\\
&=-2i,
\end{align*}
as needed.
\end{proof}

Given \(\Tc\) a subset of \(\SpincX{Y_{\Sigmasq}(K)}\) define
\[\Ic_{\Tc} := \coprod_{[\tf,i] \in \Tc}\Ic_{Y,[K],\Sigmasq,[\tf,i]}.\]
We will be particularly interested in the case where \(\Tc\) is an orbit under conjugation or under \([\mu]\)-translated conujugation.
For \(\Tc\) an orbit under conjugation, \(N(\Ic_{\Tc})\) has a \(C_2\) action \(N(\psi)\) induced by the automorphism \(\psi\) discussed in Section \ref{subsec:SpecConj}.
For every element \([\tf,i]\) of \(\Tc\), \(\grf(Y_{\Sigmasq}(K),[\tf,i])\) is the same value which we will refer to as \([q]\) for compactness. 
Furthermore, taking homotopy colimits over \(\Ic_{\Tc}\) gives a functor
\[[N(\Ic_{\Tc})\times N(\Jc_1), \Rf\Filt_{[q]}]_{C_2}\to \left[N(\Jc_1),\prod_{[\tf,i]\in \Tc}\Rf\Filt_{[q]}\right]_{C_2},\]
where \(C_2\) acts trivially on \(\Filt_{[q]}\) but shuffles the product as in \(\Psi_1\).

\begin{lemma}\label{lem:XKID1}
Let \(\Tc\) be an orbit of \(\SpincX{Y_{\Sigmasq}(K)}\) under conjugation and let \([q]=\grf(Y_{\Sigmasq}(K),[\tf,i])\) for \([\tf,i]\in \Tc\).
There exists a functor from \(\Cc_{Y,[K]}^{\Jk}\) to \([N(\Ic_{\Tc})\times N(\Jc_1),\Filt_{[q]}]_{C_2}\) so that for \([\tf,i]\in \Tc\) restriction to \(N(\Ic_{[\tf,i]})\times \{a\}\) yields \(\XD^{[\tf,i]}\).
\end{lemma}
\begin{proof}
Because \(N(\Ic_{\Tc})\) is 1-skeletal, a functor
\[\XKI_{\Sigmasq}^1\colon\Cc_{Y,[K]}\to[N(\Ic_{\Tc})\times N(\Jc_2),\Rf\Filt_{[q]}]\]
can be constructed by, for every morphism \(f\in \Ic_{\Tc}\), constructing a functor 
\[G_f\colon \Cc_{Y,[K]}^{\Jk}\to[N(\Jc_1)\times N(f),\Rf\Filt_{[q]}]\]
 so that these functors agree on the objects of \(\Ic_{\Tc}\).
We will define \(G_{\lambda_{\tf,i}}\) using \(\XKID_{\Sigmasq}^{\tf,i}\) constructed in Lemma \ref{lem:XKIpiece}, and we will define \(G_{\rho_{\tf,i}}\) using the diagram\(\phi^{\ast}\XKID_{\Sigmasq}^{\overline{\tf+[K]},-i}\), where \(\phi\) represents the action on \(N(\Jc_1)\).

To check compatibility at \((\tf,i,a)\), let \(i_0\) represent the inclusion \(N(\Jc_1)\times\{0\}\) into \(N(\Jc_1)\times\Delta^1\).
Then, by the hypotheses of Lemma \ref{lem:XKIpiece} and Lemma \ref{lem:AiPsi2check} and the symmetry of \(\pi_2\)
\begin{align*}
i_0^{\ast}\XKD_{\Sigmasq}^{\tf,i} &=A_i^{\ast}\pi_{2}^{\tf}[\grf_{\Sigmasq}(i)] \\
i_0^{\ast}\phi^{\ast}\XKD_{\Sigmasq}^{\overline{\tf+[K]},-i}&= \phi^{\ast}[\grf_{\Sigmasq}(-i)]A_{-i}^{\ast}\pi_2^{\overline{\tf+[K]}} \\
&=[\grf_{\Sigmasq}(-i)]A_{-i}^{\ast}\pi_2^{\overline{\tf+[K]}}\Psi_2 \\
&=A_i^{\ast}\pi_2^{\tf}[\grf_{\Sigmasq}(i)],
\end{align*}
as needed.
A similar argument works to check compatibility at \((\tf,i,b)\), so \(\XKID_{\Sigmasq}^1\) is well defined.
The hypotheses of Lemma \ref{lem:XKIpiece} force restriction to \(N(\Ic_{[\tf,i]})\) to yield \(\XD^{[\tf,i]}\).

To show that \(\XKID_{\Sigmasq}^1\) satisfies the needed symmetries, we only need that for a simplex \((\sigma_1,\sigma_2)\) of \(N(\Ic_{\Tc})\times N(\Jc_1)\) that \(\XKID_{\Sigmasq}^1(\sigma_1,\sigma_2)= \XKI_{\Sigmasq}^1(\psi(\sigma_1),\phi(\sigma_2))\).
We can assume WLOG that \(\sigma_1\) is contained in \(N(\lambda_{\tf,i})\) for some \((\tf,i)\) and in particular it suffices to check for \(\lambda_{\tf,i}\).
However, we defined \(\XKID_{\Sigmasq}^1\) on the \(\rho_{\overline{\tf+[K]},-i}\) precisely to enforce this symmetry. 
\end{proof}

\begin{theorem}\label{thm:XKI}
There exists a functor \(\XKI_{\Sigmasq}\colon \Cc_{Y,K}^{\Jk}\to \Cc_{Y_{\Sigmasq}(K),[\mu]}^{\Jk}\), so that 
\[\Fc\XKI_{\Sigmasq} = \XK_{\Sigmasq}\Fc.\]
\end{theorem}
\begin{proof}
Lemma \ref{lem:XKID1} and the observations preceding it, show that applying the functor \(\hocolim \XKID_{\Sigmasq}^1\) will yield a functor \(\XKI_{\Sigmasq}^1\) landing in \([\N(\Jc_1),\Cc_{Y_{\Sigmasq}(K),[\mu]}]_{C_2}\).
We will now focus on creating functors
\begin{align*}
\XKI_{\Sigmasq}^{\tilde{\Jc}}\colon &\Cc_{Y,[K]}^{\Jk}\to [N(\tilde{\Jc}),\Cc_{Y_{\Sigmasq}(K),[\mu]}] \\
\XKI_{\Sigmasq}^2\colon &\Cc_{Y,[K]}^{\Jk}\to [N(\Jc_2),\Cc_{Y_{\Sigmasq}(K),[\mu]}]_{C_2}
\end{align*}
so that \(a^{\ast}\XKI_{\Sigmasq}^2= \XK_{\Sigmasq}^1\) and  \(\Jc_2^{\ast}\XKI_{\Sigmasq}=p_1j_2\XKI_{\Sigmasq}^2\).

Now, define the functor
\begin{align*}
G\colon \Ic\times \tilde{\Jc}\to \Ic\times \Jc_1
\end{align*}
 be the functor that sends \(x\times a\) to \(x\times a\) and \(x\times b_i\) to \(x\times b_1\) for all \(x\in \Ic\), so in particular \(f_1g_2\) and \(f_2g_1\) map to the identity.
We can define
\[\XKID_{\Sigmasq}^{\tilde{\Jc}}:= G^{\ast}(\XKID^1).\]
Define the functor \(\XKI_{\Sigmasq}^{\tilde{\Jc}}\) by taking homotopy colimits over the \(N(\Ic)\) factors of \(\XKID_{\Sigmasq}^{\tilde{\Jc}}\), but over \(b_2\) we can use the symmetry of \(\Xi\) and the natural transfomation \(\xi\) to identify this homotopy colimit with \(p_2\XK_{\Sigmasq}\) (essentially using \(\Xi\) and \(\xi\) to modify the homotopy colimit cone).
This in particular forces \(\Jc_1^{\ast}\XKI_{\Sigmasq}^{\tilde{\Jc}}= j_1^{\ast}\XKI^1\), and restricting to \(f_1g_2\) yields \(\tilde{\Gamma}\) as constructed in Section \ref{subsec:XK}.

What remains to check is if \(\Jc_2^{\ast}\XKI_{\Sigmasq}^{\tilde{\Jc}}\) can be lifted across \(p_1j_1^{\ast}\) to get the desired functor \(\XKI_{\Sigmasq}^2\).
Note that as defined \(\Jc_2^{\ast}\XKI_{\Sigmasq}^{\tilde{\Jc}}\) is essentially \(\XKI_{\Sigmasq}^1\) but with everything precomposed with the identification \(\tilde{\Gamma}\) or equivalently postcomposed with the identification \(\tilde{\Gamma}^{-1}\).
However, by our construction \(\tilde{\Gamma}\) is filtered when going from \(h_2\) to \(h_1\), while \(\tilde{\Gamma}^{-1}\) is filtered when going from \(h_1\) to \(h_2\), so together we get that the resulting maps an homotopy coherent relations are doubly filtered from \(\XK_{\Sigmasq}\) to \(\sigma^{\ast}\XK_{\Sigmasq}\) as needed.
\end{proof}

\section{Confirming the surgery formula for knot lattice homotopy}\label{sec:Lattice}

We now introduce the lattice simplicial set and the knot lattice simplicial set. 
In order to produce a simplicial set that will work well with the surgery formula, our description will be similar to the construction of \(\hocolim D_{G,\emptyset}\) discussed in \cite{knotLatticeInvariance}.

\subsection{The Lattice Simplicial Set}\label{subsec:LatticeDef}
Let \(G\) be a graph with integer weights on the vertices.
Associated to this graph we can through the plumbing construction associate a compact four-manifold \(X_{G}\) with intersection form given by the adjaceny matrix for \(G\) with the integer weights on the diagonal.
We will denote the boundary of \(X_G\) by \(Y_G\).
In particular the vertices of \(G\) each represent disk bundles over spheres that have been glued together to form \(X_G\) and the homology classes of the zero sections of these disk bundles form a basis for \(H_2(X_G;\Z)\).
We will be conflating a vertex of \(G\) with the homology class of \(H_2(X_G;\Z)\) that the corresponding zero section represents.

Let \(\Char(G)\) be the characteristic cohomology classes of \(X_G\), i.e. elements \(K\) of \(H^2(X_G;\Z)\) so that for all vertices \(v\), \(K(v)\cong v^2\, \mbox{mod}\, 2\).
The second relative homology \(H_2(X_G,Y_G;\Z)\) and thus \(H_2(X_G;\Z)\) acts on \(\Char(G)\) by \(K+[S]:= K+2\PD[S]\).
For plumbings, the first chern class \(c_1\colon \SpincX{X_G}\to H^2(X_G;\Z)\) provides a bijection between \(\SpincX{X_G}\) and \(\Char(G)\), and the orbits of \(\Char(G)\) under the action of \(H_2(X_G;\Z)\) form a bijection with \(\SpincX{Y_G}\).
Given \(\tf \in \SpincX{Y_G}\), define \(\Char(G,\tf)\) as the orbit of \(\Char(G)\) associated to \(\tf\).

We will now be assuming that \(G\) is a forest with weighted vertices, and that the intersection form on \(X_G\) is negative definite, and that \(\tf \in \SpincX{Y_G}\).
This implies that the three-manifold is a link of a normal complex analytic singularity and a rational homology three-sphere.
We will spend the rest of the section defining the simplicial set \(\CFb^{\nat}(G,\tf)\).

With the preferred basis for \(H_2(X_G;\Z)\) coming from the vertices of \(G\), each \(\Char(G,\tf)\) forms an affine integer lattice, providing a decompositon of \(H^2(X_G;\R)\) into cubes.
In particular, the convex hull of \(\{K + \sum_{v\in I} v\, |\,I\subseteq E\}\) can be represented as a cube \([K,E]\) with \([K,\emptyset]\) as \(K\).
We will refer to the collection of all such cubes as \(\Qc_{G}\) and the cubes associated to a particular \(\tf\) as \(\Qc_{G,\tf}\). 
We can define a height function on \(\Qc_{G,\tf}\), where for \(K\in \Char(G,\tf)\),
\[h_U(K):=\grf(X_G,\sfk)=\frac{K^2+|V|}{4},\]
where \(c_1(\sfk)=K\).
For \(E\neq \emptyset\) we can then define
\[h_U([K,E]):=\min\{h_U(K')\, |\, K'\subseteq [K,E]\}\]
Lattice homology is traditionally defined with these cubes as the filtered cells, providing a filtration on Euclidean space, but it will be useful for use to consider another cell decomposition.
In particular let \(D_G\) be the category with objects \(\Qc_{G}\) and a morphism \([K,E]\) to \([K',E']\) when \([K',E']\) is in the boundary of \([K,E]\), and \(D_{G,\tf}\) the portion of \(D_G\) associated to \(\tf\).
We can then consider a functor \(\CFD\colon D_{G,\tf}\to \Filt_{\grf(Y,\tf)}\) for which \(\CFD([K,E])\) has a single point \(p_{[K,E]}\) with height \(h_U([K,E])\), and define \(\CFbn(G,\tf):=\hocolim \CFD \).
This not only produces an element of \(\Filt_{\grf(Y,\tf)}\) as a filtered simplicial set rather than a filtered simplicial space but also will be more compatible with the arguments we would like to make.

The category \(D_G\) comes equipped with an involution \(\Ii\) induced by sending \(K\) to \(-K\) where 
\[\Ii([K,E])= \left[-K-\sum_{v\in E}v,E\right].\]
Furthermore, the identity \(h_U(K)=h_U(-K)\) extends to 
\[h_U([K,E])=h_U(\Ii([K,E])).\]
This induces an involution, which we will also call \(\Ii\), on \(\sqcup_{\tf\in \SpincX{Y_G}}\CFbn(G,\tf)\) that takes \(\CFbn(G,\tf)\) to \(\CFbn(G,\overline{\tf})\).
We will refer to \(PC_{\bullet}\CFbn(G,\tf)\) as \(\CFb(G\,\tf)\), which is known to be homotopy equivalent to the Heegaard Floer homology of \((Y_G,\tf)\) after taking coefficients in \(\Fb_2[[U]]\) (see \cite{knotLatticeInvariance} for why this presentation is equivalent to the presentation used in \cite{ZemkeLattice}).
Note that this does mean that they agree with coefficients in \(\Fb_2[U]\) but without a preferred isomorphism.

\subsection{Knot Lattice Simplicial Set}\label{subsec:knotLatticeDef}

Now let \(G_{v_0}\) be a graph with all vertices except a vertex \(v_0\) having integral weights, and let \(G\) be \(G_{v_0}-\{v_0\}\).
As before we will assume that \(G_{v_0}\) is a forest, that \(X_G\) has negative definite intersection form, and that \(\tf\in \SpincX{Y_G}\).
This means that the corresponding knot \(K_{G_{v_0}}\) is a weak generalized algebraic knot (see \cite{knotLatticeInvariance} for discussion of strong versus weak generalized algebraic knots).

The addition of \(v_0\) can be seen as taking the boundary connect sum of a collection of fibers on \(G\), one fiber in each disk bundle to which \(v_0\) is adjacent.
The boundary of this fiber produces a knot \(K_{G_{v_0}}\) in \(Y_G\).
Additionally viewing \(v_0\) as a class in \(H_2(X_G,Y_G;\Z)\), we have that \(v_0\) acts on \(\Char(G)\) by \(K+v_0:= K+2\PD[v_0]\).
This then creates an action on \(D_G\) which sends \([K,E]\) to \([K+v_0,E]\)
We can pull back \(h_U\) along this automorphism to define \(h_V([K,E]):= [K+v_0,E]\), which makes \(\CFD\) into a functor
\[\CFKD\colon D_{G,\tf} \to \Filt_{\grf(Y,\tf),\grf(Y,\tf+v_0)}.\]
We will define the knot lattice space \(\CFKbn(G_{v_0},\tf)\) as \(\hocolim D_{G,\tf}\).
Again, knot lattice has traditionally been defined with the cubes \([K,E]\) as the cells, but the structure of a simplicial set here will be more useful for our purposes.

The action of \(v_0\) on \(D_G\) in the definition of \(\CFKbn(G_{v_0},\tf)\) means that we have a canonical choice of map \(\Gamma\colon p_{2,!}\CFKbn(G_{v_0},\tf)\to p_{1,!}\CFKbn(G_{v_0},\tf+v_0)\).
Applying \(\Gamma\) then \(\Ii\) or equivalently \(\Ii\) then \(\Gamma^{-1}\) then yields an involuiton \(\Jj\) of \(D_G\) where
\[\Jj[K,E] :=\left[K-v_0 - \sum_{v\in E}v,E\right].\]
Recalling that \(\sigma\) is the functor the swaps the order of the filtraitions on a doubly filtered simplicial set, this induces a doubly filtered map
\[\Jj\colon \CFKbn(G_{v_0},\tf) \to \sigma(\CFKbn(G_{v_0},\overline{\tf+v_0}).\]
Note that composing \(\sigma(\Jj)\) and \(\Jj\) yields the identity on \(\CFKbn(G_{v_0},\tf)\).
Additionally if \(G_{v_0}^1\) and \(G_{v_0}^2\) are two forests each with only one unweighted vertex, let \(G_{v_0}^1\# G_{v_0}^2\) be the pointed coproduct, i.e. the disjiont union with the two unweighted vertices identified, which corresponds to the connect sum of the corresponding knots \((Y_{G^1},K_{G_{v_0}^1})\) and \((Y_{G^2},K_{G_{v_0}^2})\) .
Then,
\[\CFKbn\left(G_{v_0}^1\# G_{v_0}^2,\tf_1\#\tf_2\right) \cong \CFKbn\left(G_{v_0}^1,\tf_1\right)\otimes\CFKbn\left( G_{v_0}^2,\tf_2\right).\]
We will use \(\CFKIn(G_{v_0})\) to refer to the combined package of the \(\CFKb(G_{v_0},\tf)\) with \(\Ii\) and \(\Jj\) as an object of \(\Cc_{Y_G,[K_{G_{v_0}}]}^{\Jc}\).
Because \(\Ii\) and \(\Jj\) are involutions on the nose for \(\sqcup_{\tf}\CFb(G,\tf)\) and \(\sqcup_{\tf} \CFKbn(G,\tf)\) respectively, they can be canonically extended to compatible choices in the following quasi-catgories
\[[N(\Jc_1),\Cc_Y]_{C_2}, \quad [N(\tilde{\Jc}),\Cc_Y], \quad [N(\Jc_2),\Cc_{Y,[K]}]_{C_2}.\]
Similarly, we will use \(\CFKbn(G_{v_0})\) to represent the combined package of the various \(\CFKb(G_{v_0},\tf)\) and \(\Gamma\) as an object of \(\Cc_{Y_G,[K_{G_{v_0}}]}^{\Gamma}\).
We will refer to the complex \(PC_{\bullet\bullet}\CFKbn(G_{v_0})\) as \(\CFKb(G_{v_0})\), and this is known to agree with knot Floer after taking coefficients in \(\Fb_2[[\Uc,\Vc]]\).
Note that this will mean agreement with coefficients in \(\Fb_2[\Uc,\Vc]/(\Uc,\Vc)\), i.e. the hat version.

\begin{remark}
The author \cite{knotLatticeInvariance} showed that the maps for naturality commuted up to homotopy with the maps \(\Ii\), \(\Jj\) and \(\Gamma\) using an equivalent model as doubly filtered topological spaces.
One can  produce a functor in \([N(\Cc_{\GenAlgKnot}),\Cc_{Y,[K]}^{\Jk}]\) (where \(\Cc_{\GenAlgKnot}\) is defined similarly to \(\Cc_{\GenAlgKnot}^{\Spin^c}\) but not tracking individual \Spinc structures, as that is now handled in the \(\Cc_{Y,{[K]}}^{\Jk}\).
The argument follows exactly as in \cite{knotLatticeInvariance} except using the homotopy coherent natural transformations on \(D_{G_{min}}\) we can instead consider symmetric homotopy coherent natural transformations on \(D_{G_{min}}\times \Jc_i\) and the \(\infty\)-categorical Grothendieck construction.
The appeal to subcontractibility on objects also ensures that whenever a homotopy is desired one exists, and is unique up to further choices of homotopy, thus ensuring that the morphism spaces are contractible for each morphism in \(\Cc_{\GenAlgKnot}\) and hence the produced domain category is equivalent to \(\Cc_{\GenAlgKnot}\) and the resulting functor can be lifted (as an \(\infty\)-functor) to \(\Cc_{\GenAlgKnot}\).
\end{remark}

\subsection{Verifying the Surgery Formula for Knot Lattice Homotopy}\label{sec:Verify}

In \cite{knotsAndLattice}, \OSS produce a surgery formula mimicking that of knot Floer homology for lattice homology in terms of what they call the master complex, which is essentially the knot lattice complex plus the flip map.
This construction can naturally be expressed in terms of a homotopy colimit, and this section will include how to formalize that notion.
Producing the corresponding height function will require us to be able to produce the corresponding gradings, which were omitted from the construction in \cite{knotsAndLattice}.
Furthermore, we will find the action of the dual knot on this complex, allowing us to recover both height functions for the knot lattice type of the dual knot as well as the involutive maps \(\Jj\) and \(\Ii\).

As in section \ref{subsec:knotLatticeDef} let \(G_{v_0}\) represent a negative-definite forest with a single unweighted vertex \(v_0\), and \(G\) will represent \(G_{v_0}\) without \(v_0\).
We will let \(G_{v_0}(n)\) represent the graph that is the same as \(G_{v_0}\) but has assigned the vertex \(v_0\) a weight of \(n\), which we will assume is negative enough to make \(G_{v_0}(n)\) negative definite.
Finally, let \(G_{v_0}(n)_{u_0}\) be the \(G_{v_0}(n)\) with the addition of a new unweighted vertex \(u_0\) adjacent only to \(v_0\).
We will let \(X_G\) and \(X_{G_{v_0}(n)}\) be the plumbings associated to these graphs and \(Y_{G}\) and \(Y_{G_{v_0}(n)}\) be their three manifold boundaries.
Let \(\Sigma\) be defined as in Section \ref{subsec:backTopSpecify}, though note that it can be speficied using just \(H_2(X_G;\Q)\) rather than working with \(H_2(Y,K;\Q)\).
In particular, let \(\Sigma_0\) be the unique element of \(H_2(X_G;\Q)\) such that \(\PD[\Sigma_0]= \PD[v_0]\).
The class \(\Sigma\) be the element of \(H^2(X_{G_{v_0}(n)},\Q)\) defined as \(\Sigma := v_0-\Sigma_0\), and viewing \(v_0\) as an element of \(H_2(X_G,K_{G_{v_0}};\Q)\) lets us reinterpret \(v_0-\Sigma_0\) as existing as a surface entirely in \(X_G\) rather than in the larger \(X_{G_{v_0}(n)}\).
The Poincare dual of this surface is 0 on \(H^2(X_G;\Q)\), so it algebraically has zero intersection with the zero sections of \(X_G\). As such, there is a choice of surface that intersects the zero sections of \(X_G\) geometrically 0 times and can be pushed off to be a surface in \(Y\) with boundary \(K\).
As noted in \cite{knotLatticeInvariance}, one can show \(\Sigma^2= n- \Sigma_0^2\).



First, we will prove that \(\Xb_{\Sigmasq}(\CFKbn(G_{v_0}))\) recovers \(\CFbn(G_{v_0}(n))\). This essentially translates the proof done by \OSS in \cite{knotsAndLattice} into homotopy theoretic terms, as while as keeping explicit track of the height functions/gradings, which were only encoded in a relative sense in the original.
We will be constructing the dual knot information and the involutive information on top of this base since \(\XK_{\Sigmasq}\) and \(\XKI_{\Sigmasq}\) are successive lifts of \(\Xb_{\Sigmasq}\).
In particular, reviewing the proof for why
\[\Xb_{\Sigmasq}(\CFKb(G_{v_0})) \to \sqcup_{\tf \in \SpincX{Y_{G_{v_0}(n)}}}\CFb(G_{v_0}(n),\tf)\]
will give us a continuous isomoprhism will allow us to show that the same ismorphism works once keeping track of the knot and involutive data.

In \cite{knotsAndLattice}, \OSS noted that \(\CFb(G_{v_0}(n))\) satisfies a short exact sequence of chain complexes
\[0 \to \Bb \to \CFb(G_{v_0}(n)) \to \Tb \to 0,\]
where both \(\Bb\) and \(\Tb\) decompose as a direct sum of pieces \(\Bb_i(\tf)\) and \(\Tb_i(\tf)\).
Here, \(\Bb_i(\tf)\) is the subcomplex of \(\CFb(G_{v_0}(n))\) generated by \([L,E]\) such that \(v_0 \notin E\), \(L|_G \in \Char(G,\tf)\), and 
\[\frac{L(\Sigma)+\Sigmasq}{2}=i\]
Equivalently this equation restricts \(L\) to the cobordism \(W_{G_{v_0}}\) from \(X_G\) to \(X_{G_{v_0}(n)}\) reinterprets it as a \Spinc structure then applies the function \(\hat{A}\) defined in Section \ref{subsec:backTopSpecify}.
We will as such refer to this equation as \(\hat{A}(L)=i\) without ambiguity, and for example, \OSS note in \cite{knotsAndLattice} that for \(v\in G\) that if \(\hat{A}(L)=i\) then \(\hat{A}(L+v)=i\).
The complex \(\Tb_i(\tf)\) is similarly defined but one requires \(v_0 \in E\) and furthermore the differential on \(\Tb_i(\tf)\) does not count the portion of the differential coming from the front and back faces of \([L,E]\) in the \(v_0\) direction, i.e. the cubes \([L,E\backslash \{v_0\}]\) and \([L+v_0,E\backslash \{v_0\}]\).

Both the \(\Bb_i(\tf)\) and the \(\Tb_i(\tf)\) can be realized as the complexes associated to filtered simplicial subsets of \(\CFbn(G_{v_0}(n))\).
In particular, the cubes \(\Qc_{G_{v_0}(n)}\) get partitioned between the \(\Bb_i(\tf)\) and \(\Tb_i(\tf)\).
For a given \(\tf' \in \SpincX{Y_{G_{v_0}}(n)}\) this partitioning realizes \(D_{G_{v_0}(n),\tf'}\) as a product \(\tilde{D}_{G,\tf'} \times \Ic_{Y_G,[K_{G_{v_0}}],\Sigmasq,\tf'}\) coming directly from how the cube decompositions of Euclidian space given by a choice of basis can be split into products of cube decompositions by partitioning said basis.
We can first take homotopy colimit of \(\CFD\) across each each part of the partition first to get
\begin{align*}
\Bbn_i(\tf) &:= \hocolim_{\tilde{D}_{G,\tf'}\times (\tf,i,b)} \CFD(G_{v_0}(n)) \\
\Tbn_i(\tf) &:= \hocolim_{\tilde{D}_{G,\tf'}\times (\tf,i,b)} \CFD(G_{v_0}(n)).
\end{align*}
Including the maps \(l_{\tf,i}\) and \(r_{\tf,i}\) induced by the functoriality of the homotopy colimit recovers the diagram in Figure \ref{fig:TBdiagram}.
Using that homotopy colimits commute across products (\cite{HTT}, Lemma 5.5.2.3), taking the homotopy colimit of said diagram yields \(\CFbn(G_{v_0}(n))\).

\begin{figure}

\[
\begin{tikzcd}
\cdots\arrow[rd,"r_{i-\Sigma,\tf-v_0}^{\nat}"]& & \Tbn_{i}(\tf) \arrow[ld,"l_{i,\tf}^{\nat}"] \arrow[rd,"r_{i,\tf}^{\nat}"'] && \Tbn_{i+\Sigma^2}(\tf+v_0) \arrow[ld,"r_{i+\Sigma^2,\tf+v_0}^{\nat}"] \arrow[rd,"l_{i+\Sigma^2,\tf+v_0}^{\nat}"] \\
&\Bbn_{i}(\tf) && \Bb_{i+\Sigma^2}(\tf+v_0) && \cdots
\end{tikzcd}
\]

\begin{caption}
A diagram showing how the \(\Tbn_i(\tf)\) and \(\Bbn_i(\tf)\) fit together in a given \Spinc structure \(\tf'\) on \(Y_{G_{v_0}(n)}\). The homotopy colimit of this diagram provides \(\CFbn(G_{v_0}(n),\tf')\). \label{fig:TBdiagram}
\end{caption}

\end{figure}

The next part of the surgery formula involves identifying the \(\Tbn_i(\tf)\), the \(\Bbn_i(\tf)\), and the maps \(l_{i,\tf}\) and \(r_{i,\tf}\) with aspects from the knot lattice simplicial set in a way that preserves the diagram.
On the level of chain complexes, \(\Tbn_i(\tf)\) gets identified with \(\CFKb(G_{v_0},\tf,i)\), the Alexander \(i\) portion of \(\CFKb(G_{v_0},\tf)\), while \(\Bbn_i(\tf)\) gets identified with \(\CFb(G,\tf) \cong \CFKb_{V^{-1}}(G_{v_0},\tf,i)\).
On the level of simplicial sets these become the \(A_{\tf,i}\) and \(B_{\tf,i}\) making up \(\Xb_{\Sigmasq}\) in Section \ref{subsec:Xb}.

Before confirming the formula, we will discuss how the functors \(A_{\tf,i}\) and \(B_{\tf,i}\) relate to the discussion of \OSS in \cite{knotsAndLattice} as to better build off of their work.
In particular, \OSS work with the Alexander filtration on \(\CFb^{\infty}(G_{v_0},\tf)\), which looks at sublevel sets of the Alexander grading on 
\[\CFKb^{\infty}(G_{v_0},\tf):=\CFb(G,\tf)\otimes \Fb_2[U,U^{-1}].\]
See the discussion in \cite{knotLatticeInvariance} for the relation between this formulation of the Alexander filtration and that in \cite{knotsAndLattice}.
There is a canonical way to interpolate between this approach and that of \(\Fb_2[\Uc,\Vc]\) modules, that maps \(U^j[K,E]\) to \(\Uc^j\Vc^{j-A([K,E])}[K,E]\), and tracking the maslov grading and the grading given by \(|E|\) allows us to reconstruct the height functions and thus the filtrations.

In \cite{knotsAndLattice}, the functor corresponding to \(A_{\tf,i}\) would be the subcomplex
\[S_i(\tf):=\Span\{U^j[K,E]\, |\, j\geq 0 \, A(U^j[K,E])= A([K,E])-j\leq i\}.\]
The functor corresponding to \(p_{2,!}(X_{\ast\ast}^{\tf})[i]\) would be
\[
V_i(\tf) = \Span\{U^j[K,E] \in \CFb^{\nat}(G,\tf) |\, A(U^j[K,E])\leq i\}.
\]
and the functor corresponding to \(B_{\tf,i}\) is represented as \(\CFb(G,\tf)\) itself.
This identification respects \(\eta_1\) and \(\Gamma \circ \eta_2\), though the notation in \cite{knotsAndLattice} has \(\eta_2\) defined to include the composition with \(\Gamma\)

Proposition \ref{prop:surgForm} summarizes the work of \OSS in proving the surgery formula in that context, and explains why it still holds for simplicial sets.

\begin{proposition}\label{prop:surgForm}
For \([\tf,i]\in \SpincX{Y_{G_{v_0}(n)}}\) the diagram  in Figure \ref{fig:TBdiagram} is isomorphic via a natural transforamtion \(\Phi\) to the diagram \(\XD_{\Sigmasq}^{[\tf,i]}(\CFKbn(G_{v_0}))\)and therefore the homotopy colimit of \(\XD_{\Sigmasq}^{[\tf,i]}(\CFKbn(G_{v_0}))\) is isomorphic to \(\CFbn(G_{v_0}(n))\).
\end{proposition}
\begin{proof}
The work in \cite{knotsAndLattice} leading up to Theorem 5.1 confirms this on the algebraic level as as relatively graded \(\Fb[U]\) chain complexes.
What remains is to pull this work to the space level and recover the gradings and thus the height maps
In particular, Proposition 5.5 in \cite{knotsAndLattice} provides the map \(F'\) between \(\Bb_i(\tf)\) and \(B_{i,\tf}\) while Lemma 5.9 provides the map \(F\) between \(\Tb_i(\tf)\) and \(A_{i,\tf}\).
In particular,
\begin{align*}
F'([L,E]) &= [L|_G,E] \\
F([L,E\cup \{v_0\}])&= U^{a_{v_0}[L,E\cup\{v_0\}]}[L|_G,E].
\end{align*}
Both of these extend to isomorphisms between \(\tilde{D}_{G,[\tf,i]} \times (\tf,i,b)\) and \(D_{G,\tf}\) and between \(\tilde{D}_{G,[\tf,i]} \times (\tf,i,a)\) and \(D_{G,\tf}\).
Furthermore commuting with the \(l_{i,\tf}\) and \(\lambda_{i,\tf}\) is baked into these definitions.
The check that \(F'\) and \(F\) commute with the \(r_{i,\tf}\) and \(\rho_{i,\tf}\) essentially occurs directly after Proposition 5.12 in \cite{knotsAndLattice} (which was itself essentially proving that \(\Gamma\) is a filtered isomorphism).
In that proof they have
\begin{align*}
\rho \circ F([L,E \cup \{v_0\}]) &= U^{\heartsuit}[L|_G+v_0, E]\\
&= U^{\heartsuit}[(L+v_0)|_G,E] \\
&= F' \circ r,
\end{align*}
where \(\heartsuit\) is a particular non-negative number.
Using the standard identifications for \(F\) and \(F'\) identity in the middle becomes an equality functors.
Together all the maps \(F\) and \(F'\) form the natural transformation \(\Phi\)

The work of \OSS in \cite{knotsAndLattice} does establish that these maps preserve relative gradings on \(\CFKbn(G_{v_0},\tf)\), since that information is encoded in the \(U\) powers and the Alexander grading.
To establish an absolute grading then, we only need to pin down the grading shift on \(F'\).
Furthermore, we can do so on the vertices \(L\), since both the heights on \([L|_G,E]\) and \([L,E]\) are calculated as the minimum over the heights on the vertices.
Finally on the vertices we can verify that
\begin{align*}
h_U(L)-h_U(F'(L))&= \frac{L^2 +s+1}{4} - \frac{(L|_G)^2 +s}{4}\\
&= \frac{L^2-(L|_G)^2}{4} + \frac{1}{4} \\
&= \frac{L(\Sigma)^2}{\Sigmasq 4} + \frac{1}{4} \\
&= \frac{(2i-\Sigmasq)^2}{\Sigmasq 4} + \frac{1}{4} \\
&= \grf_{\Sigmasq}(i)
\end{align*}
as needed.
\end{proof}

\begin{proposition}\label{prop:surgFormKnot}
The map \(\hocolim \Phi\) extends to an isomoprhism between and \(\CFKbn(G_{v_0}(n)_{u_0})\) and \(\XK_{\Sigmasq}(\CFKbn(G_{v_0}))\).
\end{proposition}
\begin{proof}
Let \([\tf,i]\in \SpincX{Y_{G_{v_0}(n)}}\) and \([L,E]\in \CFb(G_{v_0}(n),[\tf,i])\), so \(\Gamma([L,E])\) is \([L+u_0,E]\).
Note that \(u_0\) is only adjacent to \(v_0\), so the end result is that for \(v\in G\), \((L+u_0)(v)=L(v)\) while \((L+u_0)(v_0)= L(v_0)+2\).
So, \(L|_G= (L+v_0)|_G\) and \(\hat{A}(L+v_0)=\hat{A}(L)+1\).
As such the restriction to \(G\) is the same in both cases, the only difference is if we are interpreting \(F'([L,E]) = [L|_G,E]\) as living in \(B_{\tf,\hat{A}(L)}\) or \(B_{\tf,i+1}\) in the case where \(v_0\notin E\) or \(F([L,E])=[L|_G,E\backslash \{v_0\}]\) as living in \(A_{\tf,\hat{A}(L)}\) or \(A_{\tf,\hat{A}(L)+1}\).
This is precisely the definition of \(\Gamma_D\) in the definition of \(\XK_{\Sigmasq}(\CFKbn(G_{v_0}))\), as needed. 
\end{proof}


Now that we have constructed the filtered simplicial set \(\CFKbn(G_{v_0}(n)_{u_0})\) in terms of the data of the knot lattice simplicial sets \(\CFKbn(G_{v_0})\), we now turn our attention to reconstructing the involution \(\Jj\) on \(\CFbn(G_{v_0}(n))\) using the data of the involutions \(J_{v_0}\) and \(J\) in \(\CFKI(G_{v_0})\).

\begin{proposition}\label{Prop:JonGv0n}
The map \(\hocolim \Phi\) extends to an isomorphism between \(\CFKIn(G_{v_0}(n)_{u_0})\) and \(\XKI_{\Sigmasq}(\CFKIn(G_{v_0}))\). 
\end{proposition}
\begin{proof}
To distinguish between the various involutions at play, we will use \(\Ii_G\) for \(\Ii\) on \(\CFKIn(G_{v_0},\tf)\), \(J_{G_{v_0}}\) for \(\Jj\) on \(\CFKbn(G_{v_0})\) and \(\Ii_{G_{v_0}(n)}\) for \(\Ii\) on \(\CFbn(G_{v_0}(n))\) and \(J_{G_{v_0}(n)_{u_0}}\) for \(\Jj\) on \(\CFKbn(G_{v_0}(n)_{u_0})\).
It suffices to prove the ismorphism extends between \(\CFKbn(G_{v_0}(n)_{u_0})\) equipped with \(\Ii_{G_{v_0}(n)}\) and \(\pi_1\XKI_{\Sigmasq}(\CFKIn(G_{v_0}))\), as we know that on \(\CFKIn(G_{v_0}(n)_{u_0})\) \(J_{G_{v_0}(n)_{u_0}}\) is defined as \(\Ii_{G_{v_0}(n)}\Gamma\) as in the construction of \(\pi_2\XKI_{\Sigmasq}(\CFKIn(G_{v_0}))\).
Showing that 
\begin{align*}
\Ii_{G_{v_0}(n)}(\Bbn_i(\tf))&=\Bbn_{\Sigmasq-i}(\overline{\tf})\\
\Ii_{G_{v_0}(n)}(\Tbn_i(\tf))&= \Tbn_{-i}(\overline{\tf+v_0})
\end{align*}
will establish that \(\Ii_{G_{v_0}}\) can be deconstructed in a similar way as the functor 
\[\pi_1\XKI_{\Sigmasq}(\CFKIn(G_{v_0})).\] 
Notably because we do not need coherence relations, this depends only on how \(\Ii_{G_{v_0}(n)}\) restricts to each \(\Bbn_i(\tf)\) and \(\Tbn_i(\tf)\).
These arrange in diagrams
\[
\begin{tikzcd}
\Tbn_i(\tf) \ar[r,"l_{\tf,i}"] \ar[d,"\Ii_{G_{v_0}(n)}"] & \Bbn_i(\tf) \ar[d,"\Ii_{G_{v_0}(n)}"]\\
\Tbn_{-i}(\overline{\tf+v_0}) \ar[r,"r_{\tf,i}"] & \Bbn_{\Sigmasq-i}(\overline{\tf})
\end{tikzcd}
\]
for which we can check that \(\Phi\) is a natural transformation and thus a natural isomorphism.
These two steps will in fact be rolled into one.

Consider \([L,E] \in \Bbn_i(\tf)\).
Then \(\Ii_{G_{v_0}(n)}[L,E]=[-L-\sum_{v\in E}v,E]\), and we can see directly that when restricting \(L\) to \(G\), we recover the action of \(\Ii_{G}\).
Furthermore, as per the discussions in Section \ref{subsec:SpecConj} \(\hat{A}(-L)= \Sigmasq-\hat{A}(L)\) and action by vertices in \(G\) does not affect \(\hat{A}\).
Now consider \([L,E\cup\{v_0\}]\in \Tbn_i(\tf)\).
Then \(\Ii_{G_{v_0}(n)}[L,E]=[-L-v_0-\sum_{v\in E}v,E]\), which when restricted to \(G\) directly recovers the action of \(J_{G_{v_0}}\). 
Additionally, \(\hat{A}(-L-v_0)=-\hat{A}(L)\), as needed.

Finally, note that \(\Jj\) and \(\Ii\) are both involutions on the nose, and the higher homotopy data on \(\CFKIn(G_{v_0})\) and \(\CFKIn(G_{v_0}(n)_{u_0})\) reflect that.
In terms of functors of quasi-categories we are guaranteed that all the simplicies of degree 2 or higher will be degenerate.
This will be preserved when taking lifts, so \(\hocolim \Phi\) can be lifted to an equivalence in \(\Cc_{Y_{\Sigmasq}(K),[\mu]}^{\Jj}\).
\end{proof}

\subsection{Comparison to Other Surgery Formulas}\label{subsec:CompareForm}
In this section we will be discussing how this surgery formula lines up against various other surgery formulas in the literature.
We handled the comparison to the original surgery formula of \OSS in the previous section as that is the formula from \cite{knotsAndLattice} that we built off of, and what remains is comparing our formula to the surgery formula's recent refinements that track the dual knot and involutive structures.
To do so we will be working with the associated simplical persistence complex \(PC_{\bullet \bullet}\) with coefficients in \(\Fb_2\).
While \(PC_{\bullet \bullet}\) can most direclty be viewed as an \(\Fb_2[\Uc_1,\Vc_1]\) module, it contains the equivalent information to a doubly filtered \(\Fb_2\) module.
As such, the categories \(\Cc_{Y,[K]}^{\Gamma}\) and \(\Cc_{Y,[K]}^{\Jj}\) and the functors \(\Xb_{\Sigmasq}\), \(\XK_{\Sigmasq}\), and \(\XKI_{\Sigmasq}\) all have counterparts in terms of \(\Fb_2[U]\) and \(\Fb_2[\Uc,\Vc]\) modules.


\subsubsection{The Filtered Surgery Formula}

Hedden-Levine \cite{FilteredSurgeryFormula} packages knot Floer so that the generators of \(CFK^{\infty}(Y,K,\tf)\) as a \(\Fb_2\)-vector space have the form \([x,i,j]\) with \(j-i=A(x)\) and \(U\) acting by subtracting 1 to both \(j\) and \(i\).
In interpreting \(CFK(Y,K,\tf)\) as a\(\Fb_2[\Uc,\Vc]\) module this becomes \(\Uc^{-i}\Vc^{-j}x\), where under this interpretation \(CFK^{\infty}(Y,K,\tf)\) is the Alexander grading 0 \(\Fb[\Uc\Vc]\) submodule of the free \(\Fb[\Uc,\Vc,\Uc^{-1},\Vc^{-1}]\) module generated by various \(x\).
In terms of the model for the surgery formula Hedden and Levine come up with two height functions \(\Ic\) and \(\Jc\) of their own on the surgery formula that produces \(CF^{\infty}(Y_{\Sigmasq}(K),\tf)\) (creating filtrations via sublevel sets).
Because \(\Ic\) and \(\Jc\) are tracking how much \(\Uc\) and \(\Vc\) are in an element rather than their first an second Maslov gradings, we instead focus on checking that they compute the same Alexander grading \(\Jc-\Ic = \frac{h_U-h_V}{2}\).
Note that \(\Jc\) and \(\Ic\) are defined on generators (as a \(\Fb_2\) vector space) of the form \([x,i,j]\), while \(h_U\) and \(h_V\) are defined on simplices of the underlying simplicial set (rather than a simplex in an individual layer).
However, because \(\Jc-\Ic\) does not depend on \(i\) and \(j\), it can be seen as depending only on \(x\), which is comparable to a simplex of the underlying simplicial set.

There is an additional difference we need to make note of, namely that their surgery formula uses the left handed meridian of their original knot rather than the right handed meridian, as is our choice.
So, in particular, they calculate the Floer homology of the reverse of our knot.
Now, normally lattice homology is defined assuming all vertices are joined in a right handed way (i.e. the spheres in our plumbing intersect positively), but it is perfectly reasonable (and proofs for invariance carry through) to define a similar filtered simplicial set for the reverse using the action of \(-v_0\) instead of \(v_0\).
In \(\XK_{\Sigmasq}\) the effect is that instead of \(h_2\) being pulled back from \(A_{\tf,i+1}\) to \(A_{\tf,i}\) and from \(B_{\tf,i+1}\) to \(B_{\tf,i}\), we are pulling back \(A_{\tf, i-1}\) to \(A_{\tf,i}\) and \(B_{\tf,i-1}\) to \(B_{\tf,i}\), with a corresponding change in \(\Gamma\).
We will refer to these as \(A_{\tf,i}^{-\mu}\) and \(B_{\tf,i}^{-\mu}\).

To confirm our model and their model are equivalent, we need to check that they compute the same Alexander grading for a simplex \(\sigma\) in a given \(A_{\tf,i}^{-\mu}\) or \(B_{\tf,i}^{-\mu}\) that 
\[A(\sigma):=\Jc([\sigma,i,j])-\Ic([\sigma,i,j]) = \frac{h_U(\sigma)-h_V(\sigma)}{2}.\]
Because these values are preserved by the action of \(\Uc\Vc\), we may assume that \(i=0\) and \(j\) is the original Alexander grading of \(x\) pre-surgery formula.

\begin{proposition}
Let \(X \in \Cc_{Y,[K]}^{\Gamma}\), \(\tf \in \SpincX{Y}\) and \(s\in \frac{\grf(Y,\tf)- \grf(Y,\tf+[K])}{2}\).
Let \(\sigma\) be some simplex in \(A_{\tf,s}^{-\mu}\) or in \(B_{\tf,s}^{-\mu}\) 
Then,
\[ (\Jc-\Ic)(\sigma) = \frac{h_1(\sigma)-h_2(\sigma)}{2}.\]
Here \(h_1\) and \(h_2\) are the two height functions on \(A_{\tf,s}^{-\mu}\) or \(B_{\tf,s}^{-\mu}\).
In comparison with \(A_{\tf,s}^{-\mu}\), \(\Jc\) and \(\Ic\) are respectively Equations (1.5) and (1.6) of \cite{FilteredSurgeryFormula} and in comparison with \(B_{\tf,s}^{-\mu}\), one uses Equations (1.8) and (1.9).
\end{proposition}
\begin{proof}
We will consider the case of \(A_{\tf,s}^{-\mu}\) as the case of \(B_{\tf,s}^{-\mu}\) follows similarly.
Let \(h_U\) and \(h_V\) represent the original height functions on \(X_{\ast\ast}^{\tf}\).
Additionally, while calculating \((\Jc-\Ic)(\sigma)\), we may use
\begin{align*}
\Ic\left(\left[\sigma,0, A(\sigma) \right]\right) &= \max\left\{0,A(\sigma)-s\right\} \\
\Jc\left(\left[\sigma,0, A(\sigma) \right]\right) &= \max\left\{-1, A(\sigma)-s\right\} + \frac{2ds+k-d}{2k}\\
&= \max\left\{0,A(\sigma)-(s-1)\right\} + \frac{2ds-k-d}{2k}
\end{align*}
where \(\Sigmasq := \frac{k}{d}\).
Additionally on \(\XK_{\Sigmasq}\CFKb(G_{v_0})\) we have that the two height functions are
\begin{align*}
h_1(\sigma) &= \min\{h_U(\sigma), h_V(\sigma)+2s\} + \grf_{\frac{k}{d}}(s) \\
h_2(\sigma) &= \min\{h_U(\sigma),h_V(\sigma)+2(s-1)\} + \grf_{\frac{k}{d}}(s-1).
\end{align*}
Before diving into the full calculation, we first have
\begin{align*}
\frac{\grf_{\Sigmasq}(s)- \grf_{\Sigmasq}(s-1)}{2} &= \frac{(2s-\Sigmasq)^2+\Sigmasq}{8\Sigmasq} - \frac{(2s-\Sigmasq-2)^2+\Sigmasq}{8\Sigmasq}\\
&=\frac{4(2s-\Sigmasq) -4}{8\Sigmasq}\\
&= \frac{2sd-k-d}{2k}.
\end{align*}

Together we have,
\begin{align*}
\frac{h_1(\sigma)-h_2(\sigma)}{2}&= \min\left\{\frac{1}{2} h_U(\sigma), \frac{1}{2}h_V(\sigma) +s\right\} + \grf_{\frac{k}{d}}(s)\\
&-\left(\min\left\{\frac{1}{2}h_U(\sigma),\frac{1}{2}h_V(\sigma)+(s-1)\right\} + \grf_{\frac{k}{d}}(s-1)\right) \\
&= \min\left\{0,\frac{h_V(\sigma)-h_U(\sigma)}{2}+s\right\} \\
&- \min\left\{0,\frac{h_V(\sigma)-h_U(\sigma)}{2}+(s-1)\right\} + \frac{2sd-k-d}{2k} \\
&= -\max\left\{0,A(\sigma)-s\right\} \\
&+\max\left\{0, A(\sigma)+(s-1)\right\} + \frac{2sd-k-d}{2k} \\
&=(\Jc- \Ic)(\sigma).
\end{align*}
\end{proof}

\subsubsection{Involutive Surgery Formula- \(\Ii\)}

We now move to compare with the work of Hendricks, Hom, Stoffergen and Zemke \cite{InvolMappingCone,InvolDualKnot}, starting with the surgery formula of page 202 of \cite{InvolMappingCone}.
A key thing to note when comparing their formula to that here is a different convention in how knots are oriented, reversing how \(\PD[K]\) acts in the description of \(\Gamma\) and also \(\Phi_2\) and thus \(\Jj\).
Additionally, the work of \cite{InvolMappingCone} does not provide a thorough lift of higher homotopy coherent structures, especially as some of those homotopy coherent structures would likely require higher order naturality statements than are currently available for Heegaard Floer homology.
Instead, they assume the Heegaard Floer homology of the ambient three-manfold is as simple as possible, which allows for all maps except for the map \(\id +\iota_{\mathbb{A}}\) to be determined up to a contractible choice.
As such, we can focus our attention specifically on confirming that \(\iota_{\mathbb{A}}\) has the correct form.

One should note for comparision that the construction of the involutive chain complex encodes the homotopy involution \(\iota\) on a chain complex \(C\) by creating the mapping cone of \(\id+\iota\colon C\to Q\cdot C\), where \(Q\) is a new formal variable with \(Q^2=0\).
In particular in the mapping cone, applying \(Q\) sends something from the copy of \(C\) in the domain to the copy of \(C\) in the codomain.
The mapping cone thus measures how different \(\iota\) is from \(\id\).
The maps \(D_{\Sigmasq}\colon \mathfrak{A}\to \mathfrak{B}\) provide the information of the diagram \(\XD_{\Sigmasq}\) in a condensed form, with 
\begin{align*}
\mathfrak{A}&:=\bigoplus_{i\in \lkc([K],[K])}A_{\tf,i}\\
\mathfrak{B}&:= \bigoplus_{i\in \lkc([K],[K])}B_{\tf,i}
\end{align*}
and \(D_{\Sigmasq}\) being the sum of the \(\lambda_{\tf,i}\) and \(\rho_{\tf,i}\).
Note that the reason for the product rather than the direct sum in \cite{InvolMappingCone} comes from using power series coefficients rather than polynomial coefficients, which is needed to provide a general formula for both positive and negative surgeries.
However, for negative surgeries the formula can be reduced to the sum instead.
As such, we can view the figure of Theorem 22.6 in \cite{InvolMappingCone} as corresponding to the information contained in Figure \ref{fig:hcFillIn}, where \(\iota_{\mathbb{A}}\) plays the role of the maps from \(A_i^{\ast}\pi_2\), i.e. \(\iota_K\).
Fortunately point (2) of Theorem 22.6 states that \(\iota_{\mathbb{A}}\) does indeed coincide with the knot involution.

The work of \cite{InvolMappingCone} also provides some useful observations to note.
First, the map \(\iota_{\mathbb{B}}\) is defined by composing the flip map with the knot involution, which should in general recover the three-manifold involution \(\iota\) in general.
On \(\CFKI(G_{v_0})\), we do have that \(\Gamma \Jj = \Ii\) even for non-\(L\)-spaces.
Another parallel with lattice homotopy is that while \(\iota_K\) is not always exactly an involution, when one applies \(A_i^{\ast}\), it does become a homotopy involution (Lemma 3.16 of  \cite{InvolMappingCone}).
This provides some hope that the differences between \(\iota_K\) and \(\Jj\) may wash out some in the surgery formula itself.

\subsubsection{Involutive Surgery Formula -\(\Jj\)}

To compare with the dual knot formula of \cite{InvolDualKnot}, we only need to check that their proposed new \(\iota_K\) agrees with our construction.
In particular as noted in Remark 3.5 of \cite{InvolDualKnot}, they have the opposite orientation convention for the dual knot to Hedden and Levine, and thus the same orientation convention as chosen in this paper.
Further, they state that their small model for the dual knot surgery formula corresponds to the Hedden-Levine formula, though the notational conventions and strategy of proof are sufficiently different to obscure the similarities.
Additionally, they assume that they are working with a knot in \(S^3\), though similar results should hold for knots in integer homology (and likely even rational homology) \(L\)-spaces.

Due to significant differences in notational convention between the work of \cite{InvolDualKnot} and both the work here and the work of Hedden-Levine \cite{FilteredSurgeryFormula}, we will first provide confirmation that our model and the model of \cite{InvolDualKnot} agree for \(\XK_{\Sigmasq}\), especially since we will need to know the identification in order to compare choices of the involutive structure.

\begin{lemma}\label{lem:AAgreeHHSZ}
Given \(X\in \Cc_{Y,[K]}^{\Jj}\) with \(Y\) an integer homolog spherey \(L\)-space with single \Spinc structure \(\tf\), then applying the construction of the small model of \(\Ac_s^{\mu}\) Hendricks-Hom-Stoffergen-Zemke to \(PC_{\bullet\bullet}X\) gives an isomorphic chain complex to computing \(PC_{\bullet\bullet}A_{\tf,s}^{\mu}(X)\).
\end{lemma}
\begin{proof}
For ease of notation we will refer to \(PC_{\bullet\bullet}(X^{\tf}_{\ast\ast})\) as \(\Xc\).
First, the description in \cite{InvolMappingCone} of \(\Ac_s^{\mu}\) has as an underlying \(\Fb_2[\Uc_1,\Uc_2]\) module generated by
\[A_{s+1}^{\ast}(\Xc)/\Vc_1 A_{s}^{\ast}(\Xc)\oplus A_s^{\ast}(\Xc)/\Uc_1 A_{s+1}^{\ast}(\Xc).\]
In particular, generators from \(A_{s+1}^{\ast}(\Xc)\) have the form \(\Uc_1^i\sigma\), where \(\sigma\) is a simplex from \(X^{\tf}_{\ast\ast}\) and \(i\geq 0\).
The condition that \(\Uc_1^i\sigma\) is in \(A_{s+1}^{\ast}(\Xc)\) means that
\[A(\Uc_1^i\sigma)= A(\sigma)-i = s+1,\]
and thus \(i = A(\sigma)-s-1\).
The condition that \(i\geq 0\) places the condition on \(\sigma\) that \(A(\sigma)\geq s+1\).
Similarly, a generator of \(A_s^{\ast}(\Xc)/\Uc_1A_{s+1}^{\ast}(\Xc)\) has the form \(\Vc_1^j\sigma\) where \(\sigma\) is a simplex form \(X^{\tf}_{\ast}\) and \(j\geq 0\).
The condition that \(\Vc^j\sigma\) is in \(A_s^{\ast}(\Xc)\) means that
\[A(\Vc^j\sigma) = A(\sigma)+j= s,\]
and thus \(j= s- A(\sigma)\).
The condition that \(j\geq 0\) then places the condition on \(\sigma\) that \(s \geq A(\sigma)\).

Since all underlying simplices \(\sigma\) of \(X^{\tf}_{\ast\ast}\) and thus of \(A_{\tf,s}^{\mu}\) have either \(A(\sigma)\leq s\) or \(A(\sigma)\geq s+1\), there is an ismorphism of underlying \(\Fb_2[\Uc_2,\Vc_2]\) modules
\[ \phi\colon PC_{\bullet\bullet}A_{\tf,s}^{\mu}\to \Ac^{\mu}_s(\Xc).\]
We will now show that this is an isomorphism of \(\Fb_2[\Uc_2,\Vc_2]\) chain complexes.
In particular, we will use Lemma 4.1 of \cite{InvolDualKnot} to confirm the that this map \(\phi\) respects the boundary map.
The first thing to notice is that some \(\sigma_2\) appears as a term in \(\bou \sigma_1\) (with some powers of \(\Uc_1\) and \(\Vc_1\)) in \(X^{\tf}_{\ast\ast}\) if and only if it appears as a term in \(\Ac^{\mu}_s(\Xc)\) (with some powers of \(\Uc_2\) and \(\Vc_2\)).
In the language of doubly filtered simplicial sets, this would show that \(\phi\) at least respects the underlying simplicial sets and what remains to be shown is that it respects the filtrations.
Since \(\Uc_1,\Vc_2,\Uc_2,\Vc_2\) are bigraded variables their appearance in the boudnary formula at least picks up on the difference in bigradings and thus heights between \(\sigma_1\) and \(\sigma_2\).
In this discussion we will let \(h_U\) and \(h_V\) represent the two height functions on \(X_{\ast\ast}^{\tf}\) and \(\tilde{h}_U, \tilde{h}_V\) the two height functions on \(A_{\tf,s}^{\mu}\).

The authors of \cite{InvolMappingCone} in Lemma 4.1 break their cases down into which summand of \(\Ac_s^{\mu}(\Xc)\) that \(\sigma_1\) comes from, i.e. letting \(x=  \Uc_1^{A(\sigma_1)-s-1}\sigma_1\) or \(x= \Vc_1^{s-A(\sigma_1)}\sigma_1\) as apropriate. Then if \(\Uc_1^m\Vc_1^n\sigma_2\) is a summand of \(\bou x\), there are further cases based on the relative powers of \(m\) and \(n\).
This corresponds to saying if \(A(\sigma_1)\geq s+1\) or \(A(\sigma_1)\leq s\) and then if \(A(\sigma_2)\geq s+1\) or \(A(\sigma_2)\leq s\).
For example, if \(A(\sigma_1)\geq s+1\), then
\[s+1=A(\Uc_1^m\Vc_1^n\sigma_2)=A(\sigma_2)-m+n\]
implies \(A(\sigma_2)\geq s+1\) if and only if \(n\leq m\).
Similarly if \(A(\sigma_1)\leq s\) then
\[s = A(\Uc_1^m\Vc_1^n\sigma_2)= A(\sigma_2)-m+n\]
implies \(A(\sigma_2) \leq s\) if and only if \(m\leq n\).

We will cover the case of \(A(\sigma_1)\geq s+1\) and \(A(\sigma_2)\leq s\), which s case 1a of Lemma 4.1.
If \(\bou \Uc_1^{A(\sigma_1)-s-1}\sigma_1 = \Uc_1^m\Vc_1^n\sigma_2\), then we can conclude that
\begin{align*}
\frac{h_U(\sigma_2)-h_U(\sigma_1)}{2}&= m+1+s-A(\sigma_1)\\
&=m+1+s-\frac{h_U(\sigma_1)-h_V(\sigma_1)}{2} \\
m &= \frac{h_U(\sigma_2)-h_V(\sigma_1}{2}-s-1 \\
n&=\frac{h_V(\sigma_2)-h_V(\sigma_1)}{2}.
\end{align*}
In this case Lemma 4.1 of \cite{InvolDualKnot} states that in \(\Ac_s^{\mu}\) that \(\bou x= \Uc_2^{m+1}\Vc_2^m(\Vc_1^{n-m-1}\sigma_2)\),
so in particular that
\begin{align*}
\frac{\tilde{h}_U(\sigma_2)-\tilde{h}_U(\sigma_1)}{2} &= m+1 \\
&= \frac{h_U(\sigma_2)-h_V(\sigma_1)}{2}-s \\
\frac{\tilde{h}_V(\sigma_2)-\tilde{h}_U(\sigma_1)}{2}&= m \\
&= \frac{h_U(\sigma_2)-h_V(\sigma_1)}{2}-s-1.
\end{align*}
We now confirm this calculation.
Before diving in note that the overall grading shift applied to \(A_{\tf,i}^{\mu}\) does not affect differences in heights and thus is omitted here.
Now observe that
\begin{align*}
\frac{\tilde{h}_U(\sigma_2)- \tilde{h}_U(\sigma_1)}{2} &= \frac{1}{2} \left(\min\{h_U(\sigma_2), h_V(\sigma_2)+2s\}  - \min\{h_U(\sigma_1),h_V(\sigma_1)+2s\}\right) \\
&=\min\{0,s-A(\sigma_2)\}+\frac{h_U(\sigma_2)}{2} - \min\{A(\sigma_1), s\} -\frac{h_V(\sigma_1)}{2}\\
&= \frac{h_U(\sigma_2)-h_V(\sigma_1)}{2} - s\\
\frac{\tilde{h}_V(\sigma_2)-\tilde{h}_V(\sigma_1)}{2} &= \frac{1}{2} \min\{h_U(\sigma_2), h_V(\sigma_2)+2s+2\} \\
&-  \frac{1}{2}\min\{h_U(\sigma_1),h_V(\sigma_1)+2s +2\} \\
&= \min\{0,s+1-A(\sigma_2)\}+\frac{h_U(\sigma_2)}{2} \\
&- \min\{A(\sigma_1), s+1\} -\frac{h_V(\sigma_1)}{2}\\
&=\frac{h_U(\sigma_2)-h_V(\sigma_1)}{2} - s-1,
\end{align*}
as needed.
The other cases follow similarly, with the relations of \(A(\sigma_i)\) to \(s\) determining whether to pull out a \(\frac{h_U(\sigma_i)}{2}\) or \(\frac{h_V(\sigma_i)}{2}\) from the minimum, so that the minimum is either 0 (if \(A(\sigma_i)\leq s\)), \(s\) if \(A(\sigma_i)\geq s+1\) and we are dealing with \(\tilde{h}_U\), or \(s+1\) if \(A(\sigma_i)\geq s+1\) and we are dealing with \(\tilde{h}_V\). 
\end{proof}

\begin{lemma}
For \(X \in \Cc_{Y,[K]}^{\Gamma}\) with \(Y\) an integral homology sphere there is natural ismorphism between \(PC_{\bullet\bullet}\XK_{\Sigmasq}(X)\) and the small model of \cite{InvolDualKnot} applied to \(PC_{\bullet\bullet}X\). 
\end{lemma}
\begin{proof}
Let \(\tf\) be the only \Spinc structure on \(Y\).
Lemma \ref{lem:AAgreeHHSZ} gives us an ismorphism up to between \(A_{\tf,s}^{\mu}\) and \(\Ac_s^{\mu}\) up to grading shifts.
These grading shifts are determined by the grading shifts determined by the maps on the cobordism from \(Y\) to \(Y_{\Sigmasq}(K)\) and thus should agree regardless.
Fortunately identifying \(\Bc_s^{\mu}\) with \(B_{\tf,s}\) is more straightforard than that of \(A_{\tf,s}^{\mu}\) and \(\Ac_s^{\mu}\). 
In particular, \(\Bc_s\) is the Alexander grading \(s\) portion of \(PC_{\bullet\bullet}X\otimes \Fb_2[\Uc_1,\Vc_1,\Vc_1^{-1}]\), which can be identified with \(PC_{\bullet\bullet}X\otimes \Fb_2[\Uc_1,\Vc_1]/(\Vc-1)\).
This is directly analogous to the difference between computing \(p_{1,!}\), i.e. only remembering the first height function and setting the second height function to be identically \(\infty\).
By having the underlying \(\Fb_2[\Uc_2,\Vc_2]\) module be given by \((\Bc_s/\Uc_1\Bc_{s+1})[\Uc_2,\Vc_2]\), we have that \(\Bc_s/\Uc_1\Bc_{s+1}\) picks out the generators of \(\Bc_s\) as a \(\Fb_2[U]\) module and says that  those should also be the generators of \(\Bc_s^{\mu}\) as a \(\Fb_2[\Uc_2,\Vc_2]\) module.
In this way we verify that both the preferred generators of \(B_{\tf,s}^{\mu}\) and of \(\Bc_s^{\mu}\) agree with the preferred generators of \(PC_{\bullet\bullet}X^{\tf}_{\ast\ast}\).

Observe \(B_{\tf,s}^{\mu}\) has the powers of \(\Uc_2\) and \(\Vc_2\) in the differential match the powers of \(\Uc_1\) in the differential on \(PC_{\bullet\bullet}X^{\tf}_{\ast\ast}\), as only \(h_U\) remains after applying \(p_{1,!}\) and the different grading shifts applied on \(B_{\tf,s}\) and \(B_{\tf,s+1}\) do not affect differential.
This matches the claim on page 21 of \cite{InvolDualKnot} regarding the differential of  \(\Bc_s^{\mu}\).

Hendricks, Hom, Stoffergen, and Zemke claim their formulas for \(v^{\mu}\) and \(\tilde{v}^{\mu}\) match the inclusions in the paper of Hedden and Levine, which in turn should match our \(\lambda_{\tf,s}\) and \(\rho_{\tf,s}\).
However, we can also check this on generators using Lemma 4.4, focusing on the case of \(v^{\mu}\) as \(\tilde{v}^{\mu}\) is defined similarly.
We will do this similar to the check of Lemma 4.1 in \ref{lem:AAgreeHHSZ} by noting that the \(\Uc_1,\Vc_1,\Uc_2,\Vc_2\) powers track the differencs in height functions between output and input to the morphism, and that grading shifts applied to both output and input cancel out in this cancellation.
For example if \(\Uc_1^nx \in \Ac_{s+1}/\Vc_1\Ac_s\) then \(n =A(x)-s-1\).
Lemma 4.4 of \cite{InvolDualKnot} predicts that \(v^{\mu}(\Uc_1^nx)) = \Uc_2^{n+1}\Vc_2^n(\Vc_1^{-n-1}x)\).
Note that \(\Vc_1^{-n-1}x\) is the generator of \(\Bc_s^{\mu}\) corresponding to \(x\).
Assuming \(x\) represents a simplex in \(X^{\tf}_{\ast\ast}\) the powers of \(\Uc_2\) and \(\Vc_2\) from \(PC_{\bullet\bullet} \XK_{\Sigmasq}X)\) can be computed
\begin{align*}
\frac{h_U(x)- \min\{h_U(x),h_V(x)+2s\}}{2} &= \frac{h_U(x)}{2} - \min \{0, s- A(x)\} - \frac{h_U(x)}{2} \\
&=A(x)-s = n+1 \\
\frac{h_U(x)- \min\{h_U(x),h_V(x)+2s+2\}}{2} &= \frac{h_U(x)}{2} - \min \{0, 1+s- A(x)\} - \frac{h_U(x)}{2} \\
&=A(x)-s-1 = n,
\end{align*}
as needed.

If instead we have \(\Vc^m x\in \Ac_s/\Uc_1\Ac_{s+1}\), then \(A(x)\leq s\), and the corresponding generator of \(\Bc_s\) is \(\Vc^mx\).
Additionally, when we do the above calculation, we will have \(\min\{0,s-A(x)\}=0\) and \(\min\{0,1+s-A(x)\}=0\).
As such the corresponding powers on \(\Uc_2\) and \(\Vc_2\) will both be 0, as indeed predicted by Lemma 4.4 of \cite{InvolDualKnot}.
This confirms that the our identifications between \(\Ac_s^{\mu}\) and \(A_{\tf,s}\) and between \(\Bc_s^{\mu}\) and \(B_{\tf,s}^{\mu}\) respect \(v^{\mu}\). A similar calculation shows that they respect \(\tilde{v}^{\mu}\).
Together this shows that they respect the diagrams and hence respect the homotopy colimits as needed.
\end{proof}

Having established an isomorphism between the small model of the dual knot complex from \cite{InvolDualKnot} and \(PC_{\bullet\bullet}\XK_{\Sigmasq}\), we will now explore how our construction for the new involutive map compares with that of \cite{InvolDualKnot}.
As before, their assumption that they are working with an \(L\)-space ensures that the primary map of interest is between the \(\oplus_iA_{\tf,i}^{\mu}\).
Theorem 5.4 of \cite{InvolDualKnot} has this map given by  \(\iota_K^{\mu}+\Omega^{\mu}\iota_K^{\mu}\).
We will confirm that \(\iota_K^{\mu}\) is defined analogously to how \(\Jj\) is lifted in Theorem \ref{thm:XKI}.
The map on \(\oplus_i B_{\tf,i}^{\mu}\) is \(\iota_K^{\mu}+\Omega^{\mu}\iota_K^{\mu}\) post-composed with the flip map, which up to the factor of \(\Omega^{\mu}\iota_K^{\mu}\), would then agree with the map constructed for \(\XKI_{\Sigmasq}\).

Unfortunately, \(\Omega^{\mu}\iota_K^{\mu}\) is not representable in our category in the same hom space as \(\iota_K\) due to how \(\Omega^{\mu}\) interacts with gradings.
In particular, \(\Omega^{\mu}\) is the transfer of \(\Phi_{K,1}\otimes \Psi_{H,1}\) to the \(\Ac^{\mu}\), but \(\Phi_{K,1}\) and \(\Psi_{H,1}\) are maps defined as taking the derivative with respect to \(\Uc_1\) or \(\Vc_1\) of the differential map.
That means that it is impossible for them to respect the grading given by the dimension of the simplices.
Working stably would allow you to include maps with grading shifts on the level of the dimension of simplices, but that still would not allow us to represent \(\iota_K^{\mu}\) and \(\Omega^{\mu}\iota_K^{\mu}\) in the same hom set to take their sum without passing to a more permissive category.
As such, at this point confirming a match with \(\iota_K^{\mu}\) is the best we could hope for.

\begin{proposition}
Given \(X \in \Cc_{Y,[K]}^{\Jj}\) with \(Y\) an integral homology \(L\)-space with \Spinc structure \(\tf\), then the isomorphisms 
\[\phi_s\colon PC_{\bullet\bullet}A_{\tf,s}^{\mu} \to \Ac_s^{\mu}\]
satisfy \(\phi_{-s-1} \circ\Jj^{\mu} = \iota_K^{\mu} \phi_s\), where \(\Jj^{\mu}\) is the lift of \(\Jj\) constructed in Theorem \ref{thm:XKI}.
\end{proposition}
\begin{proof}
The generators of \(PC_{\bullet\bullet}A_{\tf,s}^{\mu}\) and \(\Ac_s^{\mu}\) can both be associated to the underlying simplices of \(X\), and in the work of \cite{InvolDualKnot} this involves looking at the underlying generators without distinguishing their \(\Uc_1\) and \(\Vc_1\) powers.
Note further that the \(\Uc_2\) and \(\Vc_2\) are determined by differences in the two new maslov gradings and thus if two skew-linear maps that swap maslov gradings act the same on these generators then the \(\Uc_2\) and \(\Vc_2\) powers will also agree.
Lemma 6.1 of \cite{InvolDualKnot} highlights that \(\iota^{\mu}_K\) acts by \(\iota_K\) (and thus \(\Jj\) in our context) on the underlying generators of \(PC_{\bullet\bullet}X\) with most of the work done by tracking powers of \(\Uc_1,\Vc_1,\Uc_2,\Vc_2\).
Further, the codomain of \(\iota_K^{\mu}\) is \(\Ac_{-s-1}^{\mu}\) as needed
Similarly \(\Jj^{\mu}\) is defined in Theorem \ref{thm:XKI} to be \(\Jj\) on \(A_{\tf,s+1}\) precomposed with the identification of underlying simplicial sets of \(A_{\tf,s}\) and \(A_{\tf,s+1}\).
As such, on the level of the underlying simplicies of \(X\), \(\Jj^{\mu}\) acts entirely by \(\Jj\), as needed.
\end{proof}

\section{Almost Rational Knots}\label{sec:ARknots}

In this section we will focus on a specific class of knots, which I will call almost rational knots. 
In particular, an \emph{almost rational graph} \(G\) is a negative definite plumbing graph so that there exists a vertex \(v\) of \(G\), so that by lowering the weight on only \(G\) one can get an rational plumbing, i.e. one with an \(L\)-space boundary.
N\'emethi has shown that in that case the corresponding filtered space is also subcontractible, i.e. is homotopy equivalent to a filtered point \cite{latticeCohomNormSurf}, and more generally the reduction theorem  of \cite{reductionTheorem} shows that the filtered space associated to an almost rational graph is supported on a cube decomposition of \(\R^1\) instead of \(\R^{|V|}\) where \(V\) is the set of vertices of \(G\).

We can define an \emph{almost rational knot} as a knot represented by a graph with unweighted vertex \(G_{v_0}\), so that \(v_0\) is adjacent to a single vertex \(v\) that realizes \(G\) as an almost rational graph.
If we are talking about a graph \(G_{v_0}\) representing an almost rational knot then we will assume that the graph \(G_{v_0}\) can be used to show the knot is almost rational.
Let \(G'_{u_0}\) be the graph \(G\) with the weight on \(v\) removed.
The condition for \(G\) to be almost rational then becomes that \(G'_{u_0}\) has rational negative surgeries, and \(v_0\) is the dual knot to the surgery that gives \(G\).
Section \ref{subsec:negLSpace} will then mimic the proof in \cite{latticeCohomNormSurf} that \(L\)-spaces have subcontractible lattice homotopy to the case of weak generalized algebraic knots with negative \(L\)-space surgeries.
This will show that the reduction to a cube decomposition of \(\R^1\) used to provide computation tools for the lattice spaces of Brieskorn spheres can be done before application of the surgery formula.
In Section \ref{subsec:ARreduc}, we will combine those computational tools with our upgrades to the surgery formula in order to provide computational tools for the knot lattice spaces of the regular fibers of Brieskorn spheres.
Finally, Section \ref{subsec:exampleARKnots} will include example calculations for the trefoil and  the regular fiber of \(\Sigma(2,3,7)\), and a proof of Theorem \ref{thm:Sigma237regFib}

\subsection{Knots with Negative \(L\)-Space Surgeries}\label{subsec:negLSpace}

Note for this section each \(\Char(G,\tf)\) will be endowed with a poset structure coming directly from the fact that it is modeled on \(H_2(X_G;\Z)\) which has a preferred basis coming from the weighted vertices of \(G\).
In particular for \(k_1,k_2\in \Char(G,\tf)\), we say that \(k_1\leq k_2\) if \(k_2= k_1 + \sum_{v\in V}n_vv\) where each \(n_v\) is non-negative.
With this idea, we can then define our preferred representatives for each \(\Char(G,\tf)\).

\begin{definition}
Given \(\tf \in \SpincX{Y_G}\)  define \(k_{\tf} \in \Char(G,\tf)\) by the property that \(k_{\tf}(v_i) \leq -2-v_i^2\) for all vertices \(v_i\) but for every \(v_j\) there exists a \(v_k\) so that \((k_{\tf}-v_j)(v_k)>-2-v_k^2\).
Such a \(k_{\tf}\) always exists and is uniquely determined.
The canonical characteristc cohomology class for a graph \(G\), denoted \(K_{con}\), is defined to have \(K_{con}(v_i)=-v_i^2-2\), and \(\tf_{con}\) to be the \Spinc sturcutre on \(Y_G\) viewed as on orbit of \(\Char(G,\tf)\) that contains \(K_{con}\).
\end{definition}
Note in particular, that \(K_{con}= k_{\tf_{con}}\), so the \(k_{\tf}\) are each the best approximation for \(K_{con}\) in the orbit representing the \Spinc structure \(\tf\).

\begin{lemma}\label{lem:boxReduction}
Let \(G_{v_0}\) be a weak generalized algebraic knot.
Let \(k_1 \geq k_2 \in \Char(G,\tf)\). If there exists sequences of vertices \(\{v_{j(i),1}\}_{i=1}^{\infty} \{v_{j(i),2}\}_{i=1}^{\infty}\) so that in each sequence, each vertex of \(G\) appears infinitely many times and, letting
\begin{align*}
k_{1,n} &:= k_1 + \sum_{i=1}^n v_{j(i),1} \\
k_{2,n} &:= k_2 - \sum_{i=1}^n v_{j(i),2},
\end{align*}
we have both \(h_U\) and \(h_V\) are decreasing along \(k_{1,n}\) and \(k_{2,n}\) then \(\CFKbn(G_{v_0},\tf)\) deformation retracts onto \(R(k_2,k_1)\), the subspace of doubly filtered cubes with vertices \(K\) satisfying \(k_2\leq K\leq k_1\) under the poset structure induced by the action of the basis of vertices.
If only \(k_1\) and \(\{v_{j(i),1}\}\) exist then \(R(-\infty,k_1)\) is a deformation retract an if only \(k_2\) and \(v_{j(i),2}\) exist then \(R(k_1,\infty)\) is a deformation retract.
\end{lemma}
\begin{proof}
The argument follows similarly to the argument given in Theorem 3.2.4 in \cite{latticeCohomNormSurf} that the filtered lattice space can be restricted given the sequence of vertices along which \(h_U\) is decreasing (there the equivalent statement is that the weight function \(w\) is increasing).
\end{proof}

\begin{lemma}\label{lem:EffectiveRestriction}
Let \(G_{v_0}\) be a weak generalized algebraic knot. Then for \(\tf \in \SpincX{Y_{G}}\) the space \(\CFKbn(G_{v_0},\tf)\) has \(R(k_{\tf},\infty)\) as a deformation retract.
\end{lemma}
\begin{proof}
We are guaranteed that there is an infinite sequence \(\{v_{j(i)}\}_{i=0}^{\infty}\) so that \(K_n:= k_{\tf}-\sum_{i=1}^nv_{j(i)}\) has \(h_U\) decreasing by the work in \cite{latticeCohomNormSurf}.
We would like to show that this sequence is also decreasing with respect to \(h_V\).
Note that \(h_U\) decreasing along \(K_n\) is equivalent to
\[K_n(v_{j(n)})-v_{j(n)}^2\geq 0\]
for all \(n \in \Z_{\geq 0}\).
This quantity is only increased by replacing \(K_n\) by \(K_n+v_0\) in the calculation of \(h_V\) guarantees that \(h_V\) is decreasing along this sequence as well
\end{proof}

\begin{definition}
Given a negative definite graph \(G\) \emph{the minimal cycle} \(Z_{min}\) is the minimal (according to the poset structure given by the preferred basis) class in \(H_2(X_G;\Z)\backslash \{0\}\) so that \((Z_{min},v_{i})\leq 0\) for all vertices \(v_i\).
\end{definition}

\begin{lemma}\label{lem:toZmin}
Let \(G_{v_0}\) be a prime weak generalized algebraic knot in an \(L\)-space and let \(\tf_{con}\) be the canonical \Spinc structure on \(Y_{G}\). Then if there is a \(K\in \Char(G,\tf_{con})\) that maximizes \(h_U\) and \(h_V\) in \(\CFKbn(G_{v_0},\tf)\), then there exists a finite sequence of vertices \(v_{j(i)}\) so that letting \(K_n:=K_{con}+\sum_{i=1}^nv_{j(i)}\) the sequences \(h_U(K_n)\) and \(h_V(K_n)\) are non-increasing and the sequence ends with \(K_{con}+Z_{min}\) 
\end{lemma}
\begin{proof}
Let \(K_0\) be the maximizing element of \(h_U\) and \(h_V\) in \(\Char(G,\tf_{con})\). We will first show that we can assume that \(K_0\) can be taken to be \(K_{con}\).
The deformation retract provide by Lemma \ref{lem:EffectiveRestriction} will guarantee a sequence \(\tilde{v}_{j(i)}\) so that defining \(\tilde{K}_n :=K_0+ \sum_{i=0}^n \tilde{v}_{j(i)}\) the sequence starts at \(K_0\) and ends at \(K_{con}\) with both \(h_U\) and \(h_V\) preserved along it.
In particular, while \(K_0\) need not appear on the sequence from Lemma \ref{lem:EffectiveRestriction} tracing out the path \(K_0\) follows on the deformation retract will give a non-decreasing sequence in both \(h_U\) and \(h_V\) and because \(K_0\) maximized both \(h_U\) and \(h_V\) this sequence must then be constant.

Now, given any choice of starting \(v_i\) it can be completed to a sequence \(v_{j(i)}\) so that the partial sums end at \(Z_{min}\), and \(K_n:=K_{con}+\sum_{i=1}^nv_{j(i)}\) will have \(h_U\) decreasing by 2 from \(K_0\) to \(K_1\) and then remaining constant.
Furthermore, letting \(v_1\) be the vertex \(v_0\) is adjacent to (and since it is prime it is adjacent to only one vertex), we have
\begin{align*}
h_V(K_{n+1})-h_V(K_n) &= h_U(K_{n+1}+v_0)-h_U(K_n+v_0) \\
&= (K_n+2\PD[v_0])(v_{j(n)})+v_{j(n)}^2 \\
&= K_n(v_{j(n)}) +v_{j(n)}^2+2\delta_{v_{j(n)}}^{v_1}\\
&=h_U(K_{n+1}) - h_U(K_n) + 2\delta_{v_{j(n)}}^{v_1}
\end{align*}
So \(h_V\) is constant from \(K_0\) to \(K_1\) and non-decreasing afterwards.
However, because  \(h_V\) was already at its maximum value that means that \(h_V\) must be constant along that sequence.
\end{proof}

\begin{proposition}\label{Prop:LSpaceSurgContract}
Let \(G_{v_0}\) be a weak generalized algebraic knot in an \(L\)-space which has negative \(L\)-space surgeries, then \(\CFKbn(G_{v_0},\tf)\) contracts down to a doubly filtered point for all \(\tf \in \SpincX{Y_{G_{v_0}}}\).
\end{proposition}
\begin{proof}
We can assume without loss of generality that \(G_{v_0}\) is prime.
In particular if a composite generalized algebraic knot has negative \(L\)-space surgeries then so do its components (as being a rational graph is inherited by subgraphs)
Furthermore, if each summand of a knot has the homotopy type of doubly filtered point then so does their connect sum.

By lemma \ref{lem:EffectiveRestriction}, it suffices to find for every \(\tf \in \SpincX{Y_G}\) a sequence \(\{v_{j(i)}\}_{i=0}^{\infty}\) so that the partial sums starting at \(k_{\tf}\), i.e. the sequence \(K_{\tf,n}= k_{\tf}+\sum_{i=1}^nv_{j(i)}\) have \(h_U\) and \(h_V\) decreasing.

Lemma \ref{lem:toZmin} then guaruntees us that both \(h_U\) and \(h_V\) are maximized on \(K_{con}\) and we have a sequence of vertices \(\{v_{j(i)}\}_{i=1}^N\) so that the partial sums starting at \(K_{con}\), ending at \(Z_{min}\), and both \(h_U\) and \(h_V\) decrease along the partial sums.
Furthermore, the coefficient of each \(v_i\) in \(Z_{min}\) is nonzero, guaranteeing that each \(v_i\) appears in \(\{v_{j(i)}\}_{i=1}^n\) at least once.
Let \(\{v_{j(i)}\}_{i=1}^{\infty}\) be the extension of this sequence by repeating its terms on a loop.
We will show that this sequence suffices for any choice of \(k_{\tf} \in \SpincX{Y_G}\).

Given \(n \in \Z_{\geq 0}\), let \(m:= \lfloor \frac{n}{N} \rfloor\), i.e.  the number of loops that our original sequence has ben through by the time we get to \(n\).
Furthermore let \(1\leq \tilde{n}\leq N\) be congruent to \(n\) mod \(N\), so \(K_{\tf,n}= k_{\tf,\tilde{n}}+mZ_{min}\) and \(v_{j_{(n+1)}}=v_{j(\widetilde{n+1})}=v_{j(\tilde{n}+1)}\).
By the definition of \(Z_{min}\), one can deduce
\begin{align*}
h_U(K_{\tf,n+1})-h_U(K_{\tf,n}) &\leq h_U(K_{\tf,\tilde{n}+1})-h_U(K_{\tf,\tilde{n}}) \\
h_V(K_{\tf,n+1})-h_V(K_{\tf,n}) &\leq h_V(K_{\tf,\tilde{n}+1})-h_V(K_{\tf,\tilde{n}})
\end{align*}

Furthermore, the definition of \(k_{\tf}\) allows us to conclude that
\begin{align*}
h_U(K_{\tf,n+1})-h_U(K_{\tf,n})&= k_{\tf}(v_{j(n+1)}) + v_{j(n+1)}^2+ \sum_{i=1}^n(v_{j(i)},v_{j(n+1)})\\
&\leq K_{con}(v_{j(n+1)}) + v_{j(n+1)}^2 + \sum_{i=1}^n(v_{j(i)},v_{j(n+1)}) \\
&\leq h_U(K_{con,n+1})-h_U(K_{con,n}).
\end{align*}
Together these facts imply that because \(h_U(K_{con,n})\) decreases from \(K_{con}\) to \(K_{con}+Z_{min}\), that all sequences \(K_{\tf,n}\) are also decreasing and thus, Lemma \ref{lem:boxReduction} can prove the proposition.
\end{proof}

\subsection{The Almost Rational Reduction}\label{subsec:ARreduc}

\begin{definition}
We will say that that \emph{the simplicial integer line} is the 1-skeletal simplicial set that has 0-simplices the integers and nondegenerate 1-simplices given by \([n,n+1]\), so that the 0th face is \(n\) and the 1st face is \(n+1\).
\end{definition}

\begin{proposition}\label{prop:ARknots}
Let \(G_{v_0}(n)_{u_0}\) represent an almost rational knot created with Seifert framing \(\Sigmasq\).  Then for all \(\tf \in \SpincX{Y_G}\), \(\CFKb^{\nat}(G_{v_0}(n)_{u_0},\tf)\) is filtered homotopy equivalent to a filtered simplicial line. 
If in establishing this equivalence we send \(k_{\tf}\) to 0 and send the action by \(v_0\) to the action of 1 under addition, the filtration on the simplicial line then has \(h_1^{\tf}(n)= h_U(k_{\tf})-2\tau(n)\) where \(\tau\) is the function in \cite{NemAR}, and 
\[h_U([n,n+1])= \min\{h_1^{\tf}(n),h_1^{\tf}(n+1)\}.\]
\end{proposition}
\begin{proof}
By Proposition \ref{prop:surgFormKnot} \(\CFKbn(G_{v_0}(n)_{u_0}) \cong \XK_{\Sigmasq}(\CFKbn(G_{u_0}))\).
By the definition of \(G_{v_0}(n)_{u_0}\) being almost rational, \(G_{v_0}\) has rational negative surgeries and thus in particular satisfies the conditions of Proposition \ref{Prop:LSpaceSurgContract}, making \(\CFKbn(G_{v_0},\tf)\) subcontractible (i.e. homotopy equivalent to a doubly-filtered point) for every \Spinc structure \(\tf'\) on \(Y_{G}\).
Let \(X\in \Cc_{Y_{G},[v_0]}^{\Jj}\) have each \(X^{\tf}_{\ast\ast}\) be precisely that doubly filtered point.
Note that this forces a unique choice for the homotopy coherent involutive data.
Since \(\XK_{\Sigmasq}\) is an infinity functor it preserves equivalences and thus for every \(\tf \in \SpincX{Y_G}\), \(\CFKIn(G_{v_0},\tf) \cong \XKI_{\Sigmasq}(X)\).
Because \(A_{\tf,i}\) and \(B_{\tf,i}\) preserve the structure of the underlying simplicial set, we can then see that the underlying simplicial set for \(\XKI_{\Sigmasq}(X)\) in \Spinc structure \(\tf\) is a filtered nerve of \(\Ic_{Y_{G},[v_0],\Sigmasq,\tf}\).
There is then a weak equivalence from the simplicial integer line to the fibrant replacement of \(N(\Ic_{Y_{G},[v_0],\Sigmasq,\tf})\).
This sends \(n\) to the \([\tf',i,b]\) so that \(k_{\tf}+nu\) is in \(\Bb_{\tf',i}\). 
While this skips over the \(A_{\tf',i}\) the height of that vertex and the edges out of it are the minimum of the heights on \(B_{\tf',i}\) and \(B_{\tf'+v_0,i+\Sigmasq}\).

The connection to \(\tau\) then comes from description of \(\tau\) as starting from \(k_{\tf}\) then alternating between incrementing \(v_0\) up by 1 and tracing down to the minimal values of \(\chi_{k_{\tf}}\).
This allows the computation of \(\tau\) to pick out the elements of \(\Char(G_{v_0}(n),\tf')\) that minimize \(\chi_{k_{\tf}}\) in each \(B_{\tf,i}\).
To convert statements about \(\chi_{k_{\tf}}\) to statements about height functions see the comments in \cite{knotLatticeInvariance}.
\end{proof}

The second filtration \(h_2\) could be tracked by understanding better how the action by \(\Gamma\) looks like on the \((\tf,i)\) and in particular with reference to the preferred representatives \(k_{\tf}\) for each \Spinc structure on \(G_{v_0}(n)\).
This is more straightforward when the almost rational knot in question is null homologous, as the preferred characteristic cohomology class does not change upon application of the action of \(\Gamma\).

\begin{proposition}\label{prop:ARNullHomolgous}
If \(G_{v_0}(n)_{u_0}\) is a null homologous almost rational knot created with a surgery of Seifert framing \(\Sigmasq\) and \(\tf\) a \Spinc structure on \(Y_{G_{v_0}(n)}\).
Given the filtered line model for \(\CFKbn(G_{v_0}(n)_{u_0},\tf)\) then
\begin{align*}
h_2(n)&= h_1\left(n + \frac{1}{\Sigmasq}\right)\\
\Gamma(n)&= n+\frac{1}{\Sigmasq},
\end{align*}
with the filtration and action on edges defined similarly.
\end{proposition}
\begin{proof}
Let \(v_0\) have degree \(d\) in \(Y_G\).
Then given \((\tf',i)\in \Ac(Y_{G},v_0)\), each of which represents some vertex on the filtered line, the next time that a vertex on the filtered line restricts to \(\tf'\) will be \((\tf'+du_0,i+d\Sigmasq)\).
Note then that by the \(i\) coming from a coset of \(\Z\) in \(\Q\), \(d\Sigmasq\) must itself be an integer.
In fact, the set \(\{i\, |\, [\tf',i]=\tf\}\) will be a coset of \(d\Sigmasq\Z\) in \(\Q\).
If \(G_{v_0}(n)_{u_0}\) is null homologous as a knot then \([\tf',i]=[\tf',i+1]\) and as such \(d\Sigmasq =-1\) (the sign determined by \(\Sigmasq\) being negative).
As such, moving \(d\) to the right on the filtered line model for \(\CFKbn(G_{v_0}(n)_{u_0},\tf)\) moves from \((\tf',i)\) to \((\tf',i-1)\) and moving \(d\) to the left yields \((\tf',i)\) to \((\tf',i+1)\),
Moving \(d\) units to the left would also be viewed as adding \(\frac{1}{\Sigmasq}\).
\end{proof}

The following corollary basically highlights that, in the case of regular fibers of Brieskorn spheres, existing results that make the \(\tau\) function for the ambient three-manifold more computable \cite{brieskornAlgorithm,singularSemigroup,NemAR} can be leveraged for the computation of \(\CFKbn(G_{v_0})\). 
Note that while the statements of those results focus on the case of 3 or more singular fibers and often look at the positive ray of the corresponding filtered line, this has not because the arguments specifically required 3 or more regular fibers or restriction the positive ray.
Rather for the underlying three-manifold the case of 2 singular fibers always gives \(S^3\), a known calculation, and the resulting computation will always deformation retract onto the positive ray.
When adding a knot, including the case of two singular fibers will allow us to compute the homotopy type for torus knots, and Proposition \ref{prop:ARNullHomolgous} means that being able to access information contained on the negative ray will make calculation easier.

Before we provide the corollary summarizing these results, for a Brieskorn complete intersection \(\Sigma(p_1,p_2,\ldots, p_m)\), we will define the following values to help consolidate notation
\begin{align*}
\alpha &= \prod_{i=1}^mp_i \\
\gamma &= \left(\prod_{i=1}^mp_i\right)\left(m-2 - \sum_{i=1}^m \frac{1}{p_i}\right).
\end{align*}
This is consistent with the notation in \cite{singularSemigroup}.
Then, the following corollary summarizes these results in this context.

\begin{corollary}\label{cor:BrieskornReduc}
If \(G_{v_0}(n)_{u_0}\) represents the regular fiber of a Brieskorn complete intersection \(\Sigma(p_1,p_2,\ldots, p_m)\) given as an unweighted vertex adjacent to the central vertex of a star, then \(\CFKbn(G_{v_0}(n)_{u_0},\tf)\) is equivalent to a filtered line, so that
\begin{align*}
h_1(n)&=\frac{K_{con}^2 +s}{4} - 2 \tau(n) \\
h_2(n) &= h_1\left(n - \alpha\right) \\
\Gamma(n) &= n-\alpha\\
J(n) &= 1+ \gamma-n \\
J_{v_0}(n)&= 1+ \gamma + \alpha-n,
\end{align*}
with higher order coherence relations on \(\Ii\) and \(\Jj\) determined by \(\Ii^2=\id\) and \(\Jj^2 =\id\).
In fact supposing \(1+\gamma\) is positive, \(\CFKbn(G_{v_0}(n)_{u_0})\) is supported on the filtered line segment from 0 to \(1+\gamma + \alpha\).

Here \(\tau\) can be pinned down by the following conditions
\begin{align*}
\tau(0)&=0\\
\tau(n+1)-\tau(n) &= 1+ e_0n - \sum_{i=1}^m \left\lceil \frac{nq_i}{p_i} \right\rceil,
\end{align*}
where \(e_0\) and \(q_i\) are the unique solution to
\[ e_0p_1p_2\cdots p_m + q_1p_2p_3\cdots p_m + \cdots q_mp_1p_2\cdots p_{m-1} \]
satisfying \(1\leq q_i\leq p_i-1\) for all \(i\).
\end{corollary}
\begin{proof}
The formula for \(\tau\) is precisely that provided by \cite{NemAR}.
After all \(B_{\tf,i}\) have been shrunk to a point the actions of \(\Gamma\), \(\Ii\) and \(\Jj\) are determined completely by how they act on the indices of the \(B_{\tf,i}\).
The formula for the second height function and action of \(\Gamma\) come from computing \(\frac{1}{\Sigmasq}\) for a Brieskorn sphere.
Note that a Brieskorn sphere \(\Sigma(p_1,\ldots, p_m)\) is surgery on a connect sum of cores of lens spaces, so \(H_1(Y_{G};\Z)=\alpha\).
In order for this to reduce to a integer homology sphere, the order of \(v_0\) in \(H_1(Y_G;\Z)\) must be \(\alpha\).

Next we move to the formula for \(J\).
Note that \cite{singularSemigroup} calculates the coefficient of \(v_0\) of the anticanonical class in \(G_{v_0}(n)\) is \(1+\gamma\).
The distance in terms of \(v_0\) between the canonical and anticanonical classes as cohomology classes is then \(2\gamma +2\), but the action on \(\Char(G_{v_0}(n))\) includes a factor of 2, so this distance is actually \(\gamma+1\).
This verifies that \(0\), representing the canonical class on the filtered line, is sent to \(\gamma+1\), representing the anticanonical class.
More generally under \(\Ii\), the action of \(v_0\) will be comes the action of \(-v_0\) and thus \(n\) is sent to \(\gamma+1-n\).
The formula for \(J_{u_{0}}\) can be constructed from the formula for \(\Ii\) and the formula for \(\Gamma\).

The proof of Lemma \ref{lem:EffectiveRestriction} works to show that if \(n\) is negative \(h_1(n)\leq h_1(0)\) and \(h_2(n)\leq h_2(0)\) and thus \(\CFKbn(G_{v_0}(n)_{u_0},\tf)\) is supported on the filtered positive ray.
The map \(\Jj\) can take the infinite sequence realizing that reduction via the sequence \(k_{2,n}\) of Lemma \ref{lem:boxReduction} to an infinite sequence \(k_{1,n}\) starting at \(1+\gamma+\alpha\) and moving upwards.
As such this allows us to use Lemma \ref{lem:boxReduction} to restrict ourselves to the line segment.
\end{proof}

\subsection{Examples with analysis}\label{subsec:exampleARKnots}
\begin{example}
We will now consider the case where \(G_{v_0}\) represents the trefoil as depicted in Figure \ref{fig:trefGraph}, which can be interpreted as the regular fiber of \(\Sigma(2,3)\).
We will show that this has knot lattice space equivalent to that in Figure \ref{fig:trefFiltered}.
First note that for the trefoil 
\[\alpha = 6 \quad \gamma = 6\left(\frac{-1}{2}+\frac{-1}{3}\right)=-5,\]
so in particular \(\CFKbn(G_{v_0},\tf)\) is supported on the line segment from 0 to 2.
Additionally, the trefoil has
\[\tau(n+1) -\tau(n)= 1+n - \left\lceil \frac{n}{2}\right\rceil - \left\lceil \frac{n}{3}\right\rceil.\]
Note that \(\tau(1)-\tau(0)=1\) and \(\tau(2)-\tau(1)=0\).
We can also conclude from \(S^3\) being an \(L\)-space that \(K_{con}\) must maximize \(h_U\) across the vertices, and which for \(S^3\) is known to be maximized at 0.
As such,
\begin{align*}
h_1(0) &=0 -2 \tau(0)=0  \\
h_1(1) &= h_1(0)-2(\tau(1)-\tau(0))=-2 \\
h_1(2) &= h_1(1)-2(\tau(2)-\tau(1))=-2 \\
h_2(0)&= h_1(\Jj(0))= h_1(2)=-2 \\ 
h_2(1) &=h_1(\Jj(1))=h_1(1)=-2 \\
h_2(2) &=h_1(\Jj(2))=h_1(0)=0.
\end{align*}
These heights are extended across the edges by taking minima.
The resulting three vertex simplicial set is given by Figure \ref{subfig:trefcubesExpand}, which is homotopy equivalent to Figure \ref{subfig:trefcubes}.

Additionally, we can see that \(\Ii\) acts here by \(\Ii(n)=1-2-3-n = -4-n\).
This takes the line segment from 0 to 2 to the line segment from \(-6\) to \(-4\), which is homotopic to the map going to 0.
We could have made this argument knowing that the underlying three-manifold is \(S^3\) and thus an \(L\)-space and the filtered lattice simplicial set is subcontractible (and thus every map is homotopic to 0).
However, even for regular fibers of non-\(L\)-space Brieskorn spheres, the map \(\Ii\) will likely reflect the portion of the line segment from \(1+\gamma\) to \(1+\gamma+\alpha\) into negative indices, and that portion is then deformed to 0 relative to the rest of \(\Ii\) under the homotopy equivalence between the whole line and the line segment.
Meanwhile, the map \(\Jj\) acts by \(\Jj(n)= 2-n\), i.e. a reflection of the line segment in Figure \ref{fig:trefFiltered} across its middle segment.
The map \(\Gamma\) acts by \(\Gamma(n)= n-6\), which on the line segment sends everything to 0.
\end{example}

\begin{figure}
\begin{center}
\begingroup%
  \makeatletter%
  \providecommand\color[2][]{%
    \errmessage{(Inkscape) Color is used for the text in Inkscape, but the package 'color.sty' is not loaded}%
    \renewcommand\color[2][]{}%
  }%
  \providecommand\transparent[1]{%
    \errmessage{(Inkscape) Transparency is used (non-zero) for the text in Inkscape, but the package 'transparent.sty' is not loaded}%
    \renewcommand\transparent[1]{}%
  }%
  \providecommand\rotatebox[2]{#2}%
  \newcommand*\fsize{\dimexpr\f@size pt\relax}%
  \newcommand*\lineheight[1]{\fontsize{\fsize}{#1\fsize}\selectfont}%
  \ifx\svgwidth\undefined%
    \setlength{\unitlength}{123.93844128bp}%
    \ifx\svgscale\undefined%
      \relax%
    \else%
      \setlength{\unitlength}{\unitlength * \real{\svgscale}}%
    \fi%
  \else%
    \setlength{\unitlength}{\svgwidth}%
  \fi%
  \global\let\svgwidth\undefined%
  \global\let\svgscale\undefined%
  \makeatother%
  \begin{picture}(1,0.67205921)%
    \lineheight{1}%
    \setlength\tabcolsep{0pt}%
    \put(0,0){\includegraphics[width=\unitlength]{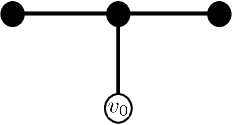}}%
    \put(0.46297578,0.54646565){\color[rgb]{0,0,0}\makebox(0,0)[lt]{\lineheight{1.25}\smash{\begin{tabular}[t]{l}-1\end{tabular}}}}%
    \put(0.00874146,0.54647702){\color[rgb]{0,0,0}\makebox(0,0)[lt]{\lineheight{1.25}\smash{\begin{tabular}[t]{l}-2\end{tabular}}}}%
    \put(0.9036697,0.54765033){\color[rgb]{0,0,0}\makebox(0,0)[lt]{\lineheight{1.25}\smash{\begin{tabular}[t]{l}-3\end{tabular}}}}%
  \end{picture}%
\endgroup%

\end{center}
\caption{Graph depicting the trefoil, aka the regular fiber of \(\Sigma(2,3)\).}\label{fig:trefGraph}
\end{figure}

\begin{figure}

\begin{subfigure}[b]{\textwidth}
\begin{center}
\begingroup%
  \makeatletter%
  \providecommand\color[2][]{%
    \errmessage{(Inkscape) Color is used for the text in Inkscape, but the package 'color.sty' is not loaded}%
    \renewcommand\color[2][]{}%
  }%
  \providecommand\transparent[1]{%
    \errmessage{(Inkscape) Transparency is used (non-zero) for the text in Inkscape, but the package 'transparent.sty' is not loaded}%
    \renewcommand\transparent[1]{}%
  }%
  \providecommand\rotatebox[2]{#2}%
  \newcommand*\fsize{\dimexpr\f@size pt\relax}%
  \newcommand*\lineheight[1]{\fontsize{\fsize}{#1\fsize}\selectfont}%
  \ifx\svgwidth\undefined%
    \setlength{\unitlength}{155.44858437bp}%
    \ifx\svgscale\undefined%
      \relax%
    \else%
      \setlength{\unitlength}{\unitlength * \real{\svgscale}}%
    \fi%
  \else%
    \setlength{\unitlength}{\svgwidth}%
  \fi%
  \global\let\svgwidth\undefined%
  \global\let\svgscale\undefined%
  \makeatother%
  \begin{picture}(1,0.34069901)%
    \lineheight{1}%
    \setlength\tabcolsep{0pt}%
    \put(0,0){\includegraphics[width=\unitlength]{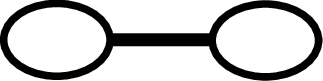}}%
    \put(0.18082044,0.11905248){\makebox(0,0)[t]{\lineheight{1.25}\smash{\begin{tabular}[t]{c}$\left(0,-2\right)$\end{tabular}}}}%
    \put(0.83155499,0.12717908){\makebox(0,0)[t]{\lineheight{1.25}\smash{\begin{tabular}[t]{c}$\left(-2,0\right)$\end{tabular}}}}%
    \put(0.52084746,0.03975565){\makebox(0,0)[t]{\lineheight{1.25}\smash{\begin{tabular}[t]{c}$\left(-2,-2\right)$\end{tabular}}}}%
  \end{picture}%
\endgroup%

\end{center}
\caption{A simplified representation of the knot lattice simplicial set for the trefoil.}\label{subfig:trefcubes}
\end{subfigure}

\begin{subfigure}[b]{\textwidth}
\begin{center}
\begingroup%
  \makeatletter%
  \providecommand\color[2][]{%
    \errmessage{(Inkscape) Color is used for the text in Inkscape, but the package 'color.sty' is not loaded}%
    \renewcommand\color[2][]{}%
  }%
  \providecommand\transparent[1]{%
    \errmessage{(Inkscape) Transparency is used (non-zero) for the text in Inkscape, but the package 'transparent.sty' is not loaded}%
    \renewcommand\transparent[1]{}%
  }%
  \providecommand\rotatebox[2]{#2}%
  \newcommand*\fsize{\dimexpr\f@size pt\relax}%
  \newcommand*\lineheight[1]{\fontsize{\fsize}{#1\fsize}\selectfont}%
  \ifx\svgwidth\undefined%
    \setlength{\unitlength}{258.16956287bp}%
    \ifx\svgscale\undefined%
      \relax%
    \else%
      \setlength{\unitlength}{\unitlength * \real{\svgscale}}%
    \fi%
  \else%
    \setlength{\unitlength}{\svgwidth}%
  \fi%
  \global\let\svgwidth\undefined%
  \global\let\svgscale\undefined%
  \makeatother%
  \begin{picture}(1,0.20185076)%
    \lineheight{1}%
    \setlength\tabcolsep{0pt}%
    \put(0,0){\includegraphics[width=\unitlength]{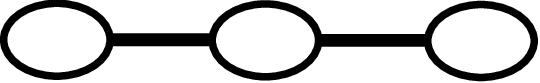}}%
    \put(0.1088753,0.06840557){\makebox(0,0)[t]{\lineheight{1.25}\smash{\begin{tabular}[t]{c}$\left(0,-2\right)$\end{tabular}}}}%
    \put(0.8865995,0.06870464){\makebox(0,0)[t]{\lineheight{1.25}\smash{\begin{tabular}[t]{c}$\left(-2,0\right)$\end{tabular}}}}%
    \put(0.31361175,0.02065951){\makebox(0,0)[t]{\lineheight{1.25}\smash{\begin{tabular}[t]{c}$\left(-2,-2\right)$\end{tabular}}}}%
    \put(0.49320466,0.06907699){\makebox(0,0)[t]{\lineheight{1.25}\smash{\begin{tabular}[t]{c}$\left(-2,-2\right)$\end{tabular}}}}%
    \put(0.69927345,0.02785382){\makebox(0,0)[t]{\lineheight{1.25}\smash{\begin{tabular}[t]{c}$\left(-2,-2\right)$\end{tabular}}}}%
  \end{picture}%
\endgroup%

\end{center}
\caption{The representation of the knot lattice simplicial set for the trefoil produced by the arguments of Corollary \ref{cor:BrieskornReduc}.}\label{subfig:trefcubesExpand}
\end{subfigure}
\caption{The doubly filtered line segment associated to the trefoil.}\label{fig:trefFiltered}
\end{figure}

\begin{example}\label{ex:regFibSig237}
Let \(G_{v_0}\) represent the regular fiber of the Brieskorn sphere \(\Sigma(2,3,7)\) with the standard star shaped plumbing.
It has formula for \(\tau\) given by
\[\tau(n+1)-\tau(n)=1+n - \left\lceil \frac{n}{2}\right\rceil -\left\lceil \frac{n}{3}  \right\rceil -\left\lceil \frac{n}{7}\right\rceil.\]
We also have that \(\CFKbn(G_{v_0},\tf)\) under the filtered line model is supported on the line segment from 0 to \(44\), as \(\gamma = 1\) and \(\alpha=42\).
As such \(\Ii(n)= 2-n\), while \(\Jj(n)= 44-n\) and \(\Gamma(n)=n-42\).
Because \(\Jj(n)\) exchanges the two height functions, it will suffice to compute \(h_1\) and \(h_2\) from \(0\) to \(22\).
Using \(\tau\) as above as well as a computer we can construct a table of heights on the vertices of the linear model of \(\CFKbn(G_{v_0},\tf)\).
This is still long enough that we will do some simplification to focus on the key values rather than overwhelming with a table of 45 different numbers.
This process will also provide a framework for doing simplifications in other similar scenarios.

To conduct those simplifications, we will use a fact stated in \cite{singularSemigroup} and previously observed in \cite{singDim}.
In particular, for a Brieskorn complete intersection, the value of \(\tau(l+1)- \tau(l)\) can be written as \(1+N(l)\) and \(N(l+\alpha)=N(l)+1\).
It follows that
\begin{align*}
h_2(n+1)- h_2(n)&= h_1(n+1+\alpha)-h_1(n+\alpha) \\
&= -2(\tau(n+1+\alpha) -\tau(n+\alpha)) \\
&= -2(1+N(l+\alpha)) \\
&=-2(1+N(l)+1) \\
&= -2(\tau(n+1)-\tau(n)+1) \\
&=h_1(n+1)-h_1(n)-2.
\end{align*}
Direct computation reveals that for all \(0\leq n \leq 22\), \(h_1(n+1)-h_1(n)\) equals 2,0, or -2.
In those cases \(h_2(n+1)-h_2(n)\) is 0, -2, and -4 respectively.
With this in mind, it is not possible for one of \(h_1\) and \(h_2\) to be strictly increasing across an edge in that range and the other height function to be strictly decreasing across that same edge.
Either both are decreasing, \(h_1\) is increasing and \(h_2\) is constant, or \(h_1\) is constant and \(h_2\) is decreasing.

As such the only vertices \(n\) that matter are the vertices which are local minima (relative to their neighboring vertices) for both \(h_1\) and \(h_2\) or local maxima for both \(h_1\) and \(h_2\).
In the first case, the edge from \(n-1\) to \(n\) will have \(h_1\) change by 0 or -2 (and thus \(h_2\) also decreases), while the edge from \(n\) to \(n+1\) will have \(h_1\) change by 2 (and thus \(h_2\) remains the same).
The second case swaps the roles of the incoming and outgoing edges.
Any vertex \(n\) that has \(h_1\) and \(h_2\) increasing or decreasing across both the incoming and outgoing edges can be removed, replacing the incoming and outgoing edges with a single edge minimizing \(h_1\) and \(h_2\) across both the two removed edges.
Next, the vertices that locally minimize \(h_1\) and \(h_2\) can also be removed, and their heights will be the height given to the edges between local maximums interspersed between them.
We report these local minimums and maximums and their heights below.
\[
\begin{array}{c|c|c|c}
n & (h_1(n),h_2(n)) & n & (h_1(n),h_2(n)) \\ \hline
0 & (0,-44) & 24 & (-12,-8) \\
1 & (-2, -44) & 25 & (-14,-8) \\
6 & (0,-32) & 26 & (-14,-6) \\
7 & (-2,-32) & 29 & (-20,-6) \\
12 & (-2,-22) & 30 & (-20,-4) \\
13 & (-4,-22) & 31 & (-22,-4) \\
14 & (-4,-20) & 32 & (-22,-2) \\
15 & (-6,-20) & 37 & (-32,-2) \\
18 & (-6,-14) & 38 & (-32,0) \\
19 & (-8,-14) & 43 & (-44,-2) \\
20 & (-8,-12) & 44 & (-44,0) \\
22 & (-12,-2) & & 
\end{array}
\]
\end{example}

We now provide some analysis of the regular fiber of \(\Sigma(2,3,7)\) using the example calculation above that the persistence complex of the knot lattice simplicial set computes knot Floer homology \cite{linkLattice}.
Note that we will talk about the knot Floer complex in terms of the vertices and edges of the simplicial line segment of Example \ref{ex:regFibSig237}, but these arguments require functoriality statements that are only known in terms of the chain complexes and  not the doubly filtered simplicial sets.

It is known that knot Floer homology detects the genus for null homologous knots in rational homology three spheres as \(\widehat{HFK}(Y,K,g(K))\neq  0\) and for \(i>g(K)\), \(\widehat{HFK}(Y,K,i)=0\), where \(\widehat{HFK}(Y,K,i)\) is the Alexander grading \(i\) portion of knot Floer homology taken with \(\Fb_2[\Uc,\Vc]/(\Uc,\Vc)\) coefficients.
A proof of this fact for knots in \(S^3\) was given in \cite{origGenus}, though the result was stated for null-homologous knots in rational homology three-spheres in \cite{suturedGenus}.
Note that the largest Alexander grading of any simplex in figure \ref{fig:XKIZ} is \(22\) acheived on the vertex \(n=0\), which provides the generator of \(\widehat{HFK}(Y,K,6)\), and thus 22 is the highest inhabited Alexander grading.
Therefore the genus of this knot is 22.

Given that this is a generalized algebraic knot, it should not be surprising that the top grading has rank 1, indicative of the knot being fibered \cite{fiberedHFK} as the same argument from algebraic knots in \(S^3\) carries over to the case of strong generalized algebraic knots.
Note strong algebraic knots are those realized with an actual nested analytic singularity (in contrast to those that are only guaranteed to be analytic in any resolution), which is equivalent to the knot being null homologous (see the discussion in Section 2.1.2 of \cite{knotLatticeInvariance}).
In this argument, the Milnor fiber of the singularity defining the knot gives the Seifert surface, and varying the parameter used to define the milnor fiber provides the fibration.

This argument also shows that the genus of the knot in \(Y\times I\) matches the Seifert genus, a bound that is also encoded in the knot lattice homology.
In particular, a genus \(g\) surface in \(Y\times I\) would provide a map from \(A_{g,!}(\CFb(G,\tf))\) of the underlying 3-manifold to \(\CFKb(G_{v_0},\tf)\) that is a homotopy equivalence after the application of \(p_{1,!}\).
We can further assume that up to further homotopy, it is the homotopy equivalence given by naturality, so representing \(\CFb(G,\tf)\) by \(p_{1,!}(\CFKb(G_{v_0},\tf)\) the map should be homotopic to the identity.
Because the vertex \(n=0\) that realized the relevant Alexander grading was isolated, allowing the identity up to homotopy does not actually provide more freedom, and one needs \(g=22\).

Finally, we can make claims about the genus of \((Y,K)\) in self-homology cobordisms of \(Y\), i.e. the genus of a surface \(S\) in a homology cobordism \((W,S)\) from \((Y,\emptyset)\) to \((Y,K)\), where by homology cobordism we mean that the inclusion of \(Y\) on either end of \(W\) induces an isomorphism on singular homology with integer coefficients.
This relaxes the condition on the map of lattice chain complexes for the underlying three-manifolds from being a chain homotopy equivalence to only being so after the inversion of the variable \(U\).
In other words we want a map
\[f\colon A_{g,!}(\CFb(G,\tf))\to \CFKb(G_{v_0},\tf)\]
so that after forgetting both filtrations it becomes a homotopy equivalence.
The process of forgetting both filtrations can be modeled algebraically by inverting both \(\Uc\) and \(\Vc\), where inverting \(\Vc\) forgets the knot filtration and inverting \(\Uc\) then forgets the original filtration.
An example of this would be the map that sends the entire complex to the vertex \(n=6\), which has height \((0,-32)\), indicating that the genus of \((Y,K)\) in any self homology cobordism is at least 16.

However, we do know in this case that the involutive map from lattice homology agrees with that from Heegaard Floer homology (see \cite{involCheck}).
This allows us to use the additional restriction that our maps on the level of lattice chain complexes have to commute with \(\Ii\) up to homotopy, so our maps on the lattice chain complexes need to be local equivalences in the sense of \cite{involConnect}.
In that case, one could not send both vertices of height \((0,-2g)\) to the vertex \(n=6\) as \(\Ii\) interchanges the vertices \(n=0\) and \(n=6\).
However, that means that some vertex of height \((0,-2g)\) must be mapped to a linear combination of the vertices \(n=0\) and \(n=6\).
As such, the minimal genus in a self homology cobordism would actually be 22.
Since this is the genus of the knot, it is indeed the minimum genus in a self homology cobordism.

\section{An Iterated Example}\label{sec:IteratedEx}

\subsection{Iterating Framings}
The power of this surgery formula goes well beyond almost rational knots, as one could hypothetically use it to build out to any knot via a sequence of surgeries and connect sums.
By breaking it down into such a sequence one has the opportunity to perform simplifications at each step, thereby reducing the complexity of further computations.

However, to do so we will need to keep track of the relationship between the Seifert and graphical framings during this process.
First note the following observation from Section 5.1.2 of \cite{knotLatticeInvariance}: letting \(\Sigmasq\) be the Seifert framing, \(n\) the framing on the graph, and \(\Sigma_0 \in H_2(X_G;\Q)\) so that \(\PD[\Sigma_0]=\PD[v_0]\) we have that
\[\Sigmasq = n -\left(\Sigma_0\right)^2.\]
The following lemma will then be useful to track the value of the new \(\left(\Sigma_0\right)^2\) after doing surgery and thus what the new difference between the graphical and Seifert framings will be.

\begin{lemma}\label{lem:newframing}
Let \(G_{v_0}\) be a negative definite graph with an unweighted vertex \(v_0\) and let \(n\) be such that \(G_{v_0}(n)_{u_0}\) is also negative definite. Let \(\Sigma_0 \in H_2(X_{G_{v_0}(n)};\Q)\) have \(\PD[\Sigma_0]=\PD[u_0]\).
Then, if \(\Sigmasq\) is the Seifert framing used in doing the surgery from \(G_{v_0}\) to \(G_{v_0}(n)\), then \((\Sigma_0)^2 = \frac{1}{\Sigmasq}\).
\end{lemma}
\begin{proof}
Let \(\Gc\) represent the adjacency matrix of \(G\) and \(\Gc'\) the adjacency matrix of \(G_{v_0}(n)\).
The value of \(\left(\Sigma_0\right)^2\) will be \((v_0,v_0)\) entry in \((\Gc')^{-1}\).
In particular, with \(\PD[u_0]\) being the dual (with respect to the basis of vertices) of \(v_0\), \(\Sigma_0= \Gc^{-1} v_0^{\ast}\) and squaring uses the matrix \(\Gc\), giving 
\[\Sigma_0^T\Gc\Sigma_0= (v_0^{\ast})^T(\Gc^{-1})^T\Gc\Gc^{-1}v_0^{\ast}=(v_0^{\ast})^T\Gc^{-1}v_0^{\ast}.\]
Using the adjoint formula for the inverse of a matrix, as well as the block form of \(\Gc'\), we have that this matrix entry is can be computed as \(\frac{\det(\Gc)}{\det(\Gc')}\).
Recall that for a plumbing graph \(G\) the absolute value of the determinant of the adjacency matrix computes \(|H_1(Y_G;\Z)|\), so \(\Sigma_0^2= \frac{-|H_1(Y_G;\Z)|}{|H_1(Y_{G_{v_0}(n)};\Z)|}\), with the negative coming from how the negative seifert framing should affect signs of determinants.
At the same time our seifert framing comes from \(\Sigma^2= (v_0-\Sigma_0)^2\) with \(\Sigma\) orthogonal to everything in \(H_2(X_G;\Q)\).
This process does not change the resulting determinant of \(\Gc'\).
As such, we have \(-\Sigmasq |H_1(Y_G;\Z)|= |H_1(Y_{G_{v_0}(n)};\Z)|\).
Comparing with our formula for \(\Sigma_0^2\) yields the desired result.
\end{proof}

\subsection{The First Iteration}
In this section, we will be exploring the knot given in Figure \ref{fig:nonARknot}.
If \(G_{v_0}\) represents the trefoil as in Figure \ref{fig:trefGraph}, we can decompose the graph of Figure \ref{fig:nonARknot} as
\[\left(G_{v_0}(-8)_{u_0}\# G_{v_0}(-9)_{u_0}\right)(-1)_{w_0}.\]
We will refer to this 3-manifold and knot as \((Y,K)\), and condence notation for the graph to \(\tilde{G}_{w_0}\).
Because \(G_{v_0}(-7)\) is the integral homology sphere \(\Sigma(2,3,7)\) the graph framing on this trefoil is 6 less than the Seifert framing.
Therefore, \(G_{v_0}(-8)_{u_0}\) must be the dual of knot a Seifert framing -2 surgery and \(G_{v_0}(-9)_{u_0}\) the dual knot of a Seifert framing -3 surgery.
Since \(G\) represents \(S^3\), we can use Lemma \ref{lem:newframing} to compute that the graph framing of -1 on the subsequent connect sum corresponds to a Seifert framing of \(-1+\frac{1}{2} + \frac{1}{3} = \frac{-1}{6}\).
As such letting \(X\) represent the filtered space with involutive structures in Figure \ref{subfig:trefcubes}, we have that there exists a doubly filtered homotopy equivalence
\[ \CFKIn(\left(G_{v_0}(-8)_{u_0}\# G_{v_0}(-9)_{u_0}\right)(-1)_{w_0}) \cong \XKI_{\frac{-1}{6}}\left(\XKI_{-2}(X)\otimes \XKI_{-3}(X)\right).\]

\begin{figure}
\begin{subfigure}[b]{.47\textwidth}
\begingroup%
  \makeatletter%
  \providecommand\color[2][]{%
    \errmessage{(Inkscape) Color is used for the text in Inkscape, but the package 'color.sty' is not loaded}%
    \renewcommand\color[2][]{}%
  }%
  \providecommand\transparent[1]{%
    \errmessage{(Inkscape) Transparency is used (non-zero) for the text in Inkscape, but the package 'transparent.sty' is not loaded}%
    \renewcommand\transparent[1]{}%
  }%
  \providecommand\rotatebox[2]{#2}%
  \newcommand*\fsize{\dimexpr\f@size pt\relax}%
  \newcommand*\lineheight[1]{\fontsize{\fsize}{#1\fsize}\selectfont}%
  \ifx\svgwidth\undefined%
    \setlength{\unitlength}{155.44858437bp}%
    \ifx\svgscale\undefined%
      \relax%
    \else%
      \setlength{\unitlength}{\unitlength * \real{\svgscale}}%
    \fi%
  \else%
    \setlength{\unitlength}{\svgwidth}%
  \fi%
  \global\let\svgwidth\undefined%
  \global\let\svgscale\undefined%
  \makeatother%
  \begin{picture}(1,0.34069901)%
    \lineheight{1}%
    \setlength\tabcolsep{0pt}%
    \put(0,0){\includegraphics[width=\unitlength]{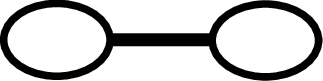}}%
    \put(0.18082044,0.11905248){\makebox(0,0)[t]{\lineheight{1.25}\smash{\begin{tabular}[t]{c}$\left(\frac{-1}{4},\frac{-7}{4} \right)$\end{tabular}}}}%
    \put(0.83155499,0.12717908){\makebox(0,0)[t]{\lineheight{1.25}\smash{\begin{tabular}[t]{c}$\left(\frac{-1}{4},\frac{1}{4}\right)$\end{tabular}}}}%
    \put(0.52084746,0.03975565){\makebox(0,0)[t]{\lineheight{1.25}\smash{\begin{tabular}[t]{c}$\left( \frac{-9}{4},\frac{-7}{4}\right)$\end{tabular}}}}%
  \end{picture}%
\endgroup%

\caption{The portion of \(\XKI_{-2}(X)\) in \Spinc structure \([0]\).}\label{subfig:XKI20}
\end{subfigure}
\hfill
\begin{subfigure}[b]{.47\textwidth}
\begingroup%
  \makeatletter%
  \providecommand\color[2][]{%
    \errmessage{(Inkscape) Color is used for the text in Inkscape, but the package 'color.sty' is not loaded}%
    \renewcommand\color[2][]{}%
  }%
  \providecommand\transparent[1]{%
    \errmessage{(Inkscape) Transparency is used (non-zero) for the text in Inkscape, but the package 'transparent.sty' is not loaded}%
    \renewcommand\transparent[1]{}%
  }%
  \providecommand\rotatebox[2]{#2}%
  \newcommand*\fsize{\dimexpr\f@size pt\relax}%
  \newcommand*\lineheight[1]{\fontsize{\fsize}{#1\fsize}\selectfont}%
  \ifx\svgwidth\undefined%
    \setlength{\unitlength}{155.44858437bp}%
    \ifx\svgscale\undefined%
      \relax%
    \else%
      \setlength{\unitlength}{\unitlength * \real{\svgscale}}%
    \fi%
  \else%
    \setlength{\unitlength}{\svgwidth}%
  \fi%
  \global\let\svgwidth\undefined%
  \global\let\svgscale\undefined%
  \makeatother%
  \begin{picture}(1,0.34069901)%
    \lineheight{1}%
    \setlength\tabcolsep{0pt}%
    \put(0,0){\includegraphics[width=\unitlength]{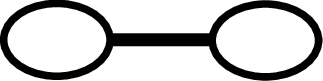}}%
    \put(0.18082044,0.12095996){\makebox(0,0)[t]{\lineheight{1.25}\smash{\begin{tabular}[t]{c}$\left(\frac{1}{4},\frac{-1}{4} \right)$\end{tabular}}}}%
    \put(0.83346243,0.12717908){\makebox(0,0)[t]{\lineheight{1.25}\smash{\begin{tabular}[t]{c}$\left(\frac{-7}{4},\frac{-1}{4}\right)$\end{tabular}}}}%
    \put(0.52275494,0.03784817){\makebox(0,0)[t]{\lineheight{1.25}\smash{\begin{tabular}[t]{c}$\left( \frac{-7}{4},\frac{-9}{4}\right)$\end{tabular}}}}%
  \end{picture}%
\endgroup%

\caption{The portion of \(\XKI_{-2}(X)\) in \Spinc structure \([1]\).}\label{subfig:XKI21}
\end{subfigure}
\begin{subfigure}[b]{.47\textwidth}
\begingroup%
  \makeatletter%
  \providecommand\color[2][]{%
    \errmessage{(Inkscape) Color is used for the text in Inkscape, but the package 'color.sty' is not loaded}%
    \renewcommand\color[2][]{}%
  }%
  \providecommand\transparent[1]{%
    \errmessage{(Inkscape) Transparency is used (non-zero) for the text in Inkscape, but the package 'transparent.sty' is not loaded}%
    \renewcommand\transparent[1]{}%
  }%
  \providecommand\rotatebox[2]{#2}%
  \newcommand*\fsize{\dimexpr\f@size pt\relax}%
  \newcommand*\lineheight[1]{\fontsize{\fsize}{#1\fsize}\selectfont}%
  \ifx\svgwidth\undefined%
    \setlength{\unitlength}{155.44858437bp}%
    \ifx\svgscale\undefined%
      \relax%
    \else%
      \setlength{\unitlength}{\unitlength * \real{\svgscale}}%
    \fi%
  \else%
    \setlength{\unitlength}{\svgwidth}%
  \fi%
  \global\let\svgwidth\undefined%
  \global\let\svgscale\undefined%
  \makeatother%
  \begin{picture}(1,0.34069901)%
    \lineheight{1}%
    \setlength\tabcolsep{0pt}%
    \put(0,0){\includegraphics[width=\unitlength]{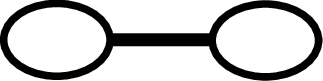}}%
    \put(0.18272791,0.12095996){\makebox(0,0)[t]{\lineheight{1.25}\smash{\begin{tabular}[t]{c}$\left(\frac{-1}{2},\frac{-11}{6} \right)$\end{tabular}}}}%
    \put(0.83155499,0.11764169){\makebox(0,0)[t]{\lineheight{1.25}\smash{\begin{tabular}[t]{c}$\left(\frac{-1}{2},\frac{1}{6}\right)$\end{tabular}}}}%
    \put(0.52084746,0.03975565){\makebox(0,0)[t]{\lineheight{1.25}\smash{\begin{tabular}[t]{c}$\left( \frac{-5}{2},\frac{-11}{6}\right)$\end{tabular}}}}%
  \end{picture}%
\endgroup%

\caption{The portion of \(\XKI_{-3}(X)\) in \Spinc structure \([0]\).}\label{subfig:XKI30}
\end{subfigure}\hfill
\begin{subfigure}[b]{.47\textwidth}
\begingroup%
  \makeatletter%
  \providecommand\color[2][]{%
    \errmessage{(Inkscape) Color is used for the text in Inkscape, but the package 'color.sty' is not loaded}%
    \renewcommand\color[2][]{}%
  }%
  \providecommand\transparent[1]{%
    \errmessage{(Inkscape) Transparency is used (non-zero) for the text in Inkscape, but the package 'transparent.sty' is not loaded}%
    \renewcommand\transparent[1]{}%
  }%
  \providecommand\rotatebox[2]{#2}%
  \newcommand*\fsize{\dimexpr\f@size pt\relax}%
  \newcommand*\lineheight[1]{\fontsize{\fsize}{#1\fsize}\selectfont}%
  \ifx\svgwidth\undefined%
    \setlength{\unitlength}{155.44858437bp}%
    \ifx\svgscale\undefined%
      \relax%
    \else%
      \setlength{\unitlength}{\unitlength * \real{\svgscale}}%
    \fi%
  \else%
    \setlength{\unitlength}{\svgwidth}%
  \fi%
  \global\let\svgwidth\undefined%
  \global\let\svgscale\undefined%
  \makeatother%
  \begin{picture}(1,0.34069901)%
    \lineheight{1}%
    \setlength\tabcolsep{0pt}%
    \put(0,0){\includegraphics[width=\unitlength]{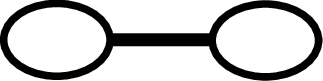}}%
    \put(0.18082044,0.12095996){\makebox(0,0)[t]{\lineheight{1.25}\smash{\begin{tabular}[t]{c}$\left(\frac{1}{6},\frac{-1}{2} \right)$\end{tabular}}}}%
    \put(0.82774003,0.11764168){\makebox(0,0)[t]{\lineheight{1.25}\smash{\begin{tabular}[t]{c}$\left(\frac{-11}{6},\frac{-1}{2}\right)$\end{tabular}}}}%
    \put(0.52084746,0.03975565){\makebox(0,0)[t]{\lineheight{1.25}\smash{\begin{tabular}[t]{c}$\left( \frac{-11}{6},\frac{-5}{2}\right)$\end{tabular}}}}%
  \end{picture}%
\endgroup%

\caption{The portion of \(\XKI_{-3}(X)\) in \Spinc structure \([2]\).}\label{subfig:XKI32}
\end{subfigure}

\begin{subfigure}[b]{.47\textwidth}
\begingroup%
  \makeatletter%
  \providecommand\color[2][]{%
    \errmessage{(Inkscape) Color is used for the text in Inkscape, but the package 'color.sty' is not loaded}%
    \renewcommand\color[2][]{}%
  }%
  \providecommand\transparent[1]{%
    \errmessage{(Inkscape) Transparency is used (non-zero) for the text in Inkscape, but the package 'transparent.sty' is not loaded}%
    \renewcommand\transparent[1]{}%
  }%
  \providecommand\rotatebox[2]{#2}%
  \newcommand*\fsize{\dimexpr\f@size pt\relax}%
  \newcommand*\lineheight[1]{\fontsize{\fsize}{#1\fsize}\selectfont}%
  \ifx\svgwidth\undefined%
    \setlength{\unitlength}{45.3893226bp}%
    \ifx\svgscale\undefined%
      \relax%
    \else%
      \setlength{\unitlength}{\unitlength * \real{\svgscale}}%
    \fi%
  \else%
    \setlength{\unitlength}{\svgwidth}%
  \fi%
  \global\let\svgwidth\undefined%
  \global\let\svgscale\undefined%
  \makeatother%
  \begin{picture}(1,0.84764082)%
    \lineheight{1}%
    \setlength\tabcolsep{0pt}%
    \put(0,0){\includegraphics[width=\unitlength]{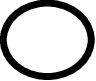}}%
    \put(0.50168222,0.38259051){\makebox(0,0)[t]{\lineheight{1.25}\smash{\begin{tabular}[t]{c}$\left(\frac{1}{6},\frac{1}{6} \right)$\end{tabular}}}}%
  \end{picture}%
\endgroup%

\caption{The portion of \(\XKI_{-3}(X)\) in \Spinc structure \([1]\).}\label{subfig:XKI31}
\end{subfigure}

\caption{Cube decompositions for \(\XKI_{-2}(X)\) and \(\XKI_{-3}(X)\) for different \Spinc structures on \(Y_{G_{v_0}(-2)}\) and \(Y_{G_{v_0}}(-3)\). Because there is only one \Spinc structure \(\tf\) on \(S^3\) that is suppressed in our notation, so \([\tf,i]\) is represented as \([i]\). Note some simplification has already been done, using that each \(A_{\tf,i}\) is homotopy equivalent to a filtered point and then using where \(\eta_1\colon A_{\tf,i}\to B_{\tf,i}\) and/or \(\Gamma^{\tf}\circ \eta_2\colon A_{\tf,i}\to B_{\tf,i+\Sigmasq}\) are themselves weak equivalences to truncate \(\Ic_{S^3,T_{2,3},\Sigmasq,[i]}\) to the portion needed }\label{fig:XKIX}
\end{figure}

First, we can compute \(\XKI_{-2}(X)\) and \(\XKI_{-3}(X)\), which are shown in Figure \ref{fig:XKIX}.
Note here that because there is only one \Spinc structure on \(S^3\), we will be suppressing its contribution to the notation, and thus \(\Ic_{S^3,T_{2,3},\Sigmasq}\) will have objects \((i,a)\) and \((i,b)\) and \Spinc structures on \(S^3_{\Sigmasq}(T_{2,3})\) have the form \([i]\) viewing \(i\) in \(\Z/\Sigmasq\Z\).
For these computations, note that the map \(\eta_1\colon A_i^{\ast}(X)\to p_{1,!}(X)\) is a filtered homotopy equivalence when \(i\geq 1\) (in fact, it is a filtered isomorphism), and similarly the map \(\eta_2\colon A_i^{\ast}(X) \to p_{2,!}(X)[2i]\) becomes a filtered homotopy equivalence when \(i\leq -1\) and thus \(\eta_2\colon A_i^{\mu}\to B_i^{\mu}\) a filtered homotopy equivalence when \(i\leq -2\)).
This allows for a deformation retract so that \(A_0^{\mu}\) and \(A_{-1}^{\mu}\) are the only \(A_i^{\mu}\) to appear.
For \([i]\) equal to \([0]\) or \([-1]\) for both \(S^3_{-2}(T_{2,3})\) and \(S^3_{-3}(T_{2,3})\) one can use the diagram containing only \(A_i^{\mu}\) along with its \(\lambda_i\) and \(\rho_i\).
For \([i]\) equal to \([-2]\) on \(S^3_{-3}(T_{2,3})\), the filtered homotopy equivalences above allow one to defomration retract \(\XKI_{-3}(X)\) in that \Spinc structure down to \(B_{-2}^{\mu}\).
These arguments are similar to other arguments in the literature using surgery formulas, such as formula (1.12) of \cite{FilteredSurgeryFormula}.

Note now that \(A_{i}^{\mu}\) only depends on \(\Sigmasq\) up to an overall grading shift.
By direct computation for both \(i=0\) and \(i=-1\), \(A_i^{\mu}\) is filtered homotopy equivalent to doubly filtered point.
The height on this point is set so that \(\lambda_{0,\ast}\) acts on persistent homology  by \(\Uc\) and \(\rho_{0,\ast}\) acts on persistent homology by \(\Uc\Vc\).
By the symmetry induced by \(J_{u_0}\), \(\lambda_{-1,\ast}\) acts by \(\Uc\Vc \) and \(\rho_{-1,\ast}\) by \(\Vc\).
Between these and using the grading shifts specified by \(\grf_{\Sigmasq}(i)\), we can reconstruct \(\XKI_{-2}(X)\) and \(\XKI_{-3}(X)\) as shown in Figure \ref{fig:XKIX}.
On \(\XKI_{-2}(X)\) this grading information is pinned down by observing that \(B_0^{\mu}\) is a point with height \(\left( \frac{-1}{4},\frac{-7}{4}\right)\), and then using the actions of \(\lambda_{i,\ast},\rho_{i,\ast}\) and \(J_{u_0}\).
On \(\XKI_{-3}(X)\) the grading information is pinned down by observing that \(B_i^{\mu}\) is always a doubly filtered point and when \(i=-3\) it has height \(\left(\frac{-1}{2},\frac{1}{6}\right)\).
The actions of \(\lambda_{i,\ast},\rho_{i,\ast},J,\) and \(J_{u_0}\) allows the rest of the grading information to be recovered.

This argument also allows us to recover both of the involutive maps, since \(\XKD_{\Sigmasq}(X)\) has objects all doubly filtered points.
Filtered points are subterminal meaning that the space of morphisms from any other doubly filtered space and the object in question is always empty or a point, so given the initial maps \(\Jj\) and \(\Ii\) needed for our homotopy coherent diagrams exist, there is a contractible choice of how to fill in the higher homotopy coherent relations.  
This persists even after taking the homotopy colimit as we can observe that \(\XKI_{\Sigmasq}(X)_{m,n}\) is always a disjoint union of contractible spaces, so, up to a contractible choice, the only relevant information is how \(J\) and \(J_{u_0}\) permute these spaces.
As such, further homotopy coherence relations do not need to be recorded.
Figure \ref{fig:XKIX} is displayed so that on \(\XKI_{-2}(X)\), \(J\) acts by reflection within each \Spinc structure and \(J_{u_0}\) acts by reflecting across the page.
The doubly filtered spaces for \(\XKI_{-3}(X)\) are displayed similarly with respect to \(J_{u_0}\) acting as a reflection across the page between \([0]\) and \([2]\), and \(J\) acts by a reflection on \([2]\).

\begin{figure}
\begin{subfigure}[b]{.95\textwidth}
\[
\begin{tikzcd}
\left([0,0],\frac{5}{12}, a\right) \ar[d] \ar[rd] & \left([1,1],\frac{1}{4}, a\right)\ar[d] \ar[rd] & \left([0,2],\frac{1}{12}, a\right)\ar[d] \ar[rd] &\cdots \ar[d]\\
\left([0,0],\frac{5}{12}, b\right) & \left([1,1],\frac{1}{4}, b\right) & \left([0,2],\frac{1}{12}, b\right) &\left([1,0],\frac{-1}{12}, b\right)
\end{tikzcd}
\]
\caption{The first half of \(\Ic_{Y,[K],\frac{-1}{6}}\) needed.}
\end{subfigure}
\begin{subfigure}[b]{.95\textwidth}
\[
\begin{tikzcd}
 \left([1,0],\frac{-1}{12}, a\right) \ar[d] \ar[rd]& \left([0,1],\frac{-1}{4}, a\right) \ar[d] \ar[rd]& \left([1,2],\frac{-5}{12}, a\right) \ar[d] \ar[rd]&   \\
\left([1,0],\frac{-1}{12}, b\right) & \left([0,1],\frac{-1}{4}, b\right) & \left([1,2],\frac{-5}{12}, b\right) &   \left([0,0],\frac{-7}{12}, b\right)
\end{tikzcd}
\]
\caption{The second half of \(\Ic_{Y,[K],\frac{-1}{6}}\) needed.}
\end{subfigure}
\caption{These figures together show the truncation of \(\Ic_{Y,[K],\frac{-1}{6}}\) needed to compute \(\Xb_{\frac{-1}{6}}(Z)\) based on when the maps \(\eta_1\) and \(\eta_2\) are weak equivalences. Due to length, the diagram has been split over 2 subfigures with \(([1,0], \frac{-1}{12},b)\) being the same object in both subfigures.
The \Spinc structures on \(Y\) are represented by \([i,j]\) for the connect sum of \([i]\) from \(S^3_{-2}(T_{2,3})\) and \([j]\) from \(S^3_{-3}(T_{2,3})\).
While this truncation would not be enough to contain all of \(\XKI_{\frac{-1}{6} }(Z)\), the action of \(\psi\circ \xi\) on top of which the action of \(J_{w_0}\) is constructed acts here as a reflection across \(([0,0],\frac{-7}{12})\), and thus by understanding \(\XKID_{\frac{-1}{6}}(Z)\) on the truncation provided, we can reconstruct a complete picture of of \(\XKI_{\frac{-1}{6}}(Z)\).}\label{fig:XKIDTrunc}
\end{figure}
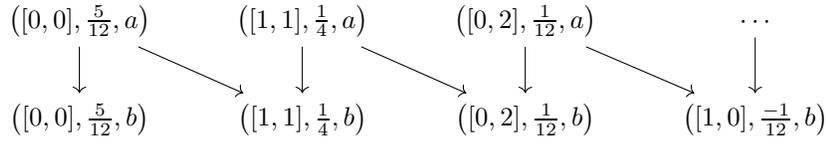
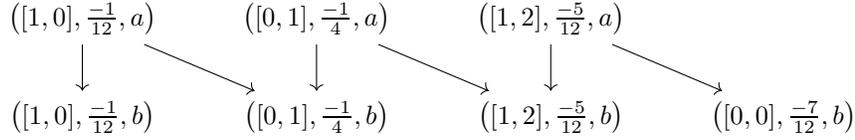

The flip map \(\Gamma\) is the composition \(JJ_{u_0}\) and sends Figure \ref{subfig:XKI21} to Figure \ref{subfig:XKI20} by translation, while Figure \ref{subfig:XKI20} is mapped to the leftmost point of Figure \ref{subfig:XKI21}.
The flip map also sends Figure \ref{subfig:XKI30} to the single point in Figure \ref{subfig:XKI31}.
That single point is then in turn sent to the left most point of  Figure \ref{subfig:XKI32}, and Figure \ref{subfig:XKI32} is mapped via translation to Figure \ref{subfig:XKI30}.

\subsection{The Second Iteration}

To conserve space, for this subsection
\begin{align*}
(Y',K')&:= (S^3_{-2}(T_{2,3})\# S^3_{-3}(T_{2,3}),\mu_{-2}\#\mu_{-3})\\
 Z&:= \XKI_{-2}(X)\otimes\XKI_{-3}(X).
\end{align*}
We will denote \Spinc structores on \(Y'\) by \([j,k]\), where \([j]\) is the restriction to \(S^3_{-2}(T_{2,3})\) and \([k]\) is the restriction to \(S^3_{-3}(T_{2,3})\).

We will now discuss how to provide a similar truncation to \(\Ic_{Y,K,\frac{-1}{6}}\) as we used for \(\Ic_{S^3,T_{2,3},\Sigmasq}\).
In particular, note that \(\eta_{1,[j,k],i}\) will be an isomorphism if \(i\) is bigger than or equal to the largest value of the Alexander grading \(\frac{h_U-h_V}{2}\) that appears on any cube of \(Z^{[j,k]}_{\ast\ast}\)
If that maximum value is realized on a simplex \(\sigma_1\) in \(\XKI_{-2}(X)^{[j]}\) and a simplex \(\sigma_2\) in \(\XKI_{-3}(X)^{[k]}\), then \((\sigma_1,\sigma_2)\) will be similarly maximizing on \(Z\).
Similarly, \(\eta_{2,[j,k],i}\) will be an ismorphism if \(i\) is less than or equal to the smallest value of the Alexander grading that appears on any cube of \(Z^{\tf}_{\ast\ast}\) which can similarly be pinned down by finding the minimums across \(\XKI_{-2}(X)^{[j]}\) and \(\XKI_{-3}(X)^{[k]}\).
This allows us to compute the following table.

\[
\setlength{\delimitershortfall}{0pt}
\begin{array}{c|c|c}
{[j,k]} & \min \frac{h_U-h_V}{2} & \max \matfrac{h_U-h_V}{2} \\ \hline
{[0,0]} & \matfrac{-7}{12} & \matfrac{17}{12}\\[.7ex]
{[1,0]} & \matfrac{-13}{12}& \matfrac{11}{12} \\[.7ex]
{[0,1]} &  \matfrac{-1}{4}& \matfrac{3}{4}\\[.7ex]
{[1,1]} & \matfrac{-3}{4}& \matfrac{1}{4} \\[.7ex]
{[0, 2]} & \matfrac{-11}{12} & \matfrac{13}{12}\\[.7ex]
{[1, 2]} & \matfrac{-17}{12} & \matfrac{7}{12}\\
\end{array}
\]

From this table we can see that if \(i\geq \frac{7}{12}\) then \(\eta_{1,[j,k],i}^{\mu}\) is an isomorphism regardless of the \Spinc structure \([j,k]\) chosen.
Similarly if \(i\leq \frac{-7}{12}\) then \(\eta_{2,[j,k],i}\) is an isomoprhism and if \(i\leq \frac{-19}{12}\) \(\eta_{2,[j,k],i}^{\mu}\) is an isomorphism.
However, note that  \(\psi \circ \xi\) acts on \(\Ic_{Y,K,\Sigmasq}\) as a reflection with \(([0,0],\frac{-7}{12},b)\) as a fixed point.
As such in particular the values of \(A_{[j,k],i}\) for \(i\leq \frac{-7}{12}\) and \(B_{[j,k],i}\) for \(i<\frac{-7}{12}\) can be reconstructed from the objects indexed with \(i\geq \frac{-7}{12}\).
The truncation we will be working with is shown in Figure \ref{fig:XKIDTrunc}.

\begin{figure}
\begin{subfigure}[c]{.3\textwidth}
\scalebox{.7}{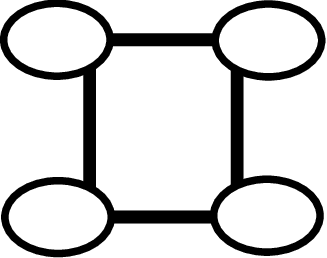}
\caption{\(A_{[0,0],\frac{5}{12}}^{\mu}\)}\label{subfig:A 00 5 12}
\end{subfigure}\hfill
\begin{subfigure}[c]{.3\textwidth}
\scalebox{.7}{
\begingroup%
  \makeatletter%
  \providecommand\color[2][]{%
    \errmessage{(Inkscape) Color is used for the text in Inkscape, but the package 'color.sty' is not loaded}%
    \renewcommand\color[2][]{}%
  }%
  \providecommand\transparent[1]{%
    \errmessage{(Inkscape) Transparency is used (non-zero) for the text in Inkscape, but the package 'transparent.sty' is not loaded}%
    \renewcommand\transparent[1]{}%
  }%
  \providecommand\rotatebox[2]{#2}%
  \newcommand*\fsize{\dimexpr\f@size pt\relax}%
  \newcommand*\lineheight[1]{\fontsize{\fsize}{#1\fsize}\selectfont}%
  \ifx\svgwidth\undefined%
    \setlength{\unitlength}{155.44858437bp}%
    \ifx\svgscale\undefined%
      \relax%
    \else%
      \setlength{\unitlength}{\unitlength * \real{\svgscale}}%
    \fi%
  \else%
    \setlength{\unitlength}{\svgwidth}%
  \fi%
  \global\let\svgwidth\undefined%
  \global\let\svgscale\undefined%
  \makeatother%
  \begin{picture}(1,0.34069901)%
    \lineheight{1}%
    \setlength\tabcolsep{0pt}%
    \put(0,0){\includegraphics[width=\unitlength]{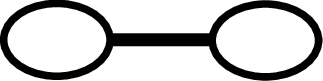}}%
    \put(0.17700548,0.11714502){\makebox(0,0)[t]{\lineheight{1.25}\smash{\begin{tabular}[t]{c}$\left(0,-10\right)$\end{tabular}}}}%
    \put(0.83155499,0.11764169){\makebox(0,0)[t]{\lineheight{1.25}\smash{\begin{tabular}[t]{c}$\left(-2,-12\right)$\end{tabular}}}}%
    \put(0.52084746,0.03975565){\makebox(0,0)[t]{\lineheight{1.25}\smash{\begin{tabular}[t]{c}$\left( -2,-12\right)$\end{tabular}}}}%
  \end{picture}%
\endgroup%
}
\caption{\(A_{[1,1],\frac{1}{4}}^{\mu}\)}\label{subfig:A 11 1 4}
\end{subfigure} \hfill
\begin{subfigure}[c]{.3\textwidth}
\scalebox{.7}{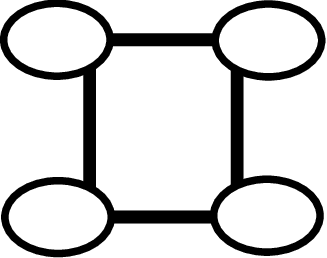}
\caption{\(A_{[0,2],\frac{1}{12}}^{\mu}\)}\label{subfig:A 02 1 12}
\end{subfigure}

\begin{subfigure}[c]{.3\textwidth}
\scalebox{.7}{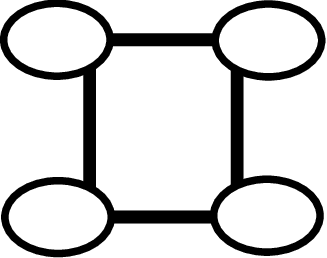}
\caption{\(A_{[1,0],\frac{-1}{12}}^{\mu}\)}\label{subfig:A 10 n1 12}
\end{subfigure}\hfill
\begin{subfigure}[c]{.3\textwidth}
\scalebox{.7}{
\begingroup%
  \makeatletter%
  \providecommand\color[2][]{%
    \errmessage{(Inkscape) Color is used for the text in Inkscape, but the package 'color.sty' is not loaded}%
    \renewcommand\color[2][]{}%
  }%
  \providecommand\transparent[1]{%
    \errmessage{(Inkscape) Transparency is used (non-zero) for the text in Inkscape, but the package 'transparent.sty' is not loaded}%
    \renewcommand\transparent[1]{}%
  }%
  \providecommand\rotatebox[2]{#2}%
  \newcommand*\fsize{\dimexpr\f@size pt\relax}%
  \newcommand*\lineheight[1]{\fontsize{\fsize}{#1\fsize}\selectfont}%
  \ifx\svgwidth\undefined%
    \setlength{\unitlength}{155.44858437bp}%
    \ifx\svgscale\undefined%
      \relax%
    \else%
      \setlength{\unitlength}{\unitlength * \real{\svgscale}}%
    \fi%
  \else%
    \setlength{\unitlength}{\svgwidth}%
  \fi%
  \global\let\svgwidth\undefined%
  \global\let\svgscale\undefined%
  \makeatother%
  \begin{picture}(1,0.34069901)%
    \lineheight{1}%
    \setlength\tabcolsep{0pt}%
    \put(0,0){\includegraphics[width=\unitlength]{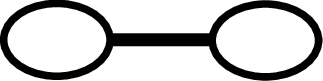}}%
    \put(0.17700548,0.11714502){\makebox(0,0)[t]{\lineheight{1.25}\smash{\begin{tabular}[t]{c}$\left(-2,-4\right)$\end{tabular}}}}%
    \put(0.83346243,0.11382673){\makebox(0,0)[t]{\lineheight{1.25}\smash{\begin{tabular}[t]{c}$\left(0,-4\right)$\end{tabular}}}}%
    \put(0.52084746,0.03975565){\makebox(0,0)[t]{\lineheight{1.25}\smash{\begin{tabular}[t]{c}$\left( -2,-6\right)$\end{tabular}}}}%
  \end{picture}%
\endgroup%
}
\caption{\(A_{[0,1],\frac{-1}{4}}^{\mu}\)}\label{subfig:A 01 n1 4}
\end{subfigure}\hfill
\begin{subfigure}[c]{.3\textwidth}
\scalebox{.7}{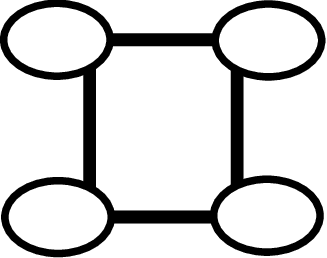}
\caption{\(A_{[1,2],\frac{-5}{12}}^{\mu}\)}\label{subfig:A 12 n5 12}
\end{subfigure}

\caption{The various filtered cube decompositions for \(A_{\tf,i}^{\mu}\) when calculating \(\XKI_{\frac{-1}{6}}(Z)\). Note that when computing \(Z\) the contribution from \(\XKI_{-2}(X)\) is the horizontal direction and the contribution from \(\XKI_{-3}(X)\) is the vertical direction.}\label{fig:Afigs}
\end{figure}

\begin{figure}
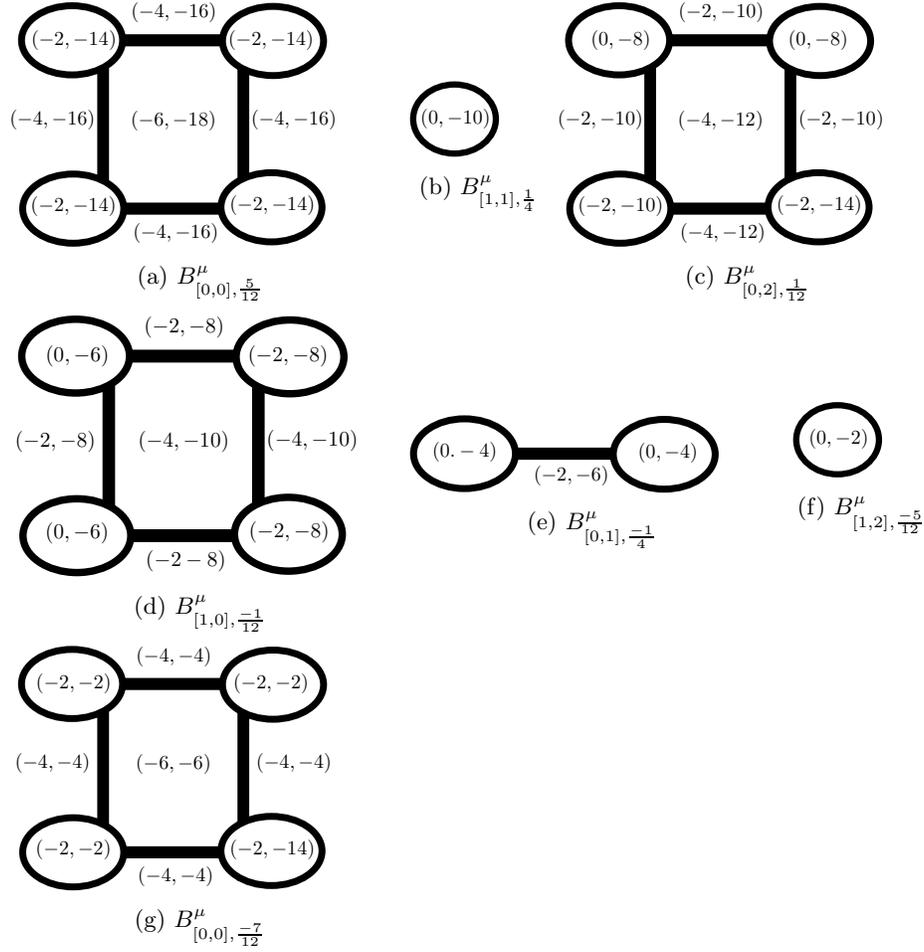

\begin{subfigure}[c]{.4\textwidth}
\scalebox{.75}{}
\caption{\(B_{[0,0],\frac{5}{12}}^{\mu}\)}\label{subfig:B 00 5 12}
\end{subfigure} \hfill
\begin{subfigure}[c]{.15\textwidth}
\scalebox{.75}{
\begingroup%
  \makeatletter%
  \providecommand\color[2][]{%
    \errmessage{(Inkscape) Color is used for the text in Inkscape, but the package 'color.sty' is not loaded}%
    \renewcommand\color[2][]{}%
  }%
  \providecommand\transparent[1]{%
    \errmessage{(Inkscape) Transparency is used (non-zero) for the text in Inkscape, but the package 'transparent.sty' is not loaded}%
    \renewcommand\transparent[1]{}%
  }%
  \providecommand\rotatebox[2]{#2}%
  \newcommand*\fsize{\dimexpr\f@size pt\relax}%
  \newcommand*\lineheight[1]{\fontsize{\fsize}{#1\fsize}\selectfont}%
  \ifx\svgwidth\undefined%
    \setlength{\unitlength}{45.3893226bp}%
    \ifx\svgscale\undefined%
      \relax%
    \else%
      \setlength{\unitlength}{\unitlength * \real{\svgscale}}%
    \fi%
  \else%
    \setlength{\unitlength}{\svgwidth}%
  \fi%
  \global\let\svgwidth\undefined%
  \global\let\svgscale\undefined%
  \makeatother%
  \begin{picture}(1,0.84764082)%
    \lineheight{1}%
    \setlength\tabcolsep{0pt}%
    \put(0,0){\includegraphics[width=\unitlength]{}}%
    \put(0.50168222,0.38259051){\makebox(0,0)[t]{\lineheight{1.25}\smash{\begin{tabular}[t]{c}$\left(0,-10 \right)$\end{tabular}}}}%
  \end{picture}%
\endgroup%
}
\caption{\(B_{[1,1],\frac{1}{4}}^{\mu}\)}\label{subfig:B 11 1 4}
\end{subfigure}\hfill
\begin{subfigure}[c]{.4\textwidth}
\scalebox{.75}{}
\caption{\(B_{[0,2],\frac{1}{12}}^{\mu}\)}\label{subfig:B 02 1 12}
\end{subfigure}

\begin{subfigure}[c]{.4\textwidth}
\scalebox{.8}{}
\caption{\(B_{[1,0],\frac{-1}{12}}^{\mu}\)}\label{subfig:B 10 n1 12}
\end{subfigure} \hfill
\begin{subfigure}[c]{.4\textwidth}
\scalebox{.75}{
\begingroup%
  \makeatletter%
  \providecommand\color[2][]{%
    \errmessage{(Inkscape) Color is used for the text in Inkscape, but the package 'color.sty' is not loaded}%
    \renewcommand\color[2][]{}%
  }%
  \providecommand\transparent[1]{%
    \errmessage{(Inkscape) Transparency is used (non-zero) for the text in Inkscape, but the package 'transparent.sty' is not loaded}%
    \renewcommand\transparent[1]{}%
  }%
  \providecommand\rotatebox[2]{#2}%
  \newcommand*\fsize{\dimexpr\f@size pt\relax}%
  \newcommand*\lineheight[1]{\fontsize{\fsize}{#1\fsize}\selectfont}%
  \ifx\svgwidth\undefined%
    \setlength{\unitlength}{155.44858437bp}%
    \ifx\svgscale\undefined%
      \relax%
    \else%
      \setlength{\unitlength}{\unitlength * \real{\svgscale}}%
    \fi%
  \else%
    \setlength{\unitlength}{\svgwidth}%
  \fi%
  \global\let\svgwidth\undefined%
  \global\let\svgscale\undefined%
  \makeatother%
  \begin{picture}(1,0.34069901)%
    \lineheight{1}%
    \setlength\tabcolsep{0pt}%
    \put(0,0){\includegraphics[width=\unitlength]{}}%
    \put(0.17700548,0.11714502){\makebox(0,0)[t]{\lineheight{1.25}\smash{\begin{tabular}[t]{c}$\left(0.-4\right)$\end{tabular}}}}%
    \put(0.83346243,0.11382673){\makebox(0,0)[t]{\lineheight{1.25}\smash{\begin{tabular}[t]{c}$\left(0,-4\right)$\end{tabular}}}}%
    \put(0.52084746,0.03975565){\makebox(0,0)[t]{\lineheight{1.25}\smash{\begin{tabular}[t]{c}$\left( -2,-6\right)$\end{tabular}}}}%
  \end{picture}%
\endgroup%
}
\caption{\(B_{[0,1],\frac{-1}{4}}^{\mu}\)}\label{subfig:B 01 n1 4}
\end{subfigure}\hfill
\begin{subfigure}[c]{.15\textwidth}
\scalebox{.75}{
\begingroup%
  \makeatletter%
  \providecommand\color[2][]{%
    \errmessage{(Inkscape) Color is used for the text in Inkscape, but the package 'color.sty' is not loaded}%
    \renewcommand\color[2][]{}%
  }%
  \providecommand\transparent[1]{%
    \errmessage{(Inkscape) Transparency is used (non-zero) for the text in Inkscape, but the package 'transparent.sty' is not loaded}%
    \renewcommand\transparent[1]{}%
  }%
  \providecommand\rotatebox[2]{#2}%
  \newcommand*\fsize{\dimexpr\f@size pt\relax}%
  \newcommand*\lineheight[1]{\fontsize{\fsize}{#1\fsize}\selectfont}%
  \ifx\svgwidth\undefined%
    \setlength{\unitlength}{45.3893226bp}%
    \ifx\svgscale\undefined%
      \relax%
    \else%
      \setlength{\unitlength}{\unitlength * \real{\svgscale}}%
    \fi%
  \else%
    \setlength{\unitlength}{\svgwidth}%
  \fi%
  \global\let\svgwidth\undefined%
  \global\let\svgscale\undefined%
  \makeatother%
  \begin{picture}(1,0.84764082)%
    \lineheight{1}%
    \setlength\tabcolsep{0pt}%
    \put(0,0){\includegraphics[width=\unitlength]{}}%
    \put(0.50168222,0.38259051){\makebox(0,0)[t]{\lineheight{1.25}\smash{\begin{tabular}[t]{c}$\left(0,-2 \right)$\end{tabular}}}}%
  \end{picture}%
\endgroup%
}
\caption{\(B_{[1,2],\frac{-5}{12}}^{\mu}\)}\label{subfig:B 12 n5 12}
\end{subfigure}

\begin{subfigure}[c]{.4\textwidth}
\scalebox{.75}{}
\caption{\(B_{[0,0],\frac{-7}{12}}^{\mu}\)}\label{subfig:B 00 n7 12}
\end{subfigure}
\caption{The various filtered cube decompositions for \(B_{\tf,i}^{\mu}\) when calculating \(\XKI_{\frac{-1}{6}}(Z)\). Note that when computing \(Z\) the contribution from \(\XKI_{-2}(X)\) is the horizontal direction and the contribution from \(\XKI_{-3}(X)\) is the vertical direction.
To make identifications with Figure \ref{fig:XKIX} clearer especially in the identification of \(\Gamma\), simplifications have only been done when it results in a doubly filtered point}\label{fig:Bfigs}
\end{figure}

\begin{figure}
\begin{subfigure}{\textwidth}
\input{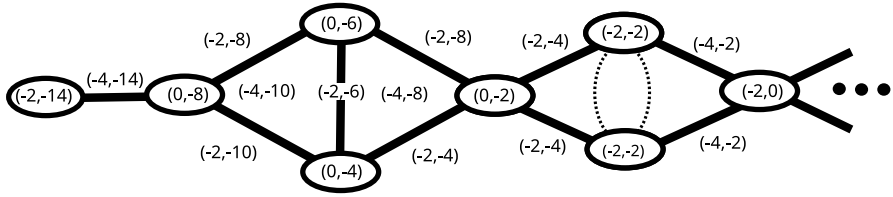}
\caption{The doubly filtered homotopy type for \(\XK_{\frac{-1}{6}}(Z)\).}\label{subfig:XKIZ}
\end{subfigure}
\centering
\begin{subfigure}{.3\textwidth}
\input{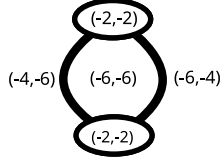}
\caption{The cross section of the right most section of Figure \ref{subfig:XKIZ}.}\label{subfig:crossSec}
\end{subfigure}
\caption{The doubly filtered homotopy type for \(\XK_{\frac{-1}{6}}(Z)\) with cross section provided to understand the double filtration on the 2 cell and the cells in its boundary. This can be represented with a simplicial cube \(\left(\Delta^1\right)^3\).}\label{fig:XKIZ}
\end{figure}

Figure \ref{fig:Afigs} shows the \(A_{[j,k],i}^{\mu}\) for different \(([j,k],i)\) in \(\Ac(Y,K)\) that show up in Figure \ref{fig:XKIDTrunc}, and similarly Figure \ref{fig:Bfigs} shows the  \(B_{[j,k],i}^{\mu}\).
Here the horizontal direction is the contribution from \(\XKI_{-2}(X)\) and the vertical contribution is from \(\XKI_{-3}(X)\).
These have not been further simplified in order to the process of transfering the maps such as the involutive maps and the flip maps from the individual factors to the tensor product.
We have presented the complexes in Figures \ref{fig:Afigs} and \ref{fig:Bfigs} as cube complexes to reduce clutter, but splitting the squares along their northeast diagonal would give the corresponding simplicial complexes.

Then Figure \ref{fig:XKIZ} illustrates the homotopy type of \(\XKI_{\frac{-1}{6}}(Z)\) after the surgery formula has been applied, with the ellipses representing a continuation of the structure provided as needed under symmetry.
Note that for simplicity, we have treated \(\XKI_{\frac{-1}{6}}(Z)\) as a cell complex rather than a simplicial set, though an appropriate simplicial structure can be put on it.
The involutive map \(\Jj\) acts as a reflection around an origin placed between the two rightmost \((-2,-2)\) points, so in particular \(\Jj\) interchanges these two points.
Additionally note that this origin would be placed in the middle of the cross section in Figure \ref{subfig:crossSec}, so the 2-cell of height \((-4,-6)\) is interchanged with the 2-cell of height \((-6,-4)\).
The 3-cell of height \((-6,-6)\) has its orienatation reversed.
Meanwhile, \(\Ii\) acts by reflection around an origin placed in the middle of the 1-cell of height \((-2,-6)\) connecting the vertices with heights \((0,-6)\) and \((0,-4)\).
 As such the orientation of the 1-cell is swapped by \(\Ii\) and the two vertices at its end are swapped.
The two vertices at height \((-2,-2)\) are, along with the vertex at height \((0,-2)\) are sent by \(\Ii\) to the vertex of height \((0,-8)\), while the vertex of \((-2,0)\) sent by \(\Ii\) to the vertex of height \((-2,-14)\).

We will now describe some of the simplifications needed to get from Figures \ref{fig:Afigs} and Figure \ref{fig:Bfigs} to Figure \ref{fig:XKIZ}.
We will start with the simplifications done to the section from \(B_{[0,2],\frac{1}{12}}^{\mu}\) to \(B_{[0,1],\frac{-1}{4}}^{\mu}\), then cover the simplifications to \(B_{[0,0],\frac{5}{12}}^{\mu}\) and then to the portion depicted in the cross section in Figure \ref{subfig:crossSec} (i.e. between \(B_{[1,2],\frac{-5}{12}}^{\mu}\) and \(B_{[1,1],\frac{-3}{4}}^{\mu}\).

\begin{figure}
\[
\begin{tikzcd}
(0,-10) \ar[r,dash,"{(-2,-12)}"] \ar[rd,dash,"{(0,-10)}"'] & (0,-8) \ar[r,dash,"{(0,-8)}"]\ar[d,dash] & (0,-6) \ar[r,dash,"{(-2,-8)}"]\ar[d,dash]\ar[rd,dash]& (0,-4)\ar[d,dash]\ar[rd,dash,"{(0,-4)}"] & \\
& (0,-8) \ar[r,dash,"{(2,-10)}"']\ar[ru,dash] & (0,-6)\ar[r,dash,"{(0,-6)}"'] & (0,-4) \ar[r,dash,"{(-2,-4)}"'] & (0,-2)
\end{tikzcd}
\]
\caption{The shape of the portion of \(\XKI_{\frac{-1}{6}}(Z)\) going from \(B_{[1,1],\frac{1}{4}}^{\mu}\) to \(B_{[1,2],\frac{-5}{12}}^{\mu}\). Each colum with vertices and edges represents a different \(B_{[i,j],k}^{\mu}\), and the columns with edges and 2-cells between represent different \(A_{[i,j],k}^{\mu}\). Only the outer edges have their heights labeled to provide some way to anchor as to which edges are which without cluttering the diagram so much it becomes hard to read.} \label{fig:semiSimplified}
\end{figure}
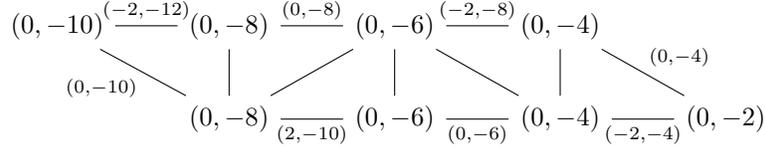

Both \(B_{[0,2],\frac{1}{12}}^{\mu}\) and \(B_{[1,0],\frac{-1}{12}}^{\mu}\) can be simplified down to doubly filtered 1-cubes, with \(B_{[0,2],\frac{1}{12}}^{\mu}\) deformation retracting to the top edge and \(B_{[1,0],\frac{-1}{12}}^{\mu}\) deformation retracting to the left edge.
Meanwhile both \(A_{[0,2],\frac{1}{12}}^{\mu}\) and \(A_{[1,0],\frac{-1}{12}}^{\mu}\) deformation retract to the subcomplex consisting of the leftmost and topmost edges, via respectively a vertical and horizontal homotopy.
A careful analysis of the flip maps induced from \(\XKI_{-2}(X)\) and \(\XKI_{-3}(X)\) reveal  that between \(B_{[0,2],\frac{1}{12}}\) and \(B_{[0,1]\frac{-1}{4}}\), the space \(\XKI_{-1/6}(Z)\) has a form summarized in Figure \ref{fig:semiSimplified}.
For example, the flip map on \(\XKI_{-2,[0]}(X)\) in \Spinc structure \([0]\) collapes all of \(\XKI_{-2,[0]}(X)\) to the leftmost point in \(\XKI_{-2,[1]}(X)\), while the flip map on \(\XKI_{-3,[2]}(X)\) acts simply by translation.
As such \(\eta_{2,[0,2],\frac{1}{12}}\) acts by collapsing \(A_{[0,2],\frac{1}{12}}^{\mu}\) in the horizontal direction to the leftmost edge of \(B_{[1,0],\frac{-1}{12}}^{\mu}\), and thus \(\XKI_{\frac{-1}{6}}(Z)\) has two 1-cells from \(A_{[0,2],\frac{1}{12}}^{\mu}\) that attach to the uppermost vertex in \(B_{[1,0],\frac{-1}{12}}^{\mu}\), one of which is the 1-cell with height \((0,-8)\).
Meanwhile for \(\eta_{1,[1,0],\frac{-1}{12}}\) we take the direct inclusion of \(A_{[1,0],\frac{-1}{12}}^{\mu}\) into \(B_{[1,0],\frac{-1}{12}}^{\mu}\) then apply our homotopy equivalence for \(B_{[1,0],\frac{-1}{12}}^{\mu}\) that collapses it horizontally to its leftmost edge.
This leads the uppermost vertx in \(B_{[1,0],\frac{-1}{12}}^{\mu}\) to also have two 1-cells attaching to it from \(A_{[1,0],\frac{-1}{12}}^{\mu}\) but the 1-cell of height \((0,-6)\) attaches  to the lowermost vertex of \(B_{[1,0],\frac{-1}{12}}^{\mu}\).
Similar tracking of the flip map can be used to confirm the rest of Figure \ref{fig:semiSimplified}.

The only reason Figure \ref{fig:semiSimplified} does not appear to be symmetric with respect to \(\Ii\) is because our choices of homotopy equivalences needed to respect both filtrations, and \(\Ii\) does  not  respect the second filtration.
As such, our choices of homotopy equivalences were not symmetric with respect to \(\Ii\), for example between \(A_{[0,2],\frac{1}{12}}^{\mu}\) and \(A_{[1,0],\frac{-1}{12}}^{\mu}\).
Fortunately this can be corrected up to further homotopy.

Finally, we can use Figure \ref{fig:semiSimplified} to provide the needed simplifications to get to the relevant portion of Figure \ref{fig:XKIZ}.
In particular, by analysis of the height function on the interior of Figure \ref{fig:semiSimplified}, not only can the edges on the outside of height \((0,n)\) collapse to the vertex in their boundary of maximum height, but that this collapse extends to the 2-simplices on the interior that have them as faces.

Additionally, the fact that 
\[\eta_1\colon A_{[0,0],\frac{5}{12}}^{\mu} \to B_{[0,0],\frac{5}{12}}^{\mu}\]
would be an isomorphism if it weren't for the height \((-4,-14)\) vertex in \(A_{[0,0],\frac{5}{12}}^{\mu}\) means all but the upper left vertex in \(B_{[0,0],\frac{5}{12}}^{\mu}\) can be deformation retracted to \(B_{[1,1]\,\frac{1}{4}}^{\mu}\) through \(A_{[0,0],\frac{5}{12}}^{\mu}\).
This is what leaves us with the left most edge and vertex in Figure \ref{fig:XKIZ}.

Finally, we cover the simplifications needed between \(B_{[1,2],\frac{-5}{12}}^{\mu}\) and \(B_{[1,1],\frac{-3}{4}}^{\mu}\).
Because left-most and top-most edges of \(A_{[1,2],\frac{-5}{12}}^{\mu}\) include into \(B_{[0,0],\frac{-7}{12}}^{\mu}\) with their filtration intact, the top right height \((-2,-2)\) vertex of \(B_{[0,0],\frac{-7}{12}}^{\mu}\) can be collapsed to the vertex of \(B_{[1,2],\frac{-5}{12}}^{\mu}\).
Note that applying \(\Jj\) to relate \(A_{[1,2],\frac{-5}{12}}^{\mu}\) and \(A_{[0,0],\frac{-7}{12}}^{\mu}\), we see that the leftmost and rightmost 1-cells there have height \((-6,-4)\), and because they are connected by a 0-cell of the same height, they can be viewed as essentially one long 1-cell of height \((-6,-4)\).
In the homotopy colimit this becomes a 2-cell that gets pulled to cover one half of the height \((-6,-6)\) 3-cell's boundary, during the homotopy equivalence described above.
Applying \(\Jj\) allows us to apply a symmetric argument, yielding the cross section of Figure \ref{subfig:crossSec}.

\subsection{Analysis of the knot}\label{subsec:Analysis}

We will now analyze the knot \((Y,K)\) similar to how we analyzed the regular fiber of \(\Sigma(2,3,7)\) in Section \ref{subsec:exampleARKnots}.
For example, the largest Alexander grading of any simplex in figure \ref{fig:XKIZ} is \(6\) acheived on the leftmost vertex, which provides the generator of \(\widehat{HFK}(Y,K,6)\), and thus 6 is the highest inhabited Alexander grading.
Therefore the genus of this knot is 6, and a similar argument to the case of the regular fiber of \(\Sigma(2,3,7)\) shows that this Seifert surface fibers the complement and provides the genus in \(Y\times I\).

Finally, we can make claims about the genus of \((Y,K)\) in self-homology cobordisms of \(Y\), i.e. the genus of a surface \(\Sigma\) in a homology cobordism \((W,\Sigma)\colon (Y,\emptyset)\to (Y,K)\). 
As discussed in section \ref{subsec:exampleARKnots}, we need the existence of a map 
\[f\colon A_{g,!}(\CFb(\tilde{G},\tf))\to \CFKb(\tilde{G}_{w_0},\tf)\]
so that after inverting both \(U\) and \(V\), \(f\) becomes a chain homotopy equivalence.
An example of this would be the map that sends the entire complex to the vertex of height \((0,-2)\), which uses \(g=1\), and the lack of vertices of height \((0,0)\) means no such map exists for \(g=0\).
As such, the genus of \((Y,K)\) in any self homology cobordism is at least 1.

If the involutive map on lattice homology agrees with the involutive map from Heegaard Floer homology, we get the additional restriction that our map on the level of lattice complexes must be a local equivalence, i.e become an isomorphism after inverting \(U\) and commute up to homotpy with \(\Ii\)
In this case mapping everything to a vertex of height \(0\) would not be a local equivalence since none of the vertices of height 0 in \(\CFb(\tilde{G},\tf)\) are homotopy fixed points.
As such one would use that the \(\iota\)-local equivalence type of \(\CFKb(\tilde{G},\tf)\) is the same as that of the complex \(X'\) for \(\Sigma(2,3,7)\), a single edge of height -2, with vertices of height 0 where \(\iota\) acts by reflection.
A map from \(\CFKb(\tilde{G},\tf)\) to \(X'\), sends the top and left vertices of height 0 to one vertex and the bottom and right vertices of height 0 to the other, whereas the map from \(X'\) to \(\CFb(\tilde{G},\tf)\) sends \(X'\) to the edge between the top and bottom vertices with heights \((0,-4)\) and \((0,-6)\) respectively in Figure \ref{fig:XKIZ}.

To get a map from \(A_{g,!}(X')\) to \(\CFKb(\tilde{G}_{w_0},\tf)\) that commuted with \(\Ii\) on the level of \(p_{1,!}\) would require \(g\) to be at least 3.
In particular, we can consider what the map does specifically to cycles whose first height is 0.
These are all isolated in the filtration layer of first height 0, so considering maps up to filtered homotopy will not change the coefficients of these cycles in a map from \(A_{g,!}(X')\).
Additionally, the condition that the map be an isomorphism after inverting \(U\) then requires that the image of such a map have a cycle with non-zero coefficient from at least one of these 4 vertices.
If a map from \(A_{g,!}(X')\) contains in its image a cycle with non-zero coefficient from the vertex of height \((0,-2)\) then by the action of \(\Ii\), it must also contain in its image cycle with non-zero coefficient from the vertex of height \((0,-8)\) thus forcing \(g\) to be at least 4.
As such, we may assume that the cycles in our image only have coefficients from the vertices of heights \((0,-4)\) and \((0,-6)\).
However, the requirement that some coefficient be nonzero and the fact that \(\Ii\) swaps these two vertices, means that the vertex of height \((0,-6)\) cannot completely be avoided.

\bibliographystyle{abbrvnat}

\bibliography{references}
\end{document}